\documentclass[10pt]{amsart}
\usepackage{amsmath}
\usepackage{amsxtra}
\usepackage{amscd}
\usepackage{amsthm}
\usepackage{amsfonts}
\usepackage{amssymb}
\usepackage{eucal}
\usepackage[all]{xy}
\usepackage{graphicx}

\newtheorem{cor}[subsubsection]{Corollary}
\newtheorem{lem}[subsubsection]{Lemma}
\newtheorem{prop}[subsubsection]{Proposition}
\newtheorem{propconstr}[subsubsection]{Proposition-Construction}
\newtheorem{thmconstr}[subsubsection]{Theorem-Construction}

\newtheorem{conj}[subsubsection]{Conjecture}
\newtheorem{thm}[subsubsection]{Theorem}

\theoremstyle{definition}

\theoremstyle{remark}
\newtheorem{rem}[subsubsection]{Remark}

\newcommand{\thmref}[1]{Theorem~\ref{#1}}

\newcommand{\secref}[1]{Sect.~\ref{#1}}
\newcommand{\lemref}[1]{Lemma~\ref{#1}}
\newcommand{\propref}[1]{Proposition~\ref{#1}}
\newcommand{\corref}[1]{Corollary~\ref{#1}}
\newcommand{\conjref}[1]{Conjecture~\ref{#1}}

\numberwithin{equation}{section}

\newcommand{\nc}{\newcommand}
\nc{\renc}{\renewcommand}
\nc{\ssec}{\subsection}
\nc{\sssec}{\subsubsection}
\nc{\on}{\operatorname}

\DeclareMathOperator\Par{Par}
\DeclareMathOperator\Span{Span}
\DeclareMathOperator\Ad{Ad}
\DeclareMathOperator*\colim{colim}

\DeclareMathOperator\sign{sign}

\nc{\dR}{\on{dR}}

\nc\ol{\overline}
\nc\ul{\underline}
\nc\wt{\widetilde}
\nc\tboxtimes{\wt{\boxtimes}}
\nc\tstar{\wt{\star}}
\nc{\alp}{\alpha}

\nc{\ZZ}{{\mathbb Z}}
\nc{\NN}{{\mathbb N}}
\nc{\OO}{{\mathbb O}}
\renc{\SS}{{\mathbb S}}
\nc{\DD}{{\mathbb D}}
\nc{\GG}{{\mathbb G}}

\nc{\Fq}{{\mathbb F}_q}
\nc{\Fqb}{\ol{{\mathbb F}_q}}
\nc{\Ql}{\ol{{\mathbb Q}_\ell}}
\nc{\id}{\text{id}}
\nc\X{\mathcal X}

\nc{\Hom}{\on{Hom}}
\nc{\Lie}{\on{Lie}}
\nc{\Loc}{\on{Loc}}
\nc{\Pic}{\on{Pic}}
\nc{\Bun}{\on{Bun}}
\nc{\IC}{\on{IC}}
\nc{\Aut}{\on{Aut}}
\nc{\rk}{\on{rk}}
\nc{\Sh}{\on{Sh}}
\nc{\Perv}{\on{Perv}}
\nc{\pos}{{\on{pos}}}
\nc{\Conv}{\on{Conv}}
\nc{\Sph}{\on{Sph}}
\nc{\Sym}{\on{Sym}}
\nc{\BunBb}{\overline{\Bun}_B}
\nc{\BunNb}{\overline{\Bun}_N}
\nc{\BunTb}{\overline{\Bun}_T}
\nc{\BunBbm}{\overline{\Bun}_{B^-}}
\nc{\BunBbel}{\overline{\Bun}_{B,el}}
\nc{\BunBbmel}{\overline{\Bun}_{B^-,el}}
\nc{\Buno}{\overset{o}{\Bun}}
\nc{\BunPb}{{\overline{\Bun}_P}}
\nc{\BunBM}{\Bun_{B(M)}}
\nc{\BunBMb}{\overline{\Bun}_{B(M)}}
\nc{\BunPbw}{{\widetilde{\Bun}_P}}
\nc{\BunBP}{\widetilde{\Bun}_{B,P}}
\nc{\GUb}{\overline{G/U}}
\nc{\GUPb}{\overline{G/U(P)}}

\nc{\Hhom}{\underline{\on{Hom}}}
\nc\syminfty{\on{Sym}^{\infty}}
\nc\lal{\ol{\lambda}}
\nc\xl{\ol{x}}
\nc\thl{\ol{\theta}}
\nc\nul{\ol{\nu}}
\nc\mul{\ol{\mu}}
\nc\Sum\Sigma
\nc{\oX}{\overset{o}{X}{}}
\nc{\hl}{\overset{\leftarrow}h{}}
\nc{\hr}{\overset{\rightarrow}h{}}
\nc{\M}{{\mathcal M}}
\nc{\N}{{\mathcal N}}
\nc{\F}{{\mathcal F}}
\nc{\D}{{\mathcal D}}
\nc{\Q}{{\mathcal Q}}
\nc{\Y}{{\mathcal Y}}
\nc{\G}{{\mathcal G}}
\nc{\E}{{\mathcal E}}
\nc{\CalC}{{\mathcal C}}
\nc\Dh{\widehat{\D}}

\nc{\C}{{\mathcal C}}
\nc{\K}{{\mathcal K}}
\renewcommand{\H}{{\mathcal H}}

\nc{\T}{{\mathcal T}}
\nc{\V}{{\mathcal V}}
\renc{\P}{{\mathcal P}}
\nc{\A}{{\mathcal A}}
\nc{\B}{{\mathcal B}}
\nc{\U}{{\mathcal U}}

\nc{\Gr}{{\on{Gr}}}

\nc{\fn}{{\check{\mathfrak u}(P)}}

\nc{\fC}{\mathfrak C}
\nc{\p}{\mathfrak p}
\nc{\q}{\mathfrak q}
\nc\f{{\mathfrak f}}

\nc{\qo}{{\mathfrak q}}
\nc{\po}{{\mathfrak p}}
\nc{\s}{{\mathfrak s}}
\nc\w{\text{w}}

\renewcommand{\mod}{{\on{-mod}}}

\nc\mathi\iota
\nc\Spec{\on{Spec}}
\nc\Mod{\on{Mod}}
\nc{\tw}{\widetilde{\mathfrak t}}
\nc{\pw}{\widetilde{\mathfrak p}}
\nc{\qw}{\widetilde{\mathfrak q}}
\nc{\jw}{\widetilde j}

\nc{\grb}{\overline{\Gr}}
\nc{\I}{\mathcal I}

\nc{\lambdach}{{\check\lambda}}
\nc{\Lambdach}{{\check\Lambda}{}}
\nc{\much}{{\check\mu}}
\nc{\omegach}{{\check\omega}}
\nc{\nuch}{{\check\nu}}
\nc{\etach}{{\check\eta}}
\nc{\alphach}{{\check\alpha}}
\nc{\oblvtach}{{\check\oblvta}}
\nc{\rhoch}{{\check\rho}}
\nc{\ch}{{\check h}}

\nc{\Hb}{\overline{\H}}

\nc{\BA}{{\mathbb{A}}}
\nc{\BC}{{\mathbb{C}}}
\nc{\BE}{{\mathbb{E}}}
\nc{\BG}{{\mathbb{G}}}
\nc{\BM}{{\mathbb{M}}}
\nc{\BO}{{\mathbb{O}}}
\nc{\BD}{{\mathbb{D}}}
\nc{\BL}{{\mathbb{L}}}
\nc{\BN}{{\mathbb{N}}}
\nc{\BP}{{\mathbb{P}}}
\nc{\BR}{{\mathbb{R}}}
\nc{\BZ}{{\mathbb{Z}}}
\nc{\BS}{{\mathbb{S}}}

\nc{\CA}{{\mathcal{A}}}
\nc{\CB}{{\mathcal{B}}}

\nc{\CE}{{\mathcal{E}}}
\nc{\CF}{{\mathcal{F}}}
\nc{\CH}{{\mathcal{H}}}

\nc{\CL}{{\mathcal{L}}}
\nc{\CC}{{\mathcal{C}}}
\nc{\CM}{{\mathcal{M}}}
\nc{\CN}{{\mathcal{N}}}
\nc{\CK}{{\mathcal{K}}}
\nc{\CO}{{\mathcal{O}}}
\nc{\CP}{{\mathcal{P}}}
\nc{\CQ}{{\mathcal{Q}}}
\nc{\CR}{{\mathcal{R}}}
\nc{\CS}{{\mathcal{S}}}
\nc{\CT}{{\mathcal{T}}}
\nc{\CU}{{\mathcal{U}}}
\nc{\CV}{{\mathcal{V}}}
\nc{\CW}{{\mathcal{W}}}
\nc{\CX}{{\mathcal{X}}}
\nc{\CY}{{\mathcal{Y}}}
\nc{\CZ}{{\mathcal{Z}}}
\nc{\CI}{{\mathcal{I}}}

\nc{\oM}{{\overset{\circ}{\mathcal M}}{}}
\nc{\obM}{{\overset{\circ}{\mathbf M}}{}}
\nc{\obC}{{\overset{\circ}{\mathbf C}}{}}
\nc{\oCA}{{\overset{\circ}{\mathcal A}}{}}
\nc{\obA}{{\overset{\circ}{\mathbf A}}{}}
\nc{\ooM}{{\overset{\circ}{M}}{}}
\nc{\vM}{{\overset{\bullet}{\mathcal M}}{}}
\nc{\nM}{{\underset{\bullet}{\mathcal M}}{}}
\nc{\oD}{{\overset{\circ}{\mathcal D}}{}}
\nc{\obD}{{\overset{\circ}{\mathbf D}}{}}
\nc{\obd}{{\overset{\circ}{\mathbf d}}{}}
\nc{\oA}{{\overset{\circ}{\mathbb A}}{}}
\nc{\op}{{\overset{\bullet}{\mathbf p}}{}}
\nc{\cp}{{\overset{\circ}{\mathbf p}}{}}
\nc{\oU}{{\overset{\bullet}{\mathcal U}}{}}
\nc{\oZ}{{\overset{\circ}{Z}}{}}
\nc{\oCZ}{{\overset{\circ}{\mathcal Z}}{}}
\nc{\ofZ}{{\overset{\circ}{\mathfrak Z}}{}}
\nc{\oF}{{\overset{\circ}{\fF}}}
\nc{\osF}{{\overset{\circ}{\sF}}}

\nc{\fa}{{\mathfrak{a}}}
\nc{\fb}{{\mathfrak{b}}}
\nc{\fd}{{\mathfrak{d}}}
\nc{\ff}{{\mathfrak{f}}}
\nc{\fg}{{\mathfrak{g}}}
\nc{\fgl}{{\mathfrak{gl}}}
\nc{\fh}{{\mathfrak{h}}}
\nc{\fj}{{\mathfrak{j}}}
\nc{\fl}{{\mathfrak{l}}}
\nc{\fm}{{\mathfrak{m}}}
\nc{\frn}{{\mathfrak{n}}}
\nc{\fu}{{\mathfrak{u}}}
\nc{\fp}{{\mathfrak{p}}}
\nc{\fr}{{\mathfrak{r}}}
\nc{\fs}{{\mathfrak{s}}}
\nc{\ft}{{\mathfrak{t}}}
\nc{\fz}{{\mathfrak{z}}}
\nc{\fsl}{{\mathfrak{sl}}}
\nc{\hsl}{{\widehat{\mathfrak{sl}}}}
\nc{\hgl}{{\widehat{\mathfrak{gl}}}}
\nc{\hg}{{\widehat{\mathfrak{g}}}}
\nc{\chg}{{\widehat{\mathfrak{g}}}{}^\vee}
\nc{\hn}{{\widehat{\mathfrak{n}}}}
\nc{\chn}{{\widehat{\mathfrak{n}}}{}^\vee}

\nc{\fA}{{\mathfrak{A}}}
\nc{\fB}{{\mathfrak{B}}}
\nc{\fD}{{\mathfrak{D}}}
\nc{\fE}{{\mathfrak{E}}}
\nc{\fF}{{\mathfrak{F}}}
\nc{\fG}{{\mathfrak{G}}}
\nc{\fK}{{\mathfrak{K}}}
\nc{\fL}{{\mathfrak{L}}}
\nc{\fM}{{\mathfrak{M}}}
\nc{\fN}{{\mathfrak{N}}}
\nc{\fP}{{\mathfrak{P}}}
\nc{\fU}{{\mathfrak{U}}}
\nc{\fV}{{\mathfrak{V}}}
\nc{\fZ}{{\mathfrak{Z}}}

\nc{\bb}{{\mathbf{b}}}
\nc{\bc}{{\mathbf{c}}}
\nc{\bd}{{\mathbf{d}}}
\nc{\bbf}{{\mathbf{f}}}
\nc{\be}{{\mathbf{e}}}
\nc{\bi}{{\mathbf{i}}}
\nc{\bj}{{\mathbf{j}}}
\nc{\bn}{{\mathbf{n}}}
\nc{\bp}{{\mathbf{p}}}
\nc{\bq}{{\mathbf{q}}}
\nc{\bu}{{\mathbf{u}}}
\nc{\bv}{{\mathbf{v}}}
\nc{\bx}{{\mathbf{x}}}
\nc{\bs}{{\mathbf{s}}}
\nc{\by}{{\mathbf{y}}}
\nc{\bw}{{\mathbf{w}}}
\nc{\bA}{{\mathbf{A}}}
\nc{\bK}{{\mathbf{K}}}
\nc{\bB}{{\mathbf{B}}}
\nc{\bC}{{\mathbf{C}}}
\nc{\bG}{{\mathbf{G}}}
\nc{\bD}{{\mathbf{D}}}
\nc{\bH}{{\mathbf{H}}}
\nc{\bI}{{\mathbf{I}}}
\nc{\bM}{{\mathbf{M}}}
\nc{\bN}{{\mathbf{N}}}
\nc{\bO}{{\mathbf{O}}}
\nc{\bV}{{\mathbf{V}}}
\nc{\bW}{{\mathbf{W}}}
\nc{\bX}{{\mathbf{X}}}
\nc{\bZ}{{\mathbf{Z}}}
\nc{\bS}{{\mathbf{S}}}

\nc{\sA}{{\mathsf{A}}}
\nc{\sB}{{\mathsf{B}}}
\nc{\sC}{{\mathsf{C}}}
\nc{\sD}{{\mathsf{D}}}
\nc{\sF}{{\mathsf{F}}}
\nc{\sK}{{\mathsf{K}}}
\nc{\sM}{{\mathsf{M}}}
\nc{\sN}{{\mathsf{N}}}
\nc{\sO}{{\mathsf{O}}}
\nc{\sW}{{\mathsf{W}}}
\nc{\sQ}{{\mathsf{Q}}}
\nc{\sP}{{\mathsf{P}}}
\nc{\sZ}{{\mathsf{Z}}}
\nc{\sfp}{{\mathsf{p}}}
\nc{\sr}{{\mathsf{r}}}
\nc{\bk}{{\mathsf{k}}}
\nc{\sg}{{\mathsf{g}}}
\nc{\sff}{{\mathsf{f}}}
\nc{\sfb}{{\mathsf{b}}}
\nc{\sfc}{{\mathsf{c}}}
\nc{\sd}{{\mathsf{d}}}

\nc{\BK}{{\bar{K}}}

\nc{\tA}{{\widetilde{\mathbf{A}}}}
\nc{\tB}{{\widetilde{\mathcal{B}}}}
\nc{\tg}{{\widetilde{\mathfrak{g}}}}
\nc{\tG}{{\widetilde{G}}}
\nc{\TM}{{\widetilde{\mathbb{M}}}{}}
\nc{\tO}{{\widetilde{\mathsf{O}}}{}}
\nc{\tU}{{\widetilde{\mathfrak{U}}}{}}
\nc{\TZ}{{\tilde{Z}}}
\nc{\tx}{{\tilde{x}}}
\nc{\tbv}{{\tilde{\bv}}}
\nc{\tfP}{{\widetilde{\mathfrak{P}}}{}}
\nc{\tz}{{\tilde{\zeta}}}
\nc{\tmu}{{\tilde{\mu}}}

\nc{\urho}{\underline{\rho}}
\nc{\uB}{\underline{B}}
\nc{\uC}{{\underline{\mathbb{C}}}}
\nc{\ui}{\underline{i}}
\nc{\uj}{\underline{j}}
\nc{\ofP}{{\overline{\mathfrak{P}}}}
\nc{\oB}{{\overline{\mathcal{B}}}}
\nc{\og}{{\overline{\mathfrak{g}}}}
\nc{\oI}{{\overline{I}}}
\nc{\osM}{\overset{\circ}{\mathsf{M}}{}}
\nc{\osN}{\overset{\circ}{\mathsf{N}}{}}

\nc{\eps}{\varepsilon}
\nc{\hrho}{{\hat{\rho}}}

\nc{\one}{{\mathbf{1}}}
\nc{\two}{{\mathbf{t}}}

\nc{\Rep}{{\mathop{\operatorname{\rm Rep}}}}
\nc{\Tot}{{\mathop{\operatorname{\rm Tot}}}}
\nc{\Ker}{{\mathop{\operatorname{\rm Ker}}}}
\nc{\Hilb}{{\mathop{\operatorname{\rm Hilb}}}}
\nc{\End}{{\mathop{\operatorname{\rm End}}}}
\nc{\Ext}{{\mathop{\operatorname{\rm Ext}}}}
\nc{\CHom}{{\mathop{\operatorname{{\mathcal{H}}\it om}}}}
\nc{\GL}{{\mathop{\operatorname{\rm GL}}}}
\nc{\gr}{{\mathop{\operatorname{\rm gr}}}}
\nc{\Id}{{\mathop{\operatorname{\rm Id}}}}
\nc{\de}{{\mathop{\operatorname{\rm def}}}}
\nc{\length}{{\mathop{\operatorname{\rm length}}}}
\nc{\supp}{{\mathop{\operatorname{\rm supp}}}}

\nc{\Cliff}{{\mathsf{Cliff}}}
\nc{\Fl}{\on{Fl}}
\nc{\Fib}{{\mathsf{Fib}}}
\nc{\Coh}{{\on{Coh}}}
\nc{\QCoh}{{\on{QCoh}}}
\nc{\IndCoh}{{\on{IndCoh}}}
\nc{\FCoh}{{\mathsf{FCoh}}}

\nc{\reg}{{\text{\rm reg}}}

\nc{\cplus}{{\mathbf{C}_+}}
\nc{\cminus}{{\mathbf{C}_-}}
\nc{\cthree}{{\mathbf{C}_*}}
\nc{\Qbar}{{\bar{Q}}}
\nc\Eis{\on{Eis}}
\nc\Eisb{\ol\Eis{}}
\nc\Eisr{\on{Eis}^{rat}{}}
\nc\wh{\widehat}
\nc{\Def}{\on{Def_{\check{\fb}}(E)}}
\nc{\barZ}{\overline{Z}{}}
\nc{\barbarZ}{\overline{\barZ}{}}
\nc{\barpi}{\overline\pi}
\nc{\barbarpi}{\overline\barpi}
\nc{\barpip}{\overline\pi{}^+}
\nc{\barpim}{\overline\pi{}^-}

\nc{\fq}{\mathfrak q}

\nc{\fqb}{\ol{\fq}{}}
\nc{\fpb}{\ol{\fp}{}}
\nc{\fpr}{{\fp^{rat}}{}}
\nc{\fqr}{{\fq^{rat}}{}}

\nc{\hattimes}{\wh\otimes}

\nc{\bh}{{{\mathbf h}}}
\nc{\bOmega}{{\overline{\Omega(\check \fn)}}}

\nc{\seq}[1]{\stackrel{#1}{\sim}}

\nc{\cT}{{\check{T}}}
\nc{\cG}{{\check{G}}}
\nc{\cM}{{\check{M}}}
\nc{\cB}{{\check{B}}}
\nc{\cP}{{\check{P}}}

\nc{\ct}{{\check{\mathfrak t}}}
\nc{\cg}{{\check{\fg}}}
\nc{\cb}{{\check{\fb}}}
\nc{\cn}{{\check{\fn}}}

\nc{\cLambda}{{\check\Lambda}}

\nc{\cla}{{\check\lambda}}
\nc{\cmu}{{\check\mu}}
\nc{\cnu}{{\check\nu}}
\nc{\ceta}{{\check\eta}}

\nc{\DefbE}{{\on{Def}_{\cB}(E_\cT)}}

\nc{\imathb}{{\ol{\imath}}}
\nc{\rlr}{\overset{\longrightarrow}{\underset{\longrightarrow}\longleftarrow}}

\nc{\dgSch}{\on{DGSch}}
\nc{\dgindSch}{\on{DGindSch}}
\nc{\indSch}{\on{indSch}}
\nc{\Sch}{\on{Sch}}
\nc{\affdgSch}{\on{DGSch}^{\on{aff}}}
\nc{\affSch}{\on{Sch}^{\on{aff}}}

\nc{\oBun}{\overset{\circ}\Bun}
\nc{\LocSys}{\on{LocSys}}
\nc{\BunBbb}{\ol{\ol{Bun}}_B}
\nc{\BunBr}{\Bun_B^{rat}}
\nc{\BunBrp}{\Bun_B^{rat,polar}}
\nc{\BunTrp}{\Bun_T^{rat,polar}}
\nc{\BunNr}{\Bun_N^{rat}}
\nc{\BunNre}{\Bun_N^{enh,rat}}
\nc{\BunTr}{\Bun_T^{rat}}
\nc{\Vect}{\on{Vect}}
\nc{\Whit}{\on{Whit}}
\nc{\CTb}{\ol{\on{CT}}}
\nc{\Ran}{\on{Ran}}
\nc{\CTr}{\on{CT}^{rat}{}}
\nc\jmathr{\jmath^{rat}{}}
\nc{\ux}{\underline{x}}
\nc{\clambda}{{\check\lambda}}
\nc{\calpha}{{\check\alpha}}
\nc{\ind}{{\mathbf{ind}}}
\nc{\oblv}{{\mathbf{oblv}}}
\nc{\coeff}{\on{W-coeff}}
\nc{\Poinc}{\on{Poinc}}
\nc{\Dmod}{\on{D-mod}}
\nc{\dr}{{\on{dR}}}

\nc{\oCY}{\overset{\circ}\CY}
\nc{\Sing}{\on{Sing}}
\nc{\Maps}{\on{Maps}}

\nc{\Proj}{\on{Proj}}
\nc{\bGamma}{\mathbf{\Gamma}}
\nc{\bLoc}{\mathbf{Loc}}
\nc{\mmod}{\on{-}\mathbf{mod}}
\nc{\oInd}{\overset{\circ}\IndCoh}
\nc{\osupp}{\overset{\circ}\supp}
\nc{\ossupp}{\overset{\circ}{\on{SingSupp}}}
\nc{\coind}{{\mathbf{coind}}}
\nc{\cores}{{\mathbf{cores}}}
\nc{\sotimes}{\overset{!}\otimes}
\nc{\StinftyCat}{\on{DGCat}}
\nc{\bDelta}{\mathbf{\Delta}}
\nc{\Spr}{\on{Spr}}

\nc{\sR}{\mathsf R}
\nc{\sS}{\mathsf S}
\nc{\sX}{\mathsf X}
\nc{\sJ}{\mathsf J}
\nc{\gm}{\mathbb{G}m}
\nc{\BT}{\mathbb{BT}}

\makeatletter
\newcommand*{\@old@slash}{}\let\@old@slash\slash
\def\slash{\relax\ifmmode\delimiter'502F30E\mathopen{}\else\@old@slash\fi}
\makeatother

\begin{document}

\title[The category of singularities as a crystal]{The category of singularities as a crystal \\ and 
global Springer fibers}

\author{D.~Arinkin and D.~Gaitsgory}

\email{arinkin@math.wisc.edu, gaitsgde@math.harvard.edu} 

\date{\today}

\begin{abstract}
We prove the `Gluing Conjecture' on the spectral side of the categorical geometric Langlands conjecture.
The key tool is the structure of crystal on the category of singularities, which allows to reduce the conjecture
to the question of homological triviality of certain homotopy types. These homotopy types are obtained by gluing 
from a global version of Springer fibers.
\end{abstract}

\maketitle

\tableofcontents

\section*{Introduction}

\ssec{What is done in this paper?}

The main result of this paper is a proof of the `Gluing Conjecture' (\cite[Conjecture 9.3.7]{Ga3}), which 
constitutes one of the main steps towards the proof of the categorical geometric Langlands conjecture. To 
prove the conjecture, we develop certain techniques for working with the singular support of (ind)-coherent
sheaves. The techniques are quite general and may be of independent interest.

The paper is divided into three parts; Parts I and II contain general techniques, which are then used to 
prove the Gluing Conjecture in Part III. Here is a brief outline of the paper; we provide a more detailed
description below.

\sssec{} In Part I, we study the notion of singular support for coherent sheaves (or complexes) on a local complete intersection scheme (or, more generally, on a quasi-smooth derived scheme or stack). This notion was introduced in our previous paper, \cite{AG}; the main idea of refining the notion of support of coherent sheaves
using cohomological operators originated in \cite{BIK}. Roughly speaking, to a local complete intersection scheme $Y$, one can attach a scheme $\Sing(Y)$ equipped with a $\BG_m$-action, and to any coherent sheaf 
$\CF$ one can assign its singular support, which is a conical (that is, $\BG_m$-invariant) subset $\on{SingSupp}(\CF)\subset\Sing(Y)$.

Compared to \cite{AG}, Part I introduces two ideas:

Firstly, we work with the category of singularities in place of the category of coherent sheaves. The main
effect of this relatively minor change is that while the singular support of a coherent sheaf
is a conical subset of $\Sing(Y)$, the singular support of an object in the category
of singularities is a subset of the fiberwise projectivization $\BP\Sing(Y)$. Explicitly, $\BP\Sing(Y)$
is obtained from $\Sing(Y)$ by removing the fixed locus of $\BG_m$ (which is identified with $Y$) and
then taking the quotient by $\BG_m$.

The second and main new idea is an `upgrade' of the notion of singular support. Specifically, we show that
the category of singularities of $Y$ carries a natural structure of a \emph{crystal of categories} over $\BP\Sing(Y)$. (Informally, a crystal of categories is a local system of categories over a space.)
This construction is crucial for the rest of the paper: it provides a way to translate some complicated questions about ind-coherent sheaves into topological claims concerning $\BP\Sing(Y)$, which tend to be easier.  

\sssec{} In Part II, we develop a gluing formalism, which is motivated by applications to the geometric Langlands program. Informally, this may be viewed as a kind of descent: given a covering family 
$f_i:Z_i\to Y$ (of quasi-smooth stacks) satisfying certain conditions, we 
show that one can recover an ind-coherent sheaf $\CF$ on $Y$ from some extra structure on ind-coherent sheaves $\CF_i$ on $Z_i$. There are two key properties of this formalism:

Firstly, there are non-trivial restrictions on the singular support of sheaves $\CF$ and $\CF_i$. 
Generally speaking, the singular support of $\CF_i$'s is required to be `smaller' than the singular support
of $\CF$. In this way, the  formalism describes `complicated' (that is, having large singular support) object 
$\CF$  using
`simple' (that is, having small singular support) sheaves $\CF_i$. In fact, in the application to the Gluing Conjecture, the singular support of all $\CF_i$'s is zero, which means that $\CF_i$'s are usual \emph{quasi-coherent} sheaves, which describe the more exotic \emph{ind-coherent} sheaf $\CF$.

Secondly, the main condition on the cover $f_i:Z_i\to Y$ 
has topological nature. Specifically, the condition concerns the topology of certain natural 
correspondences between the spaces $\Sing(Z_i)$ and $\Sing(Y)$. Thus, questions about ind-coherent sheaves on stacks $Z_i$ and $Y$
are reduced to the more transparent claims about the topology of correspondences between $\Sing(Z_i)$ and 
$\Sing(Y)$. This relies on the crystal structure constructed in Part I.

\sssec{} Finally, in Part III, we prove the Gluing Conjecture. Using the gluing formalism developed in 
Part II, we reduce the conjecture to a topological statement concerning (homological) contractibility of 
certain homotopy types. These homotopy types are obtained by gluing generalized Springer fibers, which
parametrize reductions of a local system together with a nilpotent infinitesimal symmetry
to various parabolic subgroups. If the local system is trivial, we obtain the usual Springer fibers
and then the required homological contractibility follows from the Springer correspondence. The general case
relies on the study of the Bruhat-Tits stratification on the generalized Springer fibers, which is 
the main technical result of Part III.

\sssec{Remark} When one works with the derived category of an algebraic variety or a stack $Y$, one has to make
a choice between the `large' derived category (the unbounded quasicoherent derived category) and
the `small' category of perfect complexes. The same choice applies to various modifications of the derived
category: the `large' category of ind-coherent sheaves versus the `small' category of coherent sheaves (or, more
precisely, the bounded coherent derived category), and the `large' category of singularities 
(the quotient of the category of ind-coherent sheaves by the quasicoherent derived category) versus the 
`small' category of singularities (the quotient of the bounded coherent derived category by the category
of perfect complexes). If the stack $Y$ is `reasonable', the `large' categories are compactly generated by the
respective `small' categories; for this reason, it is sometimes possible to work with the more explicit
`small' categories. However, the crystal structure from Part I, as well as most results of Parts II and III,
make sense only for the `large' categories.

\ssec{The goal: the Gluing Conjecture}
We now provide more details on the content of the paper. First, to explain our motivation, let us informally
describe the Gluing Conjecture and its place in the geometric Langlands program. A precise
statement of the Gluing Conjecture can be found in \secref{ss:setting for main}.

\sssec{} 

Let $X$ be a smooth and complete curve, and $G$ a reductive group over an algebraically closed ground field $k$ of 
characteristic $0$. We work on the \emph{spectral side} of geometric Langlands for $G$, which concerns the stack 
$\LocSys_G$ that classifies $G$-local systems on $X$. 

\medskip

As was suggested in \cite{AG}, the category on the spectral side of geometric Langlands is a certain modification of 
the category of quasi-coherent sheaves on $\LocSys_G$. Namely, it is the full subcategory of $\IndCoh(\LocSys_G)$,
consisting of objects whose \emph{singular support} is contained in the global nilpotent cone. We refer the reader to \cite[Sect. 11]{AG},
where the precise meaning of these words is explained. 

\medskip

The resulting category is denoted $\IndCoh_{\on{Nilp}_{\on{glob}}}(\LocSys_G)$; the categorical geometric Langlands conjecture
predicts an equivalence between $\IndCoh_{\on{Nilp}_{\on{glob}}}(\LocSys_G)$ and 
the category $\Dmod(\Bun_\cG)$ of D-modules on the stack $\Bun_\cG$
that classifies principal $\cG$-bundles on $X$ (here $\cG$ is the Langlands dual group of $G$). 

\medskip

The category $\IndCoh_{\on{Nilp}_{\on{glob}}}(\LocSys_G)$ contains the usual category 
$\QCoh(\LocSys_G)$ of quasi-coherent sheaves as a full subcategory.  

\medskip

The Gluing Conjecture aims to express 
$\IndCoh_{\on{Nilp}_{\on{glob}}}(\LocSys_G)$
in terms of the categories $\QCoh(\LocSys_P)$, where $P$ runs through the set of standard parabolic subgroups of $G$
(including $P=G$). Essentially, the goal is to compensate for the modification 
$$\QCoh(\LocSys_G)\rightsquigarrow \IndCoh_{\on{Nilp}_{\on{glob}}}(\LocSys_G)$$
by considering all parabolic subgroups of $G$, and working with usual quasi-coherent sheaves on the
corresponding moduli stacks. 

\sssec{}

More precisely, for a standard parabolic $P$, there is a natural map
\begin{equation} \label{e:map of LocSys}
\LocSys_P\to \LocSys_G
\end{equation} 
induced by the embedding $P\hookrightarrow G$.
We  consider the category of \emph{quasi-coherent sheaves on $\LocSys_P$, equipped with a connection along
the fibers of \eqref{e:map of LocSys}}; denote this category temporarily by $\QCoh(\LocSys_P)_{\on{conn}/\LocSys_G}$. 

\medskip

Below we  make a brief digression to explain what exactly we mean by such a category. As this may appear 
too technical for an introduction, the reader may choose to skip the explanation, take the existence of a well-defined
category $\QCoh(\LocSys_P)_{\on{conn}/\LocSys_G}$ on faith, and proceed to \secref{sss:skip conn}.

\sssec{}

First off, it is \emph{impossible} to makes sense of `quasi-coherent sheaves on a stack equipped with a connection along
a fibration' without resorting to derived algebraic geometry\footnote{Unless some very stringent smoothness conditions 
are satisfied, such as the map being smooth and schematic.}.
So, for the rest of the introduction, when we say `scheme' (reps., `algebraic stack', `prestack'), we 
mean a derived scheme (reps., derived algebraic stack, prestack within derived algebraic geometry). 

\medskip

It is more natural to consider ind-coherent 
sheaves first. Given a map of prestacks $f:\CZ\to \CY$, we let $\IndCoh(\CZ)_{\on{conn}/\CY}$ be the category
of \emph{ind-coherent sheaves on $\CZ$ equipped with a connection along the fibers of $f$}, which we
define to be
$$\IndCoh(\CZ)_{\on{conn}/\CY}:=\IndCoh(\CZ_\dr\underset{\CY_\dr}\times \CY).$$
Here $\IndCoh(\CW)$ is the category of ind-coherent sheaves on a prestack $\CW$ (which is defined for any prestack  \emph{locally almost of finite type}, see 
\cite[Sect. 10]{Ga1}), and $\CW_\dr$ is the de Rham prestack corresponding to a prestack $\CW$ (see \secref{ss:dR}).

\medskip

Pullback along the map $\CZ\to \CZ_\dr\underset{\CY_\dr}\times \CY$ defines a functor
$$\IndCoh(\CZ_\dr\underset{\CY_\dr}\times \CY)\to \IndCoh(\CZ),$$
which can be viewed as the functor of forgetting the connection.

\medskip

Suppose now that $\CZ$ is a quasi-smooth algebraic stack (a.k.a. derived locally complete intersection);
see \cite[Sect. 8.1]{AG} for the definition. For example, $\CZ=\LocSys_P$ is quasi-smooth.
We then define the full subcategory $\QCoh(\CZ)_{\on{conn}/\CY}\subset\IndCoh(\CZ)_{\on{conn}/\CY}$  of 
quasi-coherent sheaves on $\CZ$ equipped with 
a connection along the fibers of $f$ by the condition that it fits into the following pullback diagram of 
categories:
$$
\CD
\QCoh(\CZ)_{\on{conn}/\CY}  @>>>  \IndCoh(\CZ)_{\on{conn}/\CY}  \\
@VVV    @VVV   \\
\QCoh(\CZ)  @>{\Xi_\CZ}>>  \IndCoh(\CZ).
\endCD
$$
Here $\Xi_\CZ$ is the tautological functor of embedding $\QCoh$ into $\IndCoh$  of \cite[Sect. 1.5]{Ga1}  
(extended to algebraic stacks in \cite[Sect. 11.7.3]{Ga1}). Note that the essential image of $\Xi_\CZ$ is 
 the full subcategory of objects with zero singular support, see \cite[Corollary 8.2.8]{AG}.

\sssec{}  \label{sss:skip conn}

Returning to the situation of $\LocSys_G$, the assignment
$$P\rightsquigarrow \QCoh(\LocSys_P)_{\on{conn}/\LocSys_G}$$
can be viewed as a diagram of categories, indexed by the poset $\on{Par}(G)$ of standard parabolics of $G$.

\medskip

Hence, we can talk about the category 
\begin{equation}  \label{e:glued category}
\on{Glue}(\QCoh(\LocSys_P)_{\on{conn}/\LocSys_G},P\in \on{Par}(G)),
\end{equation} 
obtained by \emph{gluing} the categories $\QCoh(\LocSys_P)_{\on{conn}/\LocSys_G}$.
The definition of the operation of gluing is reminded in \secref{ss:gluing}. 

\medskip

The name `gluing' is motivated by the following example: given a stratified topological space 
$Y=\underset{a\in A}\cup\, Y_a$ for a finite poset $A$, there is an equivalence
between the category $\on{Shv}(Y)$ of sheaves on $Y$ 
and the glued category
$\on{Glue}(\on{Shv}(Y_a),a\in A)$, see Example~\ref{ex:gluing sheaves on strat}.

\medskip

Informally, an object of \eqref{e:glued category} is a collection
of objects
$$\CF_P\in \QCoh(\LocSys_P)_{\on{conn}/\LocSys_G}\quad\text{for all }P\in \on{Par}(G),$$ 
plus a homotopy-coherent system of compatibility maps (but not necessarily isomorphisms)
$$\CF_{P_2}|_{\LocSys_{P_1}}\to \CF_{P_1}\quad\text{for all }P_1\subset P_2.$$ 

\sssec{}

For every $P$, pullback defines a functor
$$\IndCoh(\LocSys_G)\to \IndCoh(\LocSys_P)_{\on{conn}/\LocSys_G}.$$

Consider the composition
\begin{multline*} 
\IndCoh_{\on{Nilp}_{\on{glob}}}(\LocSys_G)\hookrightarrow \IndCoh(\LocSys_G)\to \\
\to \IndCoh(\LocSys_P)_{\on{conn}/\LocSys_G}\to
\QCoh(\LocSys_P)_{\on{conn}/\LocSys_G},
\end{multline*} 
where the last arrow is the right adjoint to the inclusion $$\QCoh(\LocSys_P)_{\on{conn}/\LocSys_G}\hookrightarrow 
\IndCoh(\LocSys_P)_{\on{conn}/\LocSys_G}.$$

As $P\in\on{Par}(G)$ varies, we obtain a functor
\begin{equation} \label{e:gluing functor}
\IndCoh_{\on{Nilp}_{\on{glob}}}(\LocSys_G)\to \on{Glue}(\QCoh(\LocSys_P)_{\on{conn}/\LocSys_G},P\in \on{Par}(G)).
\end{equation}

The Gluing Conjecture reads:

\begin{conj} \label{c:gluing prev}
The functor \eqref{e:gluing functor} is \emph{fully faithful}. 
\end{conj}

As was mentioned earlier, the goal of the present paper is to prove this conjecture. 

\ssec{The automorphic side of Langlands duality}

Let us now explain the counterpart of the Gluing Conjecture on the automorphic side of the categorical geometric
Langlands conjecture. 

\medskip

As the contents of this subsection play a motivational role only, the reader may skip it and proceed 
to \secref{ss:methods}.

\sssec{}

On the automorphic side of the categorical Langlands conjecture, we are dealing with the category $\Dmod(\Bun_\cG)$.
As explained in \cite[Sect. 8]{Ga3}, the category is equipped with the functor of \emph{extended Whittaker coefficient}
$$\coeff^{\on{ext}}_{\cG,\cG}:\Dmod(\Bun_\cG)\to \Whit^{\on{ext}}(\cG,\cG),$$
where $\Whit^{\on{ext}}(\cG,\cG)$ is the \emph{extended Whittaker category}. 

\medskip

Recall that the category $\Whit^{\on{ext}}(\cG,\cG)$ is obtained by gluing:
$$\Whit^{\on{ext}}(\cG,\cG)\simeq \on{Glue}(\Whit(\cG,\cP),P\in  \on{Par}(G)),$$
where for a parabolic $P$, we denote by $\Whit(\cG,\cP)$ the $P$-degenerate Whittaker category
(see \cite[Sect. 7]{Ga3}). 

\medskip

For example, for $P=G$, the category $\Whit(\cG,\cG)$ is the \emph{usual} (that is, non-degenerate) 
Whittaker category 
of \cite[Sect. 5]{Ga3}, and for $P=B$, the category $\Whit(\cG,\cB)$ is the \emph{principal series} category $\on{I}(\cG,\cB)$
of \cite[Sect. 6]{Ga3}. 

\sssec{}

In \cite{Ga3} there were formulated several `quasi-theorems'\footnote{By `quasi-theorems' we mean plausible 
statements within reach of current methods.}
that jointly provide a canonically defined fully faithful functor
\begin{multline*}
\on{Glue}(\QCoh(\LocSys_P)_{\on{conn}/\LocSys_G},P\in \on{Par}(G)) \hookrightarrow  \\
\on{Glue}(\Whit(\cG,\cP),P\in  \on{Par}(G))\simeq \Whit^{\on{ext}}(\cG,\cG).
\end{multline*} 

\medskip

Assuming the quasi-theorems hold, we obtain a diagram
\begin{equation} \label{e:fund diagram}
\CD
\on{Glue}(\QCoh(\LocSys_P)_{\on{conn}/\LocSys_G},P\in \on{Par}(G))   @>>>    \Whit^{\on{ext}}(\cG,\cG) \\
@A{\text{\eqref{e:gluing functor}}}AA   @AA{\coeff^{\on{ext}}_{\cG,\cG}}A   \\
\IndCoh_{\on{Nilp}_{\on{glob}}}(\LocSys_G)    & &   \Dmod(\Bun_\cG).
\endCD
\end{equation} 

The categorical Langlands conjecture claims that there exists an equivalence
$$\BL_G:\IndCoh_{\on{Nilp}_{\on{glob}}}(\LocSys_G)  \to  \Dmod(\Bun_\cG)$$
complementing \eqref{e:fund diagram} to a commutative diagram. 

\medskip

In \cite{Ga3}, the following strategy for proving the categorical Langlands conjecture is suggested. First,
one would show that the vertical arrows of \eqref{e:fund diagram} are fully faithful. Then, one would 
identify the essential images of $\IndCoh_{\on{Nilp}_{\on{glob}}}(\LocSys_G)$ and $\Dmod(\Bun_\cG)$ in $\Whit^{\on{ext}}(\cG,\cG)$ by
using some explicit generators of both categories. 

\sssec{}

Thus, one of the key steps in the proof of the categorical Langlands conjecture is to show that the vertical
arrows in \eqref{e:fund diagram} are fully faithful. At this point, we do not know whether the functor 
$\coeff^{\on{ext}}_{\cG,\cG}$ (the right vertical arrow of the diagram) is fully faithful for an arbitrary group
$G$; for $G=\GL_n$, it is a theorem, established in \cite{Be}. 

\medskip

On the other hand, the full faithfulness of the functor \eqref{e:gluing functor} (the left vertical arrow)
is precisely the Gluing Conjecture, which we prove in the present paper. 

\ssec{The methods: crystal structure}  \label{ss:methods}

We derive the Gluing Conjecture from a topological statement. Informally, the key idea is to study both sides
`microlocally'. The word `microlocally' refers here to 
the correspondence between ind-coherent sheaves on a quasi-smooth
scheme (or a stack) $Y$ and conical subsets in the `scheme of singularities' $\Sing(Y)$. We then relate
certain categories obtained by gluing categories of ind-coherent (and quasi-coherent) sheaves 
to homotopy types obtained by gluing conical subsets of schemes of singularities. In particular, this applies
to the category \eqref{e:glued category}: as a result, the Gluing Conjecture reduces to homological triviality of
certain homotopy types. Let us provide some details.

\sssec{}

In \cite[Sect.~2.3]{AG}, we explain how to associate to a quasi-smooth scheme $Y$ a classical scheme of 
singularities $\Sing(Y)$
equipped with a $\BG_m$-action. The scheme $\Sing(Y)$ measures how far $Y$ is from being smooth. 

\medskip

The main construction of the paper \cite{AG} assigns to an object
$\CF\in \IndCoh(Y)$ its singular support, denoted $\on{SingSupp}(\CF)$, which is a conical Zariski-closed subscheme of
$\Sing(Y)$. 

\medskip

It is technically easier for us to work with the category of singularities $\oInd(Y)$ instead of 
$\IndCoh(Y)$, where
$$\oInd(Y):=\IndCoh(Y)/\QCoh(Y).$$
To an object $\CF\in \oInd(Y)$ we can attach its singular support $\BP\on{SingSupp}(\CF)$, which is now a closed 
subscheme of the projectivization $\BP\Sing(Y)$ of $\Sing(Y)$, see also \cite{Ste}. 

\medskip

A key observation, articulated in \secref{s:sing as crystal} of the present paper is that the assignment 
$$\CF\rightsquigarrow \BP\on{SingSupp}(\CF)$$
can be upgraded to a certain categorical structure: $\oInd(Y)$ is in fact a \emph{crystal of categories} over $\BP\Sing(Y)$ (\thmref{t:crystal structure}). Here is a reformulation of this statement:

\begin{thm} \label{t:crystal prev}
There exists a canonical action of the (symmetric) monoidal category $\Dmod(\BP\Sing(Y))$ on $\oInd(Y)$. 
\end{thm}  

In other words, this theorem says that $\oInd(Z)$ can be `localized' onto $\BP\Sing(Z)$. 

\sssec{} The Gluing Conjecture concerns categories of sheaves with connection along fibers of a morphism. Let
us define versions of the categories $\IndCoh(Z)$, $\QCoh(Z)$, and $\oInd(Z)$ for sheaves with connection.

\medskip

Let $f:Z\to Y$ be a map of schemes. Consider the category
$$\IndCoh(Z)_{\on{conn}/Y}:=\IndCoh(Z_\dr\underset{Y_\dr}\times Y),$$
introduced above. 

\medskip 

In \secref{ss:rel as ten} (\propref{p:rel cryst as tensored up}), we  show that this category identifies with
$$\QCoh(Z_\dr\underset{Y_\dr}\times Y)\underset{\QCoh(Y)}\otimes \IndCoh(Y),$$
and therefore contains 
$$\QCoh(Z_\dr\underset{Y_\dr}\times Y)\simeq
\QCoh(Z_\dr\underset{Y_\dr}\times Y)\underset{\QCoh(Y)}\otimes \QCoh(Y)$$
as a full subcategory. 

\medskip

We are interested in the quotient $\IndCoh(Z_\dr\underset{Y_\dr}\times Y)/\QCoh(Z_\dr\underset{Y_\dr}\times Y)$,
which can be thought of as a version of the category of singularities.

\medskip

Assume now that $Y$ is quasi-smooth. In this case, we show in \secref{ss:rel as ten} 
(\propref{p:tensor up via Sing}), that the above quotient
can be expressed in terms of $\oInd(Y)$ by a \emph{topological operation} using the above-mentioned 
crystal structure on $\oInd(Y)$. 
Namely, we have
\begin{equation} \label{e:topological descr}
\IndCoh(Z_\dr\underset{Y_\dr}\times Y)/\QCoh(Z_\dr\underset{Y_\dr}\times Y) \simeq
\Dmod(Z\underset{Y}\times \BP\Sing(Y))\underset{\Dmod(\BP\Sing(Y))}\otimes \oInd(Y).
\end{equation}

The word `topological' refers to the fact that we are dealing with D-modules rather than (quasi)-coherent sheaves. 

\sssec{}

Assume now that in the above situation the scheme $Z$ is quasi-smooth as well. Recall from \cite[Sect. 2.4]{AG} that
in this case we have a canonically defined map
$$\Sing(f):Z\underset{Y}\times \Sing(Y)\to \Sing(Z),$$
called the \emph{singular codifferential} of $f$. 

\medskip

Furthermore, recall the category
$$\QCoh(Z)_{\on{conn}/Y}\subset \IndCoh(Z_\dr\underset{Y_\dr}\times Y).$$ 

We have:
$$\QCoh(Z_\dr\underset{Y_\dr}\times Y)\subset \QCoh(Z)_{\on{conn}/Y}\subset
\IndCoh(Z_\dr\underset{Y_\dr}\times Y),$$
where all the inclusions are, generally speaking, strict. 

\medskip

The key point for us is that the quotient 
$$\QCoh(Z)_{\on{conn}/Y}/\QCoh(Z_\dr\underset{Y_\dr}\times Y)\subset 
\IndCoh(Z_\dr\underset{Y_\dr}\times Y)/\QCoh(Z_\dr\underset{Y_\dr}\times Y)$$
can be described explicitly \emph{in topological terms} using the equivalence \eqref{e:topological descr}.
Namely, in \thmref{t:express category} we  prove: 

\begin{thm}
Under the identification \eqref{e:topological descr}, the full subcategory
$$\QCoh(Z)_{\on{conn}/Y}/\QCoh(Z_\dr\underset{Y_\dr}\times Y)\subset 
\IndCoh(Z_\dr\underset{Y_\dr}\times Y)/\QCoh(Z_\dr\underset{Y_\dr}\times Y)$$
corresponds to 
\begin{multline*}
\Dmod(\BP(\Sing(f)^{-1}(\{0\})))\underset{\Dmod(\BP\Sing(Y))}\otimes \oInd(Y)\subset \\
\Dmod(Z\underset{Y}\times \BP\Sing(Y))\underset{\Dmod(\BP\Sing(Y))}\otimes \oInd(Y).
\end{multline*} 
In particular, we have a canonical equivalence:
\begin{multline}  \label{e:ket top expr}
\QCoh(Z)_{\on{conn}/Y}/\QCoh(Z_\dr\underset{Y_\dr}\times Y)\simeq \\
\simeq \Dmod(\BP(\Sing(f)^{-1}(\{0\})))\underset{\Dmod(\BP\Sing(Y))}\otimes \oInd(Y).
\end{multline}
\end{thm}  

\sssec{} We have now set up an abstract framework for handling \conjref{c:gluing prev}. For simplicity, we work with schemes rather than stacks.

\medskip

Let $Z_i\overset{f_i}\to Y$ be a diagram of quasi-smooth schemes, indexed by some category $I$.
Suppose that the maps $f_i$ are proper.
Let $\CN\subset \Sing(Y)$ be a fixed conical Zariski-closed subset. For each $i\in I$, we consider the composition
\[\IndCoh_\CN(Y)\hookrightarrow \IndCoh(Y)\to \IndCoh(Z_i)_{\on{conn}/Y}\to \QCoh(Z_i)_{\on{conn}/Y}.\]
Taken together, these functors yield a functor
\begin{equation} \label{e:gluing abstract}
\IndCoh_\CN(Y)\to \on{Glue}(\QCoh(Z_i)_{\on{conn}/Y},i\in I).
\end{equation} 
We want to determine whether \eqref{e:gluing abstract} is fully faithful.

\medskip

In \thmref{t:top to IndCoh} we prove the following sufficient condition.

\begin{thm} Suppose the following two conditions hold:

\begin{enumerate}

\item
The corresponding functor
$$\QCoh(Y) \to \underset{i}{\on{lim}}\, \QCoh(Z_{i,\dr}\underset{Y_\dr}\times Y)$$ is
fully faithful.

\item
the corresponding functor
\begin{equation} \label{e:top ff abs}
\Dmod(\BP(\CN))\to \on{Glue}\left(\Dmod(\BP(\Sing(f_i)^{-1}(\{0\}))), i\in I\right)
\end{equation} 
is fully faithful.
\end{enumerate}

Then the functor \eqref{e:gluing abstract} is fully faithful as well.
\end{thm} 

Let us note that in the formation of the category $\on{Glue}\left(\Dmod(\BP(\Sing(f_i)^{-1}(\{0\}))), i\in I\right)$, the functors
$$\Dmod(\BP(\Sing(f_j)^{-1}(\{0\}))\to \Dmod(\BP(\Sing(f_i)^{-1}(\{0\}))$$
for an arrow $i\to j$ in $I$ \emph{are not mere pullbacks}, but rather are given by pull-push along the \emph{correspondence}
$$
\CD
Z_i \underset{Z_j}\times\BP(\Sing(f_j)^{-1}(\{0\}) @>>>  \BP(\Sing(f_i)^{-1}(\{0\})   \\
@VVV   \\
\BP(\Sing(f_j)^{-1}(\{0\}).
\endCD
$$

\sssec{}  \label{sss:W}

Finally, assume that in the above situation, the schemes $Z_i$ are \emph{proper} over $Y$. In this case, in 
\corref{c:pointwise bis} we  show that the question of full faithfulness of the functor \eqref{e:top ff abs} can be 
reduced to that of \emph{homological contractibility} of certain homotopy types. 

\medskip

Namely, for a $k$-point $\nu\in \CN$ let $W_{i,\nu}$ denote the preimage of $\nu$ under the map
$$\Sing(f_i)^{-1}(\{0\})\hookrightarrow Z_i\underset{Y}\times \Sing(Y)\to \Sing(Y).$$

\medskip

For an arrow $i\to j$ in the category of indices $I$, the schemes $W_{i,\nu}$ and $W_{j,\nu}$ are related by the correspondence
$$
\CD
Z_i \underset{Z_j}\times W_{j,\nu}   @>>>    W_{i,\nu}   \\
@VVV   \\
W_{j,\nu}.
\endCD
$$

In \secref{sss:glued prestack} we show how such a data gives rise to a prestack, denoted $W_{\on{Glued},\nu}$. Namely,
$W_{\on{Glued},\nu}$ is the prestack colimit over the category of \emph{strings} 
$$i_0\to i_1\to\dots\to i_n, \quad n\in \BN, \quad i_j\in I$$ 
of the diagram of schemes that assigns to a string as above the scheme
$$Z_{i_0}\underset{Z_{i_n}}\times W_{i_n,\nu}.$$

\medskip

We will prove:

\begin{thm}  \label{t:fiberwise prev}
The functor \eqref{e:top ff abs} is fully faithful if and only if for every $\nu$ not in the zero-section, the prestack $W_{\on{Glued},\nu}$
is \emph{homologically contractible}, i.e., the map
$$\on{C}_*(W_{\on{Glued},\nu})\to k$$
is an isomorphism.
\end{thm}

Here $\on{C}_*$ stands for homology (with coefficients in $k$). Note that if a prestack $\CW$ is the
 colimit of schemes
$$\CW=\underset{a\in A}{\on{colim}}\, W_a,$$
then its homology can be computed as
$$\on{C}_*(\CW)=\underset{a\in A}{\on{colim}}\, \on{C}_*(W_a).$$

If the ground field $k$ is $\BC$, we can assign to $\CW$ the homotopy type
$$\CW^{\on{top}}:=\underset{a\in A}{\on{colim}}\, W_a^{\on{top}}$$
(here the colimit is taken in the $\infty$-category of \emph{spaces}, and for a scheme $W_a$ we denote by $W_a^{\on{top}}$
the underlying analytic space).  In this case we have
$$\on{C}_*(\CW)\simeq \on{C}_*(\CW^{\on{top}}).$$

So, the homology $\on{C}_*(W_{\on{Glued},\nu})$ appearing in \thmref{t:fiberwise prev} is indeed the homology of 
a canonically defined homotopy type. 

\sssec{}  \label{sss:stack version}

The above discussion applies to the case when $Y$ is a quasi-smooth algebraic stack rather than a scheme, 
and $Z_i$'s are quasi-smooth algebraic stacks proper and schematic over $Y$. 

\medskip

The upshot
is that the question of fully faithfulness of the functor
$$\IndCoh_\CN(Y)\to \on{Glue}(\QCoh(Z_i)_{\on{conn}/Y},i\in I)$$
is equivalent to that of homological contractibility, as stated in \thmref{t:fiberwise prev}. 
 
\ssec{The methods: global Springer fibers} 

\sssec{}

Recall that our goal is to show that the functor \eqref{e:gluing functor} is fully faithful. According to
\secref{sss:stack version}, this follows from homological contractibility of certain homotopy types
constructed using the maps
$$
\CD
\LocSys_P\underset{\LocSys_G}\times \Sing(\LocSys_G)  @>>>  \Sing(\LocSys_P)  \\
@VVV  \\
\Sing(\LocSys_G).  
\endCD
$$

\medskip

Namely, fix a $k$-point $\on{Nilp}_{\on{glob}}$. We can think of such a point as a pair $(\sigma,A)$, where
$\sigma$ is a $G$-local system on $X$, and $A$ is a \emph{horizontal} section of the vector bundle $\fg_\sigma$ 
associated with the adjoint representation. 

\medskip

Then, the corresponding scheme $W_{i,\nu}$ of \secref{sss:W} for $i=P$ and $\nu=(\sigma,A)$ is that of
reductions of $\sigma$ to $P$ for which $A$ is contained in the sub-bundle $\fu(P)_\sigma$, where $\fu(P)$ denotes 
the Lie algebra of the unipotent radical of $P$. We denote this scheme by 
$$\Spr_{P,\on{unip}}^{\sigma,A}.$$

In addition, we consider the schemes
\begin{equation} \label{e:springers}
\Spr_{P,\on{unip}}^{\sigma,A}\subset \Spr_{P}^{\sigma,A}\subset \Spr_P^\sigma,
\end{equation} 
where $\Spr_{P}^{\sigma}$ is that of reductions of $\sigma$ to $P$, and $\Spr_{P}^{\sigma,A}$
is the subscheme that corresponds to those reductions for which $A$ is contained in $\fp_\sigma$.

\medskip

All three of the above schemes can be viewed as global versions of the Springer fiber. 

\sssec{}

For $P_1\subset P_2$, the schemes $\Spr_{P_1,\on{unip}}^{\sigma,A}$ and $\Spr_{P_2,\on{unip}}^{\sigma,A}$
are related by the correspondence
$$
\CD
\Spr_{P_1}^\sigma\underset{\Spr_{P_2}^\sigma} \times \Spr_{P_2,\on{unip}}^{\sigma,A}  @>>>  \Spr_{P_1,\on{unip}}^{\sigma,A} \\
@VVV \\
\Spr_{P_2,\on{unip}}^{\sigma,A}
\endCD
$$
and the colimit described in \secref{sss:W} yields a prestack, denoted $\Spr_{\on{Glued,unip}}^{\sigma,A}$.

\medskip

Combining the results of \secref{ss:methods}, we obtain that \conjref{c:gluing prev} follows from the next result
(it appears in the paper as \thmref{t:contr1}):

\begin{thm} \label{t:contr preview}
For any $(\sigma,A)$ with a nilpotent $A$, the prestack $\Spr_{\on{Glued,unip}}^{\sigma,A}$ is homologically contractible.
\end{thm} 

\sssec{}

Although \thmref{t:contr preview} is a concrete statement, it involves the prestack $\Spr_{\on{Glued,unip}}^{\sigma,A}$, which is defined 
by a complicated procedure using correspondences.
However, in \secref{s:red to contr}, we show that \thmref{t:contr preview} is equivalent to a 
statement about simpler objects.

\medskip

Namely, let $\Spr_{\on{Glued}}^{\sigma,A}$ be the colimit of the diagram of schemes
$$P\mapsto \Spr_{P}^{\sigma,A},$$
taken over the poset of standard \emph{proper} parabolics of $G$ (where $\Spr_{P}^{\sigma,A}$ is as in \eqref{e:springers}).  

\medskip

We  show, \emph{assuming that \thmref{t:contr preview} holds for proper Levi subgroups of $G$}, 
that \thmref{t:contr preview} is equivalent to the next assertion (it appears in the paper as \thmref{t:contr2}):

\begin{thm} \label{t:contr2 preview}
For any $(\sigma,A)$ with a \emph{non-zero} nilpotent $A$, the prestack $\Spr_{\on{Glued}}^{\sigma,A}$ is homologically contractible.
\end{thm} 

\sssec{}

\thmref{t:contr2 preview} is an essentially combinatorial statement that is proved in \secref{s:Schubert}. The idea of the proof
is the following:

\medskip

By the Jacobson-Morozov Theorem, the section $A$ defines a reduction of $\sigma$ to a canonically defined parabolic $P_0$. This reduction
gives rise to a stratification of each $\Spr_{P}^{\sigma,A}$ by elements of the Weyl group that measure the relative position
of a given reduction to $P$ with the canonical reduction to $P_0$.

\medskip

For each $w\in W$, let 
$$\Spr_{\on{Glued}}^{\sigma,A,< w}\subset \Spr_{\on{Glued}}^{\sigma,A,\leq w}\subset \Spr_{\on{Glued}}^{\sigma,A}$$
be the corresponding substacks. Consider also
$$\Spr_{\on{Glued}}^{\sigma,A,\leq w}/\Spr_{\on{Glued}}^{\sigma,A,< w}:=\Spr_{\on{Glued}}^{\sigma,A,\leq w}
\underset{\Spr_{\on{Glued}}^{\sigma,A,< w}}\sqcup \on{pt}.$$

We prove, by an analysis of the Weyl group combinatorics, that  the prestack
\[\Spr_{\on{Glued}}^{\sigma,A,\leq w}/\Spr_{\on{Glued}}^{\sigma,A,< w}\] is homologically contractible
for every $w$. 

\medskip

This implies \thmref{t:contr2 preview} by induction on the length of $w$. 

\ssec{Contents}

The present paper is naturally divided into three parts.

\sssec{}

In Part I we discuss the crystal structure on the category of singularities of a quasi-smooth scheme or algebraic
stack, and its corollaries. 

\medskip

In \secref{s:sing as crystal} we state the main result of Part I, \thmref{t:crystal structure}, which says that for a 
quasi-smooth scheme $Z$, there exists a canonically defined crystal of categories over $\BP\Sing(Z)$, denoted
$\oInd(Z)^\sim$, such that the category of singularities of $Z$, denoted
$$\oInd(Z):=\IndCoh(Z)/\QCoh(Z),$$
is recovered as the category of global sections of $\oInd(Z)^\sim$.

\medskip

As was mentioned above, this theorem can be viewed as saying that $\oInd(Z)$ can be `localized' onto $\BP\Sing(Z)$.
Due to the 1-affineness property of de Rham prestacks, this theorem can be equivalently phrased as saying that the
(symmetric) monoidal category
$$\Dmod(\BP\Sing(Z)):=\QCoh(\BP(\Sing(Z))_\dr)$$
acts on $\oInd(Z)$.

\medskip

In \secref{s:proof of cryst} we prove \thmref{t:crystal structure}. Let us emphasize that it is naturally proved 
in the `crystal of categories' formulation, rather than in the `action of the category of D-modules' one. 

\medskip

In \secref{s:rel} we study the category $\IndCoh(Z)_{\on{conn}/Y}$, defined for a morphism $Z\to Y$, and its various subcategories
that can be described in terms of the crystal structure. 

\sssec{}

In Part II of the paper we state our main result, \thmref{t:main}, and reduce it to the assertion that certain homotopy
types are homologically contractible, namely, \thmref{t:contr1}. 

\medskip

In \secref{s:gluing} we recall the general paradigm of gluing of DG categories and state \thmref{t:main}, which
says that the Gluing Conjecture holds. In addition, we state \thmref{t:top to IndCoh}, which says that a certain
fully faithfulness condition purely at the level of D-modules implies a fully faithfulness result for ind-coherent
sheaves.

\medskip

It is fair to say that \thmref{t:top to IndCoh} contains the \emph{main idea} of the present paper: it allows us
to reduce the Gluing Conjecture to the question of homological contractibility. 

\medskip

\secref{s:proof top to IndCoh} is devoted to the proof of \thmref{t:top to IndCoh}.

\medskip

In \secref{s:fibers} we reformulate the condition of \thmref{t:top to IndCoh} (the pullback
functor to the category obtained by gluing certain categories
of D-modules is fully faithful) as homological contractibility of certain prestacks. 

\sssec{}

In Part III of the paper we prove \thmref{t:contr1}, which verifies the required homological contractibility condition
for the Gluing Conjecture. 

\medskip

In \secref{s:red to contr} we introduce global Springer fibers, state \thmref{t:contr1}, and show that it
is equivalent to a simpler homological contractibility statement (\thmref{t:contr2}).

\medskip

In \secref{s:Schubert} we prove \thmref{t:contr2} using an analysis of Weyl group combinatorics and 
Schubert strata. 

\medskip

Finally, in \secref{s:Springer}, we give an alternative proof of a special case \thmref{t:contr2}, using 
the Springer correspondence. 

\ssec{Conventions}

\sssec{DG categories and $\infty$-categories}

This paper uses the language of $\infty$-categories. For example, the main result, \thmref{t:main}, concerns
the lax limit of $\infty$-categories.  Our conventions regarding $\infty$-categories follow those of
\cite{AG}. In particular, the reader \emph{does not} need to know how the theory of $\infty$-categories
is constructed, but rather how to use it. 

\medskip

The primary object of study in this paper is DG categories (e.g., \thmref{t:main} says that a certain 
functor between DG categories is fully faithful). Again, the conventions pertaining to DG categories
follow those of \cite{AG}. Thus, all DG categories are assumed to be presentable, and in particular \emph{cocomplete} 
(i.e., containing arbitrary direct sums); all functors are assumed \emph{continuous} (i.e., preserving colimits). 

\sssec{}

We let $\StinftyCat_{\on{cont}}$ denote the $(\infty,1)$-category of (presentable) DG categories
and continuous functors. This $(\infty,1)$-category has a natural symmetric monoidal structure,
given by tensor product
$$\bC_1,\bC_2\to \bC_1\otimes \bC_2.$$

Thus, we can talk about monoidal DG-categories (i.e., algebra objects in $\StinftyCat_{\on{cont}}$ with respect
to the above (symmetric) monoidal structure), and modules over them.

\medskip

Given a monoidal DG category $\bO$, we denote by  $\bO\mmod$ the category of $\bO$-modules.
Thus, $\bC\in\bO\mmod$ means that $\bC$ is a DG category equipped with an action of $\bO$
$$\bO\otimes \bC\to \bC.$$

\sssec{Derived algebraic geometry}

This paper concerns quasi-coherent and ind-coherent sheaves on derived stacks. 
This puts us in the framework of \emph{derived algebraic geometry}.

\medskip

Our conventions regarding derived algebraic geometry follow those of \cite{AG}. 

\medskip

By a prestack we mean an arbitrary contravariant functor form the $\infty$-category of affine DG schemes
that that of $\infty$-groupoids. (In particular, we say `prestack' rather than `DG prestack'.)
By an `algebraic stack' we mean a derived algebraic stack. 
For a prestack $\CY$ there is a canonically defined category $\QCoh(\CY)$ of quasi-coherent sheaves
on $\CY$. 

\medskip

All DG schemes and prestacks considered in this paper are locally almost of finite type. For such schemes and prestacks,
one has the theory of \emph{ind-coherent} sheaves. The key tenets of this theory are recorded in \cite{Ga1}. However,
the main construction of this theory, namely that of the !-pullback, does not as yet appear in the published literature.
A book-in-progress that contains this, as well as some other fundamental constructions of this theory, is available
in the form of \cite{GR2}. 

\medskip

The following notation is used throughout the paper: for a prestack $\CY$ (assumed as always to be locally almost of finite type)
there is a canonically defined object
$$\omega_\CY\in \IndCoh(\CY),$$
the dualizing sheaf. We have a canonically defined functor 
$$\Upsilon_\CY:\QCoh(\CY)\to \IndCoh(\CY),\quad \CF\mapsto \CF\otimes \omega_\CY.$$

\sssec{Sheaves of categories}

In Part I of the paper, we use the notion of 
\emph{sheaf of categories over a prestack} and some fundamental results about it
(such as the notion of 1-affineness, its implications and its criteria). 
The reader is referred to \cite[Sects. 1 and 2]{Ga2} for a summary. 

\ssec{Acknowledgements}

We are grateful to S.~Raskin for carefully reading the first draft of this paper and providing valuable comments. We would also like to thank the anonymous referee for insightful suggestions.

\medskip

The research of D.A. is partially supported by NSF grant DMS-1101558. The research of D.G. is partially supported by NSF 
grant DMS-1063470.

\newpage 

\centerline{\bf Part I: Crystals and singular support.}

\bigskip

\section{The category of singularities as a crystal}   \label{s:sing as crystal}

Let $Z$ be an affine DG scheme almost of finite type. In this section, we study the singularity category of 
$Z$ 
\[\oInd(Z):=\IndCoh(Z)/\QCoh(Z).\] 
The category $\oInd(Z)$ obviously `lives over' $Z$, in the sense that
its objects can be tensored by quasi-coherent sheaves on $Z$.

\medskip

In this section we  show that if $Z$ is quasi-smooth, then the category $\oInd(Z)$ has a richer structure.
Namely, it 'lives over' the relative de Rham prestack of $\Sing(Z)$, where the latter is the classical scheme
measuring how far $Z$ is from being smooth. 

\ssec{Recollections: singular support}\label{ss:singsupp}

\sssec{}

Let $Z$ be an affine quasi-smooth DG scheme. Consider the DG categories $\IndCoh(Z)$ and $\QCoh(Z)$.
Recall that according to \cite[Sect. 4.2.4]{AG}, there is a canonically defined fully faithful functor
$$\Xi_Z:\QCoh(Z)\hookrightarrow \IndCoh(Z),$$
which admits a (continuous) right adjoint, denoted $\Psi_Z$. 

\medskip

We identify $\QCoh(Z)$ with the full subcategory $\Xi_Z(\QCoh(Z))\subset\IndCoh(Z)$ using the functor 
$\Xi_Z$.

\begin{rem} \label{r:Upsilon and Xi}
Recall there is another canonically defined functor
$$\Upsilon_Z:\QCoh(Z)\to \IndCoh(Z), \quad \CF\mapsto \CF\otimes \omega_Z,$$
where $\omega_Z\in \IndCoh(Z)$ is the dualizing sheaf. 
Fortunately, when $Z$ is quasi-smooth, the functors $\Xi_Z$ and $\Upsilon_Z$ differ by tensoring by
a line bundle. Hence their essential images in $\IndCoh(Z)$ coincide. 
\end{rem}

\sssec{}

Define the \emph{singularity category} of $Z$ to be the quotient DG category
$$\oInd(Z):=\IndCoh(Z)/\QCoh(Z).$$
Note that $\oInd(Z)$ identifies with the full 
subcategory\footnote{Here and elsewhere, for a full subcategory $\bC'\subset \bC$, we denote by $(\bC')^\perp\subset \bC$ its \emph{right}
orthogonal, i.e., the full subcategory consisting of objects that receive no non-zero maps from objects of $\bC'$.}
$\QCoh(Z)^\perp\subset\IndCoh(Z)$ (which 
equals $\on{ker}(\Psi_Z)$).

\medskip

Recall also that $\IndCoh(Z)$ is naturally a module category over $\QCoh(Z)$, and both
functors $\Xi_Z$ and $\Psi_Z$ are compatible with the $\QCoh(Z)$-actions.  Hence,
$\oInd(Z)$ also acquires a natural structure of $\QCoh(Z)$-module category.

\sssec{}

Recall (see \cite[Sect. 2.3]{AG}) that to the DG scheme $Z$ one attaches the classical
scheme $\Sing(Z)$ equipped with

\begin{itemize}

\item a $\BG_m$-action,

\item a projection $\Sing(Z)\to Z$,

\item a zero section $^{\on{cl}}\!Z\to \Sing(Z)$. 

\end{itemize}

By a slight abuse of notation, we denote the image of the zero section by $\{0\}\subset \Sing(Z)$.

\medskip

The action of $\BG_m$ on $\Sing(Z)-\{0\}$ is free. Put 
$$\BP\Sing(Z):=(\Sing(Z)-\{0\})/\BG_m.$$

\sssec{}

The main construction of the paper \cite{AG} (namely, \cite[Defn. 4.1.4]{AG}, which is essentially borrowed from \cite{BIK}) assigns to an object 
$\CF\in \IndCoh(Z)$ a Zariski-closed conical subset
$$\on{SingSupp}(\CF) \subset \Sing(Z).$$ 

\medskip

Conversely, a Zariski-closed conical subset $\CN\subset \Sing(Z)$ yields a full subcategory
$$\IndCoh_{\CN}(Z):=\{\CF\,|\,\on{SingSupp}(\CF) \subset \CN\}\subset \IndCoh(Z).$$

\medskip

The following is \cite[Theorem 4.2.6]{AG}: 

\begin{thm}  \label{t:support zero}
The full subcategories $\IndCoh(Z)_{\{0\}}$ and $\QCoh(Z)$ of 
$\IndCoh(Z)$ coincide.
\end{thm} 

\sssec{}  \label{sss:Psupp}

From \thmref{t:support zero} we obtain that to an object $\CF\in \oInd(Z)$ we can assign a Zariski-closed
subset 
$$\BP\on{SingSupp}(\CF) \subset \BP\Sing(Z).$$ 

Conversely, a Zariski-closed subset $\CN\subset\BP\Sing(Z)$ yields a full subcategory
$$\oInd_{\CN}(Z):=\{\CF\,|\, \BP\on{SingSupp}(\CF) \subset \CN\}\subset \oInd(Z).$$

\ssec{Recollections: sheaves of categories}

\sssec{}

Recall the notion of a \emph{quasi-coherent sheaf of categories
over a prestack} introduced in \cite[Sect. 1.1]{Ga2}. For a prestack $\CY$, a quasi-coherent sheaf of
categories $\CC$ over $\CY$ consists of the following data:
\begin{itemize}
\item A $\QCoh(S)$-module $\CC_{S,y}\in \QCoh(S)\mmod$ for every $(S,y)\in (\affSch)_{/\CY}$;

\item An identification of $\QCoh(S')$-modules
$$\CC_{S',y'}\simeq \QCoh(S')\underset{\QCoh(S)}\otimes \CC_{S,y}$$
for every morphism $S'\overset{f}\to S$, where $(S,y)\in (\affSch)_{/\CY}$ and $y'=y\circ f$;

\item A homotopy-coherent system of compatibilities between the identifications for higher-order compositions.
\end{itemize}

Denote the category of quasi-coherent sheaves of categories over $\CY$ by $\on{ShvCat}(\CY)$.

\sssec{} 

If $\CC\in \on{ShvCat}(\CY)$, the \emph{category of global sections} of $\CC$ is defined as
$$\bGamma(\CY,\CC):=\underset{(S,y)\in \on{PreStk}_{/\CY}}{lim}\, \CC_{S,y}.$$
It is a DG category equipped with a natural action of the (symmetric) monoidal category $\QCoh(\CY)$ (see \cite[Sect. 1.2]{Ga2});
indeed, $\QCoh(\CY)$ acts on each term $\CC_{S,y}$.

\sssec{}  \label{sss:enriched}

The category $\on{ShvCat}(\CY)$ is naturally enriched over $\StinftyCat_{\on{cont}}$. Using this structure, we
can think of of $\bGamma(\CY,\CC)$ as the DG category of maps from $\QCoh_{/\CY}$ to $\CC$, where 
$\QCoh_{/\CY}$ is the \emph{unit} sheaf of categories given by 
$$(\QCoh_{/\CY})_{S,y}:=\QCoh(S)\quad\text{for all }(S,y)\in (\affSch)_{/\CY}.$$

Note that
$\bGamma(\CY,\QCoh_{/\CY})\simeq \QCoh(\CY)$.

\sssec{}  \label{sss:1-aff one}

Recall (see \cite[Definition 1.3.7]{Ga2}) that a prestack $\CY$ is said to be \emph{1-affine} if the above functor
$$\bGamma(\CY,-):\on{ShvCat}(\CY)\to \QCoh(\CY)\mmod$$
is an equivalence of categories. 

\sssec{}

For future reference, recall the following constructions. Let $g:\CZ\to \CY$ be a map 
of prestacks. In this case, we have a tautologically defined functor
$$\cores_g:\on{ShvCat}(\CY)\to \on{ShvCat}(\CZ),$$
given by restriction: for $(S,z)\in (\affSch)_{/\CZ}$ we have
$$(\cores_g(\CC))_{S,z}:=\CC_{S,g\circ z}.$$
Note that $\cores_g(\QCoh_{/\CY})\simeq\QCoh_{/\CZ}$.

\medskip

Slightly abusing the notation,  we  sometimes write 
$$\bGamma(\CZ,\CC):= \bGamma(\CZ,\cores_g(\CC))\quad\text{for }\CC\in \on{ShvCat}(\CY).$$

We  sometimes write for $(S,y)\in (\affdgSch)_{/\CY}$ and $\CC\in \on{ShvCat}(\CY)$
$$\bGamma(S,\CC):=\CC_{S,y}.$$

\sssec{}

The above functor $\cores_g$ admits a right adjoint, which we denote by 
$$\coind_g:\on{ShvCat}(\CZ)\to \on{ShvCat}(\CY).$$
It can be explicitly described as follows: 
$$(\coind_g(\CC))_{S,y}=\bGamma(S\underset{\CY}\times \CZ,\CC)\quad\text{for all }(S,y)\in (\affSch)_{/\CY},$$
see \cite[Sect. 3.1.3]{Ga2}. Here $\CC\in \on{ShvCat}(\CZ)$

\medskip

By adjunction and using \secref{sss:enriched}, we have
$$\bGamma(\CY,\coind_g(\CC))\simeq \bGamma(\CZ,\CC).$$

\ssec{Recollections: the de Rham prestack} \label{ss:dR}

\sssec{}

Recall (see e.g., \cite[Sect. 1.1.1]{GR1}) that the \emph{de Rham prestack} $\CY_\dr$ of a prestack $\CY$ 
is defined by
$$\Maps(S,\CY_\dr)=\Maps({}^{\on{red}}S,\CY), \quad S\in \affdgSch.$$
We have a tautological projection 
$$p_{\CY,\dr}:\CY\to \CY_\dr.$$

\medskip
For this paper, we only consider $\CY_\dr$ for prestacks  $\CY$ of locally (almost) finite type\footnote{The word `almost' is parenthesized because 
$\CY_\dr$ only depends on the classical prestack underlying $\CY$.}. In this case, it is shown in
\cite[Proposition 1.3.3]{GR1} that $\CY_\dr$ is classical and also locally almost of finite type. 

\medskip

The basic fact concerning $\CY_\dr$ is
a canonical equivalence of categories
$$\QCoh(\CY_\dr)\simeq \Dmod(\CY),$$
which, properly speaking, must be taken as the definition of the category $\Dmod(\CY)$. 

\sssec{}  

The main object of study in this paper is sheaves of categories over prestacks of the form $\CY_\dr$.
They can be alternatively called `crystals of categories over $\CY$'. 

\medskip

Let us now list several useful facts about crystals of categories. 
The first is the following (see \cite[Theorem 2.6.3]{Ga2}):

\begin{prop}  \label{p:dr 1-aff}
Let $Y$ be a (DG) scheme of finite type. Then $Y_\dr$ is 1-affine.
\end{prop}

\begin{rem}
We should warn the reader that not all prestacks one encounters in practice are 1-affine.  
For example, although it is shown in \cite[Theorem 2.2.6]{Ga2}
that a quasi-compact algebraic stack $\CY$ is 1-affine under some mild technical assumptions,
the de Rham prestack $\CY_\dr$ is typically \emph{not} 1-affine (see \cite[Proposition 2.6.5]{Ga2}).
\end{rem} 

\sssec{}

For the rest of this subsection we fix a prestack $\CY$ and a closed embedding $i:\CZ\to \CY$.
Note that by the finite type assumption, the complementary open embedding $j:\oCY\hookrightarrow \CY$ 
is a quasi-compact morphism.

\medskip 

We have the following assertion (see \cite[Sect. 4]{Ga2}):

\begin{prop}  \label{p:sect of sheaf on formal} Consider the maps
$$\CZ_\dr\overset{i_\dr}\longrightarrow \CY_\dr\overset{j_\dr}\longleftarrow \oCY_\dr.$$

\noindent{\em(a)}  
The functor 
$$\coind_{i_\dr}:\on{ShvCat}(\CZ_\dr)\to \on{ShvCat}(\CY_\dr)$$ is fully faithful. Its essential 
image consists of those objects that are annihilated by the functor $\cores_{j_\dr}$. 

\smallskip

\noindent{\em(b)} For $\CC\in  \on{ShvCat}(\CY_\dr)$, the functor 
$$\bGamma(\CY_\dr,\CC)\to \bGamma(\CZ_\dr,\CC)$$ induces an equivalence
$$\on{ker}\left(\bGamma(\CY_\dr,\CC)\to \bGamma(\oCY_\dr,\CC)\right)\to \bGamma(\CZ_\dr,\CC).$$

\end{prop}

\sssec{}  \label{sss:support on closed}

From now on, we use claim (b) of \propref{p:sect of sheaf on formal} to identify $\bGamma(\CZ_\dr,\CC)$ and $\on{ker}\left(\bGamma(\CY_\dr,\CC)\to \bGamma(\oCY_\dr,\CC)\right)$. Thus, we consider $\bGamma(\CZ_\dr,\CC)$
as a full subcategory of $\bGamma(\CY_\dr,\CC)$. 

\medskip

We also have the following (tautological) assertion:

\begin{lem}  \label{l:sect of sheaf on formal bis} 
If in the situation of \propref{p:sect of sheaf on formal}(b) the prestack $\CY_\dr$ is 1-affine, then the full subcategory 
$$\bGamma(\CZ_\dr,\CC)\subset \bGamma(\CY_\dr,\CC)$$
consists of objects annihilated by the monoidal ideal
$$\on{ker}\left(\QCoh(\CY_\dr)\to \QCoh(\CZ_\dr)\right)\subset \QCoh(\CY_\dr).$$
\end{lem}

\ssec{Statement of the result} 

Return now to the setup of \secref{ss:singsupp}. Thus, $Z$ is an affine quasi-smooth DG scheme. The notion
of singular support provides natural assignments
$$\CF\in \oInd(Z)\,\,\rightsquigarrow \,\,\,\,\BP\on{SingSupp}(\CF)\subset \BP\Sing(Z)$$ and  
$$\CN\subset \BP\Sing(Z)\,\, \rightsquigarrow \,\,\,\, \oInd_{\CN}(Z)\subset \oInd(Z)$$
(see \secref{ss:singsupp}). The goal of this section is to refine the assignments
to a richer structure.

\sssec{}

Consider the prestack $(\BP\Sing(Z))_\dr$. We will prove:

\begin{thmconstr}  \label{t:crystal structure}
There exists a canonically defined object
$$\oInd(Z)^{\sim}\in \on{ShvCat}((\BP\Sing(Z))_\dr),$$
equipped with an identification 
$$\bGamma((\BP\Sing(Z))_\dr,\oInd(Z)^{\sim})\simeq \oInd(Z).$$
This construction has the following properties:

\smallskip

\noindent{\em(a)} For a Zariski-closed subset $\CN\subset \BP\Sing(Z)$, the full subcategory
$\oInd_\CN(Z)\subset \oInd(Z)$ coincides with 
$$\bGamma(\CN_\dr,\oInd(Z)^{\sim})\subset \bGamma((\BP\Sing(Z))_\dr,\oInd(Z)^{\sim}).$$

\smallskip

\noindent{\em(b)} The action of $\QCoh(Z_\dr)$ on $\bGamma((\BP\Sing(Z))_\dr,\oInd(Z)^{\sim})$
coming from the (symmetric) monoidal functor $$\QCoh(Z_\dr)\to \QCoh((\BP\Sing(Z))_\dr)$$ and the
natural action of the latter on $\bGamma((\BP\Sing(Z))_\dr,\oInd(Z)^{\sim})$ identifies with the 
action of $\QCoh(Z_\dr)$ on $\oInd(Z)$ coming from the (symmetric) monoidal functor $$\QCoh(Z_\dr)\to 
\QCoh(Z),$$ and the action of the latter on $\oInd(Z)\subset \IndCoh(Z)$.

\end{thmconstr}

\begin{rem}
Note that \thmref{t:crystal structure} relates the category of singularities
$$\oInd(Z):=\IndCoh(Z)/\QCoh(Z)$$
and the projectivization $\BP\Sing(Z)$ of $\Sing(Z)$. It would be interesting to find 
a similar structure on $\IndCoh(Z)$ itself. 
\end{rem}

\sssec{}

According to \lemref{l:sect of sheaf on formal bis} and \propref{p:dr 1-aff},
\thmref{t:crystal structure} is equivalent to the following:

\begin{cor}  \label{c:tensored over}
The category $\oInd(Z)$ carries a canonically defined action of the (symmetric) monoidal
category $\QCoh((\BP\Sing(Z))_\dr)$ such that:

\smallskip

\noindent{\em(a)} For a Zariski-closed subset $\CN\subset \BP\Sing(Z)$, the full subcategory
$\oInd_\CN(Z)\subset \oInd(Z)$ coincides with the full subcategory of objects annihilated by
the monoidal ideal
$$\on{ker}(\QCoh((\BP\Sing(Z))_\dr)\to \QCoh(\CN_\dr)).$$ 

\smallskip

\noindent{\em(b)} The action of $\QCoh(Z_\dr)$ on $\oInd(Z)$ coming from
the (symmetric) monoidal functor $$\QCoh(Z_\dr)\to \QCoh((\BP\Sing(Z))_\dr)$$ 
identifies with the 
action of $\QCoh(Z_\dr)$ on $\oInd(Z)$ coming from the (symmetric) monoidal functor $$\QCoh(Z_\dr)\to 
\QCoh(Z)$$ and the action of $\QCoh(Z)$ on $\oInd(Z)\subset \IndCoh(Z)$.

\end{cor}

\ssec{Upgrade to a relative crystal of categories}

We postpone the proof of \thmref{t:crystal structure} until Sect.~\ref{s:proof of cryst}.
Let us state a slight refinement of the theorem concerning the structure of a \emph{relative} 
crystal of categories on the category of singularities.
This refined structure naturally allows us to extend the theory from the case of an affine DG scheme $Z$ to that of
an algebraic stack. 

\sssec{}

Consider the (classical reduced) scheme $\BP\Sing(Z)$, and the prestack
$$(\BP\Sing(Z))_\dr\underset{Z_\dr}\times Z.$$ 
Informally, this prestack can be thought of as the `relative' de Rham stack of $\BP\Sing(Z)$ over the base
$Z$. Let $(\on{id}\times p_{\dr,Z})$ denote the tautological map
$$(\BP\Sing(Z))_\dr\underset{Z_\dr}\times Z\to (\BP\Sing(Z))_\dr.$$

Consider the corresponding functor
$$\coind_{(\on{id}\times p_{\dr,Z})}:\on{ShvCat}((\BP\Sing(Z))_\dr\underset{Z_\dr}\times Z)\to
\on{ShvCat}((\BP\Sing(Z))_\dr).$$

\begin{propconstr}  \label{p:rel crystal structure}
There exists a canonically defined object
$$\oInd(Z)^{\sim,\on{rel}}\in \on{ShvCat}((\BP\Sing(Z))_\dr\underset{Z_\dr}\times Z),$$
equipped with an identification 
$$\coind_{(\on{id}\times p_{\dr,Z})}(\oInd(Z)^{\sim,\on{rel}})\simeq \oInd(Z)^{\sim}.$$
\end{propconstr}

Let us now derive \propref{p:rel crystal structure} from \thmref{t:crystal structure}. 

\sssec{}

First, we claim:

\begin{lem} \label{l:rel 1-aff}
The prestack $(\BP\Sing(Z))_\dr\underset{Z_\dr}\times Z$ is 1-affine. 
\end{lem}

\begin{proof}

We can realize $\BP\Sing(Z)$ as a closed subscheme of $Z\times \BP^n$. Hence,
we have a map
$$(\BP\Sing(Z))_\dr\underset{Z_\dr}\times Z\to (\BP^n)_\dr\times Z,$$
which is a closed embedding. Hence, by \cite[Corollary 3.2.7]{Ga2}, it suffices
to show that $(\BP^n)_\dr\times Z$ is 1-affine.  However, the latter follows from 
\cite[Corollary 3.2.8]{Ga2}.

\end{proof}

\sssec{}

By \lemref{l:rel 1-aff}, we obtain that in order to prove \propref{p:rel crystal structure}, we need
to extend the action of the (symmetric) monoidal category $\QCoh((\BP\Sing(Z))_\dr)$ on
$\oInd(Z)$ to that of the (symmetric) monoidal category
$$\QCoh((\BP\Sing(Z))_\dr\underset{Z_\dr}\times Z).$$

We now claim:

\begin{lem}  \label{l:QCoh on rel}
For any map of DG schemes almost of finite type $Z'\to Z$, the functor 
$$\QCoh(Z'_\dr)\underset{\QCoh(Z_\dr)}\otimes \QCoh(Z)\to \QCoh(Z'_\dr\underset{Z_\dr}\times Z)$$
is an equivalence.
\end{lem} 

\begin{proof}

Follows from \cite[Proposition 3.1.9]{Ga2}. 

\end{proof}

In particular, we obtain that the symmetric monoidal functor
$$\QCoh((\BP\Sing(Z))_\dr)\underset{\QCoh(Z_\dr)}\otimes \QCoh(Z)\to \QCoh((\BP\Sing(Z))_\dr\underset{Z_\dr}\times Z)$$
is an equivalence.

\medskip

Now, the action of $\QCoh((\BP\Sing(Z))_\dr\underset{Z_\dr}\times Z)$ on $\oInd(Z)$ is obtained
by combining \lemref{l:QCoh on rel} and the compatibility statement \thmref{t:crystal structure}(b).

\qed(\propref{p:rel crystal structure})

\ssec{Extension to algebraic stacks}  \label{ss:stacks}

\sssec{}

Let now $\CZ$ be a quasi-smooth algebraic stack with an affine diagonal (see \cite[Sect. 8.1.1]{AG} for the definition). 

\medskip

Let $\Sing(\CZ)$ be the corresponding (classical) algebraic stack, 
constructed in \cite[Sect. 8.1.5]{AG}, and consider the corresponding stack $\BP\Sing(\CZ)$. 

\sssec{}

Consider the category $\IndCoh(\CZ)$, the subcategory $\QCoh(\CZ)\overset{\Xi_\CZ}\hookrightarrow \IndCoh(\CZ)$,
and the quotient category
$$\oInd(\CZ):=\IndCoh(\CZ)/\QCoh(\CZ),$$
which identifies with the full subcategory 
$$\QCoh(\CZ)^\perp=\on{ker}(\Psi_\CZ:\IndCoh(\CZ)\to \QCoh(\CZ))\subset \IndCoh(\CZ).$$

\medskip 

The constructions of \secref{sss:Psupp} and \cite[Sect. 8.2]{AG} generalize to define for every
$\CF\in \oInd(\CZ)$ the Zariski-closed subset
$$\BP\on{SingSupp}(\CF)\subset \BP\Sing(\CZ),$$
and for a Zariski-closed subset $\CN\subset \BP\Sing(\CZ)$, the full subcategory
$$\oInd_\CN(\CZ)\subset \oInd(\CZ).$$

\sssec{}

We claim:

\begin{prop}  \label{p:sheaf over stacks}
There exists a canonically defined object
$$\oInd(\CZ)^{\sim,\on{rel}}\in \on{ShvCat}((\BP\Sing(\CZ))_\dr\underset{\CZ_\dr}\times \CZ),$$
equipped with the following system of identifications:

\smallskip

\noindent{\em(a)} For an affine DG scheme $Z$ equipped with a \emph{smooth} map $Z\to \CZ$, we have a
canonical identification
$$\bGamma\left(\left((\BP\Sing(\CZ))_\dr\underset{\CZ_\dr}\times \CZ\right)\underset{\CZ}\times Z,
\oInd(\CZ)^{\sim}\right)
\simeq \oInd(Z),$$
as categories equipped with an action of $\QCoh(Z)$.

\smallskip

\noindent{\em(b)} For a \emph{smooth} map $g:Z_1\to Z_2$ of affine DG schemes smooth over $\CZ$, the diagram
$$
\CD
\bGamma\left(\left((\BP\Sing(\CZ))_\dr\underset{\CZ_\dr}\times \CZ\right)\underset{\CZ}\times Z_2,
\oInd(\CZ)^{\sim}\right) @>>>  \oInd(Z_2)  \\
@VVV    @VV{g^!}V   \\
\bGamma\left(\left((\BP\Sing(\CZ))_\dr\underset{\CZ_\dr}\times \CZ\right)\underset{\CZ}\times Z_1,
\oInd(\CZ)^{\sim}\right) @>>> \oInd(Z_1)
\endCD
$$
commutes.
\end{prop} 

\begin{proof}

The proof is completely formal:

\medskip

Let $(\affdgSch)_{\on{smooth}/\CZ}$ be the category of affine DG schemes $Z$ equipped with a smooth map to $\CZ$. 
By \cite[Theorem 1.5.7]{Ga2},
in order to construct $\oInd(\CZ)^{\sim,\on{rel}}$, it is sufficient to construct an assignment 
$$Z\in (\affdgSch)_{\on{smooth}/\CZ}\rightsquigarrow  
\oInd(\CZ)^{\sim,\on{rel}}|_Z\in 
\on{ShvCat}\left(\left((\BP\Sing(\CZ))_\dr\underset{\CZ_\dr}\times \CZ\right)\underset{\CZ}\times Z\right)$$
equipped with a coherent system of identifications
$$g:Z_1\to Z_2\,\,\rightsquigarrow \,\,\,\,\cores_{\on{id}\times g}(\oInd(\CZ)^{\sim,\on{rel}}|_{Z_2})\simeq 
\oInd(\CZ)^{\sim,\on{rel}}|_{Z_1}.$$

\medskip

Given $Z\in (\affdgSch)_{\on{smooth}/\CZ}$, we set 
$$\oInd(\CZ)^{\sim,\on{rel}}|_Z:=\oInd(Z)^{\sim,\on{rel}}.$$
Note that 
$$
\left((\BP\Sing(\CZ))_\dr\underset{\CZ_\dr}\times \CZ\right)\underset{\CZ}\times Z\simeq 
(\BP\Sing(Z))_\dr \underset{Z_\dr}\times Z.$$ 

\medskip

It remains to construct an identification
\begin{equation} \label{e:cores ident}
\cores_{\on{id}\times g}(\oInd(Z_2)^{\sim,\on{rel}})\simeq \oInd(Z_1)^{\sim,\on{rel}}
\end{equation}
for a morphism $g:Z_1\to Z_2$ in $(\affdgSch)_{\on{smooth}/\CZ}$.
Since $(\BP\Sing(Z))_\dr \underset{Z_\dr}\times Z$ is 1-affine, an identification \eqref{e:cores ident}
amounts to an identification
\begin{multline*}
\bGamma\left((\BP\Sing(Z_2))_\dr \underset{(Z_2)_\dr}\times Z_2,\oInd(Z_2)^{\sim,\on{rel}}\right)
\underset{\QCoh(Z_2)}\otimes \QCoh(Z_1)\simeq  \\
\simeq \bGamma\left((\BP\Sing(Z_1))_\dr \underset{(Z_1)_\dr}\times Z_1,\oInd(Z_1)^{\sim,\on{rel}}\right)
\end{multline*}
in $\QCoh\left((\BP\Sing(Z_1))_\dr \underset{(Z_1)_\dr}\times Z_1\right)\mmod$. 

Since
$$\bGamma\left((\BP\Sing(Z_i))_\dr \underset{(Z_i)_\dr}\times Z_i,\oInd(Z_i)^{\sim,\on{rel}}\right)\simeq \oInd(Z_i),$$
it remains to construct an identification
$$\oInd(Z_2)\underset{\QCoh(Z_2)}\otimes \QCoh(Z_1)\simeq \oInd(Z_1).$$
Such identification is given by the functor $g^!$,  see \cite[Corollary 7.5.7]{Ga1}.
\end{proof}

\sssec{}

We now claim:

\begin{prop} There exists a canonical identification
$$\bGamma\left((\BP\Sing(\CZ))_\dr\underset{\CZ_\dr}\times \CZ,\oInd(\CZ)^{\sim,\on{rel}}\right)\simeq \oInd(\CZ).$$
Moreover, for a Zariski-closed subset $\CN\subset \BP\Sing(\CZ)$, the full subcategory
$$\oInd_\CN(\CZ)\subset \oInd(\CZ)$$ equals 
$$\bGamma\left(\CN_\dr\underset{\CZ_\dr}\times \CZ,\oInd(Z)^{\sim}\right)\subset 
\bGamma\left((\BP\Sing(Z))_\dr\underset{\CZ_\dr}\times \CZ,\oInd(Z)^{\sim}\right).$$
\end{prop}

\begin{proof}

Follows by combining \thmref{t:crystal structure}(a) and \cite[Proposition 8.3.4]{AG}.

\end{proof}  

\section{Proof of \thmref{t:crystal structure}}  \label{s:proof of cryst}

\ssec{Idea of the proof}

Before we give the proof, let us explain informally its main idea.

\sssec{}

To specify a sheaf of categories $\CC$ over $(\BP\Sing(Z))_\dr$, we need to assign a category $\bGamma(S,\CC)$
to any affine DG scheme $S$ equipped
with a map $$^{\on{red}}S\to \BP\Sing(Z).$$

\medskip

In the case of the sheaf $\CC=\oInd(Z)^{\sim}$, we take $\bGamma(S,\oInd(Z)^{\sim})$ to be a certain
full subcategory in 
$$\QCoh(S)\otimes \oInd(Z).$$

\sssec{}

Namely, for an object $\CF\in \QCoh(S)\otimes \IndCoh(Z)$ we can talk about its singular support, which is a closed subset
in $S\times \Sing(Z)$,
conical with respect to the $\BG_m$-action on the second factor. Note that if $\CF\in \QCoh(S)\otimes \QCoh(Z)$,
then its singular support is contained in $S\times \{0\}$. Hence, to an object of 
$$\QCoh(S)\otimes \oInd(Z)$$
we can attach its singular support, which is a closed subset of $S\times \BP\Sing(\CZ)$.  

\medskip

Now, let  
$$\bGamma(S,\oInd(Z)^{\sim})\subset \QCoh(S)\otimes \oInd(Z)$$
be the full subcategory of objects whose singular support is contained (set-theoretically) in the graph of the given map
$^{\on{red}}S\to \BP\Sing(\CZ)$.

\sssec{}

To prove that the above construction works, we need to do two things: 

\smallskip

\noindent(i) Show that the assignment $S\rightsquigarrow \bGamma(S,\oInd(Z)^{\sim})$ is 
indeed a sheaf of categories. This will not be difficult. 

\smallskip

\noindent(ii) Show that a naturally constructed functor $\oInd(Z)\to \bGamma(\BP\Sing(\CZ),\oInd(Z)^{\sim})$ 
is an equivalence. To do so, we will reduce to the case when $Z$ is a global complete intersection and use
some explicit analysis. 

\sssec{}

Rather than giving the proof specifically for $\oInd(Z)$, below we do it in an abstract setting, by isolating
the relevant pieces of structure. 

\medskip

Namely, instead of $\IndCoh(Z)$ we have an arbitrary DG category $\bC$,
and the role of the $\BE_2$-algebra of Hochschild cochains (whose action on $\IndCoh(Z)$ gives rise to the
notion of singular support), we use an arbitrary $\BE_2$-algebra $\CA$.

\ssec{Abstract setting for \thmref{t:crystal structure}}  \label{ss:abs setting}

\sssec{}  \label{sss:abs setting}

Let $\bC$ be a DG category, equipped with an action of an $\BE_2$-algebra $\CA$ 
(see \cite[Sect. 3.5]{AG} for what this means). Let $A$ be a commutative finitely generated algebra,
graded by even non-negative integers, equipped with a grading-preserving homomorphism
$$A\to H^\bullet(\CA):=\underset{n}\oplus\, H^n(\CA).$$

\medskip

According to \cite[Sect. 3.5]{AG} (by the construction going back to \cite{BIK}), to any $\bc\in \bC$ we can attach its support, denoted 
$\on{supp}_A(\bc)$, which is a conical Zariski-closed subset of $\Spec(A)$.  

\medskip

Vice versa, to a conical Zariski-closed subset $\CN\subset \Spec(A)$ we assign the full subcategory
$$\bC_\CN\subset \bC,$$
consisting of objects with support in $\CN$. 

\sssec{}

Let $A^0$ be the degree $0$ component of $A$. The projection $\Spec(A)\to \Spec(A^0)$
admits a canonically defined section $\Spec(A^0)\to \Spec(A)$, because we can identify
$A^0$ with the quotient algebra of $A$ by the ideal $A^{>0}$. 

\medskip

Let $\{0\}$ denote the subset of $\Spec(A)$ equal to the image $\Spec(A^0)$ under the
above section. Let $\bC_{\{0\}}$ be the corresponding full subcategory of $\bC$. Define
$$\obC:=\bC/\bC_{\{0\}}.$$

We can also think of $\obC$ as the kernel of the co-localization functor
$\bC\to \bC_{\{0\}}$, right adjoint to the tautological embedding; this is the same as $(\bC_{\{0\}})^\perp\subset \bC$.

\sssec{}

Consider the scheme $\Proj(A)$. The assignment
$$\bc\in \bC\,\,\rightsquigarrow \,\,\,\,\supp_A(\bc)\subset \Spec(A)$$
gives rise to an assignment
$$\bc\in \obC\,\,\rightsquigarrow \,\,\,\,\BP\supp_A(\bc)\subset \Proj(A).$$

\medskip

Vice versa, to a Zariski-closed subset $\CN\subset \Proj(A)$ we assign the full subcategory
$$\obC_\CN=\{\bc\in\obC\,|\, \BP\supp_A(\bc)\subset \CN\}\subset \obC.$$

\ssec{Plan of this section} 

\sssec{}

In \secref{ss:constr sheaf of cat},  we  attach a certain sheaf of categories 
$\CC_A\in\on{ShvCat}(\Proj(A)_\dr)$ to the data $(\bC,\CA,A)$ as above.  

\medskip

In \secref{ss:glob sect}, we  show that $\CC_A$ comes equipped with a functor
\begin{equation} \label{e:global sect, prev}
\obC\to \bGamma(\Proj(A)_\dr,\CC_A).
\end{equation} 
More generally, for a Zariski-closed subset $\CN\subset \Proj(A)$, there is a functor
\begin{equation} \label{e:global sect supp, prev}
\obC_\CN\to \bGamma(\CN_\dr,\CC_A).
\end{equation} 

\medskip

We then provide additional conditions on the triple $(\bC,\CA,A)$ (in \secref{sss:assump}) that
guarantee that the functor \eqref{e:global sect supp, prev}, 
and in particular \eqref{e:global sect, prev}, is an equivalence. The proof of this claim
(\propref{p:compute global sections}) occupies Sects. ~\ref{ss:globe sect 1}--\ref{ss:globe sect 3}.

\sssec{}

In \secref{ss:our case} we  apply this discussion to $$\bC:=\IndCoh(Z),\,\, \CA:=\on{HC}(Z),\,\,
A:=\Gamma(\Sing(Z),\CO_{\Sing(Z)}).$$ 

\medskip

In the above formula, $\on{HC}(Z)$ is the $\BE_2$-algebra of Hochschild cochains on $Z$, or, which is the same,
the $\BE_2$-center of the DG category $\IndCoh(Z)$, see \cite[Appendix F]{AG}.

\medskip

The resulting sheaf of categories category $\CC_A$ is the sought-for 
$$\oInd(Z)^\sim\in \on{ShvCat}((\BP\Sing(Z))_\dr).$$

\medskip

The equivalence \eqref{e:global sect supp, prev} proves point (a) of \thmref{t:crystal structure}.

\sssec{}

To establish point (b) of \thmref{t:crystal structure}, we study the interaction of the construction
$$(\bC,\CA,A)\rightsquigarrow \CC_A$$
with some pre-existing monoidal actions; this is done in \secref{ss:action of base}.

\ssec{Construction of the sheaf of categories}  \label{ss:constr sheaf of cat}

\sssec{}

For $S\in \affdgSch$ consider the category
$$\QCoh(S)\otimes \bC.$$

\medskip

The action of $\CA$ on $\bC$ and the action of the $\BE_\infty$-algebra $\Gamma(S,\CO_S)$, viewed
as a $\BE_2$-algebra, on $\QCoh(S)$ give
rise to the action of the $\BE_2$-algebra $\Gamma(S,\CO_S)\otimes \CA$ on $\QCoh(S)\otimes \bC$. 

\medskip

Note that we have a canonical map of commutative algebras
$$A_S:=H^0(\Gamma(S,\CO_S))\otimes A\to H^\bullet(\Gamma(S,\CO_S)\otimes \CA).$$

\sssec{}

Consider the corresponding categories $(\QCoh(S)\otimes \bC)_{\{0\}}\subset \QCoh(S)\otimes \bC$ and 
$$\QCoh(S)\otimes \bC/(\QCoh(S)\otimes \bC)_{\{0\}}.$$

\medskip

By \cite[Proposition 3.5.7]{AG},
we have $$(\QCoh(S)\otimes \bC)_{\{0\}}=\QCoh(S)\otimes \bC_{\{0\}},$$ 
as full subcategories in $\QCoh(S)\otimes \bC$. 

\medskip

Hence,
$$\QCoh(S)\otimes \bC/(\QCoh(S)\otimes \bC)_{\{0\}}\simeq \QCoh(S)\otimes \obC.$$

\medskip

Note that $\Proj(A_S)\simeq S\times \Proj(A)$. Thus, to a Zariski-closed subset $\CN'\subset S\times \Proj(A)$,
we can attach the full subcategory
$$(\QCoh(S)\otimes \obC)_{\CN'}\subset \QCoh(S)\otimes \obC.$$

\sssec{}

Assume now that $S$ is equipped with a map to $\Proj(A)_\dr$, i.e., 
$^{\on{red}}S$ is equipped with a map $f$ to $\Proj(A)$. 

\medskip

Define
$$\bGamma(S,\CC_A):=(\QCoh(S)\otimes \obC)_{\on{Graph}_f},$$
where $\on{Graph}_f$ is the Zariski-closed \emph{subset} of $S\times \Proj(A)$ equal to
the graph of the map $f$.  

\sssec{}

For a map $S_1\to S_2$ we have a tautological identification
$$\QCoh(S_1)\underset{\QCoh(S_2)}\otimes (\QCoh(S_2)\otimes \obC)) 
\simeq \QCoh(S_1)\otimes \obC.$$

It is easy to see that under this identification we have an inclusion
\begin{equation} \label{e:incl for graphs}
\QCoh(S_1)\underset{\QCoh(S_2)}\otimes 
(\QCoh(S_2)\otimes \obC)_{\on{Graph}_{f_2}}\subset 
(\QCoh(S_1)\otimes \obC)_{\on{Graph}_{f_1}},
\end{equation} 
where $f_2:{}^{\on{red}}S_2\to \Proj(A)$ and $f_1$ is the composition of 
$^{\on{red}}S_1\to {}^{\on{red}}S_2$ and $f_2$. 

\medskip

We claim:

\begin{lem}  \label{l:eq for graphs}
The inclusion \eqref{e:incl for graphs} is an equality.
\end{lem}

\begin{proof} 
Follows by combining \cite[Proposition 3.5.5 and Lemma 3.3.12]{AG}.
\end{proof}

\sssec{}

From \lemref{l:eq for graphs} we obtain that the assignment
$$(S,f) 
\rightsquigarrow (\QCoh(S)\otimes \obC)_{\on{Graph}_f}$$
defines an object of $\on{ShvCat}(\Proj(A)_\dr)$. 

\medskip

We  denote this object by $\CC_A$. Thus by definition, 
$$\bGamma(S,\CC_A):=(\QCoh(S)\otimes \obC)_{\on{Graph}_f}$$
for any
$(S,f) \in (\affdgSch)_{/\Proj(A)_\dr}$.

\ssec{A functor to the category of global sections}   \label{ss:glob sect}

\sssec{}

For $(S,f)\in (\affdgSch)_{/\Proj(A)_\dr}$, we define a functor
\begin{equation} \label{e:from glob to loc sections}
\obC\to (\QCoh(S)\otimes \obC)_{\on{Graph}_f}=\bGamma(S,\CC_A)
\end{equation}
as follows.

\medskip

It is the composition of the tautological functor
$$\obC\to \QCoh(S)\otimes \obC,\quad \bc\mapsto \CO_S\otimes \bc,$$
followed by the co-localization functor
$$\QCoh(S)\otimes \obC\to (\QCoh(S)\otimes \obC)_{\on{Graph}_f},$$
which is right adjoint to the tautological embedding
$$(\QCoh(S)\otimes \obC)_{\on{Graph}_f}\hookrightarrow \QCoh(S)\otimes \obC.$$

\sssec{}

The functors \eqref{e:from glob to loc sections} are clearly compatible under the maps
$S_1\to S_2$ in the category $(\affdgSch)_{/\Proj(A)_\dr}$.

\medskip

Hence, they give rise to a functor
\begin{equation} \label{e:from glob to glob sections}
\obC\to \bGamma(\Proj(A)_\dr,\CC_A).
\end{equation} 

\sssec{}

Let now $\CN\subset \Proj(A)$ be a Zariski-closed subset. Consider the corresponding full subcategory 
$$\obC_\CN\subset \obC.$$ 

On the other hand, consider $\CN_\dr\subset \Proj(A)_\dr$. By \secref{sss:support on closed},
the category $\bGamma(\CN_\dr,\CC_A)$ is naturally a full 
subcategory of $\bGamma(\Proj(A)_\dr,\CC_A)$. The following assertion results from the construction (see \propref{p:sect of sheaf on formal}(b)):

\begin{lem}  \label{l:from glob to glob sections, supp}
The essential image of the subcategory $\obC_\CN\subset\obC$ under the functor \eqref{e:from glob to glob sections} 
is contained in $\bGamma(\CN_\dr,\CC_A)\subset\bGamma(\Proj(A)_\dr,\CC_A)$. 
\end{lem} 

Thus, from \lemref{l:from glob to glob sections, supp}, for every $\CN$ as above, we obtain a functor
\begin{equation} \label{e:from glob to glob sections, supp}
\obC_\CN\to \bGamma(\CN_\dr,\CC_A).
\end{equation}

\ssec{Imposing additional conditions}  

In this subsection we will recall the setting of \cite[Sect. 3.6]{AG}, where the notion of support is 
particularly explicit.

\sssec{}

First, we recall that the symmetric monoidal category $\Vect^{\on{gr}}:=\Vect^{\BG_m}$ of $\BZ$-graded objects of $\Vect$ has
a canonical automorphism, 
\begin{equation} \label{e:shift of grading functor}
M\mapsto M^{\on{shift}}
\end{equation}
that applies the cohomological shift by $[2k]$ to the $k$-th graded component, i.e.,
$$(M^{\on{shift}})_k:=M_k[2k].$$

\medskip

Let $\CB$ be an $\BE_2$-algebra in $\Vect^{\on{gr}}$. Consider the corresponding $\BE_2$-algebra $\CB^{\on{shift}}$,
and assume that it is \emph{classical}, i.e., is concentrated in cohomological degree $0$. Thus, we can regard 
$\CB^{\on{shift}}$ as a graded commutative algebra, which we can identify with
$$B:=\underset{n}\oplus\, H^{2n}(\CB),$$
and the functor \eqref{e:shift of grading functor} gives rise to a monoidal equivalence
\begin{equation}  \label{e:graded action}
(\CB\mod)^{\on{gr}}:=(\CB\mod)^{\BG_m}\simeq \QCoh(\Spec(B)/\BG_m).
\end{equation}

\sssec{}

Let $\bC$ be a DG category, equipped with an action of $\CB$. Using the forgetful functor
$$(\CB\mod)^{\on{gr}}\to \CB\mod$$
and the equivalence \eqref{e:graded action}, we obtain that $\bC$ is acted on by the (symmetric) monoidal 
category $\QCoh(\Spec(B)/\BG_m)$. 

\medskip

Let $\CN$ be a conical closed subset $\CN\subset \Spec(B)$. Then, on the one hand, we can attach to it the full
subcategory $\bC_\CN$, singled out by the cohomological support condition, see \secref{sss:abs setting} above.
 On the other hand, we can consider
the full subcategory
$$\bC\underset{\QCoh(\Spec(B)/\BG_m)}\otimes \QCoh(\Spec(B)/\BG_m)_{\CN/\BG_m}\subset 
\bC\underset{\QCoh(\Spec(B)/\BG_m)}\otimes \QCoh(\Spec(B)/\BG_m) \simeq \bC.$$

\medskip

The following assertion is \cite[Corollary 3.6.5]{AG}:

\begin{prop}  \label{p:ssupp via shift}
The full subcategories
$$\bC_\CN \subset \bC \supset 
\bC\underset{\QCoh(\Spec(B)/\BG_m)}\otimes \QCoh(\Spec(B)/\BG_m)_{\CN/\BG_m}$$
coincide.
\end{prop}

\sssec{} \label{sss:assump}

We now return to the general setting of \secref{ss:abs setting} and make the following additional assumption 
on the pair $(\CA,A)$:

\medskip

Suppose there exists an $\BE_2$-algebra $\CB$, equipped with a homomorphism 
$$\CB\to \CA,$$
such that:

\begin{itemize}

\item $\CB$ is equipped with a grading such that $\CB^{\on{shift}}$ is classical;

\item The resulting map $B:=H^\bullet(\CB)\to H^\bullet(\CA)$ can be factored as
$$B\to A\to H^\bullet(\CA),$$
where $B\to A$ is a surjection modulo nilpotents.

\end{itemize}  

We claim:

\begin{prop} \label{p:compute global sections}
Under the above assumptions on the pair $(\CA,A)$, 
the functor \eqref{e:from glob to glob sections, supp} is an equivalence.
\end{prop} 

We prove \propref{p:compute global sections} in Sections~\ref{ss:globe sect 1}--\ref{ss:globe sect 3}.

\ssec{The case of ind-coherent sheaves}  \label{ss:our case}

In this subsection we deduce \thmref{t:crystal structure} from \propref{p:compute global sections}.

\sssec{}

Let $Z$ be an affine quasi-smooth DG scheme. In the setting of \secref{ss:abs setting} we take 
$$\bC=\IndCoh(Z),\,\, \,\, \CA:=\on{HC}(Z),\,\, A:=\Gamma(\Sing(Z),\CO_{\Sing(Z)}).$$ 

\medskip

In this case $\Proj(A)=\BP\Sing(Z)$. The construction of \secref{ss:constr sheaf of cat} defines a sheaves of categories
$\CC_A$ over $(\BP\Sing(Z))_\dr$; this is the sought-for $\oInd(Z)^\sim$.

\sssec{}

The functor \eqref{e:from glob to glob sections} gives rise to a functor
\begin{equation} \label{e:from glob to glob sections sing}
\oInd(Z)\to \bGamma\left((\BP\Sing(Z))_\dr,\oInd(Z)^\sim\right).
\end{equation} 

Furthermore, for a Zariski-closed subset $\CN\subset \BP\Sing(Z)$ we obtain a functor
\begin{equation} \label{e:from glob to glob sections sing, supp}
\oInd_\CN(Z)\to \bGamma\left(\CN_\dr,\oInd(Z)^\sim\right).
\end{equation}

To prove \thmref{t:crystal structure}(a), we need to show that the functor \eqref{e:from glob to glob sections sing, supp},
and in particular, \eqref{e:from glob to glob sections sing} is an equivalence. We  do so by reducing to the situation
when \propref{p:compute global sections} becomes applicable.

\sssec{}

First, we notice that the fact that \eqref{e:from glob to glob sections sing, supp} is an equivalence can be checked
Zariski-locally on $Z$. Hence, can (and will) assume that $Z$ is a \emph{global derived complete intersection}.
This means that $Z$ fits into a Cartesian square 
\begin{equation} \label{e:global intersect}
\CD
Z  @>>>  U  \\
@VVV   @VVV   \\
\on{pt}  @>>>  V,
\endCD
\end{equation}
where $U$ is smooth, and $V$ is a vector space. 

\medskip

We claim that in this case the additional assumptions of \secref{sss:assump} are satisfied. 

\medskip

Indeed, for $Z$ fitting into the diagram \eqref{e:global intersect}, we take $\CB$ to be the $\BE_2$-algebra 
$$\Gamma(U,\CO_U)\otimes \Sym(V[-2]),$$
see \cite[Sect. 5.3.2]{AG}. The required pieces of structure on $\CB$ are described in \cite[Formula (5.9) and Sect. 5.4]{AG},
respectively. 

\ssec{Proof of \propref{p:compute global sections}, Step 1}\label{ss:globe sect 1}

Let $(\CA,A)$ and $(\CB,B)$ be as in \secref{sss:assump}. 
Let us prove that \eqref{e:from glob to glob sections} is an equivalence in the special case 
$(\CA,A)=(\CB,B)$. 

\sssec{}

According to \cite[Sect. 3.6.2]{AG}, the category $\obC$ has a natural structure
of module over $\QCoh(\Proj(B))$. 

\medskip

Let $\CC'_B$ denote the object of $\on{ShvCat}(\Proj(B))$ equal to
$$\bLoc_{\Proj(B)}(\obC),$$
where 
$$\bLoc_{\Proj(B)}:\QCoh(\Proj(B))\mmod\to \on{ShvCat}(\Proj(B))$$ is the left adjoint functor to
$\bGamma(\Proj(B),-)$,
see \cite[Sect. 1.3.1]{Ga2}.  Explicitly, for an affine DG scheme $S$ mapping to $\Proj(B)$, we have
$$\bGamma(S,\CC'_B):=\QCoh(S)\underset{\QCoh(\Proj(B))}\otimes \obC.$$

\sssec{}

Recall that $p_{\dr,\Proj(B)}$ denotes the tautological map $\Proj(B)\to \Proj(B)_\dr$. 
The key observation is provided by the following lemma, which expresses the set-theoretic
nature of singular support: 

\begin{lem}  \label{l:sects for B}
There exists a canonical isomorphism 
$$\CC_B\simeq \coind_{p_{\dr,\Proj(B)}}(\CC'_B)$$
in $\on{ShvCat}(\Proj(B)_\dr)$;
under this identification, the composite map
$$\obC\to  \bGamma(\Proj(B),\CC'_B)\simeq  \bGamma(\Proj(B)_\dr,\coind_{p_{\dr,\Proj(B)}}(\CC'_B))\simeq
\bGamma(\Proj(B)_\dr,\CC_B)$$
identifies with \eqref{e:from glob to glob sections}.
\end{lem} 

\begin{proof}

Fix $S\overset{f}\longrightarrow \Proj(B)$, and let
$(S\times \Proj(B))^\wedge_{\on{Graph}_f}$ be the formal completion of $S\times \Proj(B)$ along the graph of $f$, i.e., 
$$(S\times \Proj(B))^\wedge_{\on{Graph}_f}:=(S\times \Proj(B))\underset{(S\times \Proj(B))_\dr}\times (\on{Graph}_f)_\dr.$$

\medskip

The sheaf of categories $\coind_{p_{\dr,\Proj(B)}}(\CC'_B)$ assigns to $(S,f)$ as above the category
$$\QCoh\left((S\times \Proj(B))^\wedge_{\on{Graph}_f}\right) \underset{\QCoh(\Proj(B))}\otimes \obC.$$
which tautologically identifies with
\begin{multline*} 
\QCoh\left((S\times \Proj(B))^\wedge_{\on{Graph}_f}\right) \underset{\QCoh(S\times \Proj(B))}\otimes 
(\QCoh(S)\otimes \obC)\simeq \\
\simeq 
\QCoh(S\times \Proj(B))_{\on{Graph}_f} \underset{\QCoh(S\times \Proj(B))}\otimes 
(\QCoh(S)\otimes \obC).
\end{multline*} 

Now, the latter category identifies with $(\QCoh(S)\otimes \obC)_{\on{Graph}_f}$ by \propref{p:ssupp via shift} above.

\end{proof} 

\sssec{}

From \lemref{l:sects for B}, we obtain that in order to prove that \eqref{e:from glob to glob sections}
is an isomorphism, it suffices to show that the map
$$\obC\to  \bGamma(\Proj(B),\CC'_B)=\bGamma(\Proj(B),\bLoc_{\Proj(B)}(\obC))$$
is an isomorphism.

\medskip

However, the latter follows from the fact that $\Proj(B)$ is 1-affine, being a quasi-compact DG scheme,
see \cite[Theorem 2.1.1]{Ga2}.

\ssec{Proof of \propref{p:compute global sections}, Step 2}  \label{ss:globe sect 2}

Suppose now that for $(\CA,A)$ as in \secref{sss:assump}, the map 
\eqref{e:from glob to glob sections} is an equivalence.

\sssec{}

Applying the construction of Sects. \ref{ss:constr sheaf of cat}-\ref{ss:glob sect} to $(\CB,B)$, we obtain 
$$\CC_B\in \on{ShvCat}(\Proj(B)),$$
and a functor 
\begin{equation} \label{e:from glob to glob sections B}
\obC\to \bGamma(\Proj(B)_\dr,\CC_B).
\end{equation}

\sssec{}

By assumption, the homomorphism $B\to A$ induces a map 
$g:\Proj(A)\to \Proj(B)$, which is moreover a closed
embedding of the underlying classical reduced schemes.

\medskip

Consider the corresponding map $g_\dr:\Proj(A)_\dr\to \Proj(B)_\dr$ and the resulting adjoint pair of functors 
functor
$$\cores_{g_\dr}:\on{ShvCat}(\Proj(B)_\dr)\rightleftarrows \on{ShvCat}(\Proj(A)_\dr):\coind_{g_\dr}.$$

\medskip

Tautologically, we have:
$$\cores_{g_\dr}(\CC_B)\simeq \CC_A.$$  Moreover, under this identification, the composite map
$$\obC\to \bGamma(\Proj(B)_\dr,\CC_B)\to \bGamma(\Proj(A)_\dr,\cores_{g_\dr}(\CC_B))\simeq 
\bGamma(\Proj(A)_\dr,\CC_A)$$
identifies with \eqref{e:from glob to glob sections}.

\medskip

By adjunction, we obtain a map in $\on{ShvCat}(\Proj(B)_\dr)$:
\begin{equation} \label{e:map to coind}
\CC_B\to \coind_{g_\dr}(\CC_A).
\end{equation}

We claim:

\begin{lem} \label{l:coind A&B}
The map \eqref{e:map to coind} is an isomorphism.
\end{lem}

\begin{proof}
Clearly, $\BP\supp_B(\bc)\subset \Proj(A)\subset \Proj(B)$ for any $\bc\in \obC$. Hence
the restriction of $\CC_B$ to
$$\Proj(B)_\dr-\Proj(A)_\dr$$
vanishes. Now the claim follows from \propref{p:sect of sheaf on formal}(a).
\end{proof}

\sssec{}

As we showed in Step 1 of the proof, the functor \eqref{e:from glob to glob sections B} is an equivalence.
Since
$$\bGamma(\Proj(B)_\dr, \coind_{g_\dr}(\CC_A))\simeq \bGamma(\Proj(A)_\dr,\CC_A),$$
\lemref{l:coind A&B} implies that \eqref{e:from glob to glob sections} is an equivalence, as claimed.

\ssec{Proof of \propref{p:compute global sections}, Step 3}  \label{ss:globe sect 3}

\sssec{}

To complete the proof, it remains to show 
that the functor 
$$\obC_\CN\to \bGamma(\CN_\dr,\CC_A)$$
of \eqref{e:from glob to glob sections, supp} is an equivalence. 

\sssec{}

Let $\CN'$ be the conical Zariski-closed subset of $\Spec(A)$ such that $\CN'\supset\{0\}$ and $\CN=\BP(\CN')$. Consider
the corresponding full subcategory $\bC':=\bC_{\CN'}\subset \bC$. We have an equality
$$\obC{}'=\obC_\CN$$
of full subcategories of $\obC$. 

\medskip

Consider the corresponding sheaf of categories $\CC'_A$ over $\Proj(A)_\dr$. We have a canonical identification
$$\bGamma(\Proj(A)_\dr,\CC'_A)\simeq \bGamma(\CN_\dr,\CC_A),$$
such that the diagram
$$
\CD
\obC_\CN  @>>>  \bGamma(\CN_\dr,\CC_A) \\
@A{\sim}AA    @AA{\sim}A   \\
\obC{}'  @>>>  \bGamma(\Proj(A)_\dr,\CC'_A)
\endCD
$$
commutes.

\medskip

As shown on Step 2 of the proof (applied to $\bC'$), the bottom arrow of this diagram is an equivalence.
Hence, the top arrow is an equivalence as well. This completes the proof.

\ssec{Compatibility of monoidal actions}  \label{ss:action of base}

\sssec{}

We  now enhance the setting of \secref{ss:abs setting} to include certain pre-existing monoidal actions. 

\medskip

Suppose $\wt\CA$ is a commutative (i.e., $\BE_\infty$) algebra and $\wt\CA\to \CA$ is a homomorphism
of $\BE_2$-algebras. Assume that

\begin{itemize}

\item $\wt\CA$ is connective, i.e., $H^n(\wt\CA)=0$ for $n>0$;

\item We are given a factorization of the homomorphism $H^0(\wt\CA)\to H^0(\CA)$ as
$$H^0(\wt\CA)\to A^0\to H^0(\CA).$$

\end{itemize}

\medskip

The homomorphism $\wt\CA\to \CA$ and the action of $\CA$ on $\bC$ define an action of $\wt\CA$ on $\bC$. 
In particular, the (symmetric) monoidal category $\wt\CA\mod=\QCoh(\Spec(\wt\CA))$ acts on $\bC$, and hence
on $\obC$. 

\sssec{}

Thus, on the one hand, the category $\QCoh(\Spec(\wt\CA)_\dr)$ acts on $\obC$ via the monoidal functor
$$\QCoh(\Spec(\wt\CA)_\dr)\to \QCoh(\Spec(\wt\CA))=\wt\CA\mod\to \CA\mod$$
(where the first arrow corresponding to the tautological projection $\Spec(\wt\CA)\to \Spec(\wt\CA)_\dr$),
and the action of $\CA\mod$ on $\obC\subset \bC$. 

\medskip

On the other hand, we have the (symmetric) monoidal functor
$$\QCoh(\Spec(\wt\CA)_\dr)\simeq \QCoh(\Spec(H^0(\wt\CA))_\dr)\to \QCoh(\Spec(A^0)_\dr)\to \QCoh(\Proj(A)_\dr),$$
while $\QCoh(\Proj(A)_\dr)$ acts on $\bGamma(\Proj(A)_\dr,\CC_A)$. 

\medskip

We claim:

\begin{prop}  \label{p:identify monoidal}
The functor \eqref{e:from glob to glob sections} intertwines the above actions of $\QCoh(\Spec(\wt\CA)_\dr)$ on
$\obC$ and $\bGamma(\Proj(A)_\dr,\CC_A)$, respectively.
\end{prop}

\sssec{}  \label{sss:apply monoidal}

We  apply \propref{p:identify monoidal} as follows. We take $\wt\CA=\Gamma(Z,\CO_Z)$, which is equipped
with a canonical map to $\on{HC}(Z)$. The conclusion of \propref{p:identify monoidal} in this case implies
the compatibility statement in \thmref{t:crystal structure}(b).

\medskip

The rest of this subsection is devoted to the proof of \propref{p:identify monoidal}. 

\sssec{}

The action of $\QCoh(\Spec(\wt\CA)_\dr)$ on $\bGamma(\Proj(A)_\dr,\CC_A)$ 
amounts to a compatible family of actions of $\QCoh(\Spec(\wt\CA)_\dr)$ on the categories
$$\bGamma(S,\CC_A)=(\QCoh(S)\otimes \obC)_{\on{Graph}_f},$$ for $(S,f)\in (\affdgSch)_{/\Proj(A)_\dr}$. 

\medskip

For every $(S,f)$, the action in question is obtained as the composition of the (symmetric) monoidal functor
\begin{multline}  \label{e:comp act1}
\QCoh(\Spec(\wt\CA)_\dr)\simeq \QCoh(\Spec(H^0(\wt\CA))_\dr)\to \QCoh(\Spec(A^0)_\dr)\to \\
\to \QCoh(\Proj(A)_\dr)\to
\QCoh(S)\to \QCoh(S)\otimes \wt\CA\mod
\end{multline}
and the action of $\QCoh(S)\otimes \wt\CA\mod$ on $(\QCoh(S)\otimes \obC)_{\on{Graph}_f}$, obtained from
the monoidal functor $\QCoh(S)\otimes \wt\CA\mod\to \QCoh(S)\otimes \CA\mod$. 

\medskip

We need to show that the functor $$\obC\to (\QCoh(S)\otimes \obC)_{\on{Graph}_f}$$
of \eqref{e:from glob to loc sections} intertwines the above action with the 
action of $\QCoh(\Spec(\wt\CA)_\dr)$ on $\obC$,
obtained from 
$$\QCoh(\Spec(\wt\CA)_\dr)\to \QCoh(\Spec(\wt\CA))=\wt\CA\mod,$$
and the action of $\wt\CA\mod$ on $\obC\subset \bC$, obtained from the monoidal functor $\wt\CA\mod\to \CA\mod$. 

\medskip

Tautologically, the functor \eqref{e:from glob to loc sections} intertwines the action of 
$\QCoh(\Spec(\wt\CA)_\dr)$ on $\obC$ with its action on $(\QCoh(S)\otimes \obC)_{\on{Graph}_f}$
obtained from the composition of (symmetric) monoidal functor
\begin{equation} \label{e:comp act2}
\QCoh(\Spec(\wt\CA)_\dr)\to \QCoh(\Spec(\wt\CA))=\wt\CA\mod\overset{\CO_S\otimes-}
\longrightarrow \QCoh(S)\otimes \wt\CA\mod
\end{equation}
and the action of $\QCoh(S)\otimes \wt\CA\mod$ on 
$(\QCoh(S)\otimes \obC)_{\on{Graph}_f}$ obtained from the 
monoidal functor $\QCoh(S)\otimes \wt\CA\mod\to \QCoh(S)\otimes \CA\mod$. 

\sssec{}

Note, however, that the action of $\QCoh(S)\otimes \wt\CA\mod$ on $(\QCoh(S)\otimes \obC)_{\on{Graph}_f}$
factors through 
\begin{multline}  \label{e:comp act3}
\QCoh(S)\otimes \wt\CA\mod=\QCoh(S)\otimes \QCoh(\Spec(\wt\CA))\simeq \\
\simeq \QCoh(S\times \Spec(\wt\CA))\to
\QCoh((S\times \Spec(\wt\CA))^\wedge_{\on{Graph}_{\wt{f}}}),
\end{multline}
where $(S\times \Spec(\wt\CA))^\wedge_{\on{Graph}_{\wt{f}}}$ is the formal completion of $S\times \Spec(\wt\CA)$
along the graph of the composite map, denoted $\wt{f}$:
$$^{\on{red}}S\to \Proj(A)\to \Spec(A^0)\to\Spec(H^0(\wt\CA)) \to \Spec(\wt\CA).$$

\medskip

Thus, we need to show that the compositions of both \eqref{e:comp act1} and \eqref{e:comp act2} with 
\eqref{e:comp act3} are canonically identified as (symmetric) monoidal functors. However, 
this follows from the commutativity of the next diagram of prestacks:

$$
\CD
(S\times \Spec(\wt\CA))^\wedge_{\on{Graph}_{\wt{f}}} @>>>  S\times \Spec(\wt\CA)  @>{\on{pr}_2}>>  \Spec(\wt\CA) \\
@VVV    & &  @VVV    \\
S\times \Spec(\wt\CA)   @>{\on{pr}_1}>>    S @>{\wt{f}}>>  \Spec(\wt\CA)_\dr.
\endCD
$$






\section{Relative crystals}   \label{s:rel}

Let $f:Z\to Y$ be a map of DG schemes almost of finite type. We are interested in the category 
$$\IndCoh(Z_\dr\underset{Y_\dr}\times Y).$$

Objects of this category can be viewed as ind-coherent sheaves on $Z$ 
\emph{equipped with a connection along the fibers of the map $Z\to Y$}. When the map is smooth,
the words `connection along the fibers' can be understood literally. In general, the definition requires
the language of de Rham prestacks.

\medskip 

When $Z$ is quasi-smooth, one can use singular support to construct subcategories of $\IndCoh(Z_\dr\underset{Y_\dr}\times Y)$.
In this section we study the interaction of this construction with the crystal structure 
on the category of singularities studied in \secref{s:sing as crystal}.   

\ssec{Relative crystals as a tensor product}  \label{ss:rel as ten}

Let $f:Z\to Y$ be a map of DG schemes almost of finite type. 
Let us describe the category $\IndCoh(Z_\dr\underset{Y_\dr}\times Y)$ in terms of $\IndCoh(Y)$.

\sssec{}

First, we claim: 

\begin{prop}  \label{p:rel cryst as tensored up}
The functor 
\begin{equation} \label{e:tensor up}
\QCoh(Z_\dr\underset{Y_\dr}\times Y) 
\underset{\QCoh(Y)}\otimes \IndCoh(Y)\to \IndCoh(Z_\dr\underset{Y_\dr}\times Y),
\end{equation}
induced by the $\QCoh(Y)$-linear functor
$$(f_\dr\times \on{id})^!: \IndCoh(Y)\to  \IndCoh(Z_\dr\underset{Y_\dr}\times Y),$$ 
is an equivalence. 
\end{prop} 

\sssec{Proof of \propref{p:rel cryst as tensored up}, Step 0}

First, it is easy to see that the assertion is Zariski-local with respect to $Z$. Hence, 
we can assume that the map $f$ can be factored as
$$Z\overset{i}\hookrightarrow Z'\to Y,$$
where $i$ is a closed embedding, and $Z'$ is of the form $W\times Y$.

\medskip

Let $\oZ\overset{j}\hookrightarrow Z'$ be the embedding of the complementary open.

\sssec{Proof of \propref{p:rel cryst as tensored up}, Step 1}

We claim that the assertion of the proposition holds for $Z'$. Indeed, we have:
$$Z'_\dr\underset{Y_\dr}\times Y\simeq W_\dr\times Y,$$
and hence
$$\QCoh(Z'_\dr\underset{Y_\dr}\times Y) 
\underset{\QCoh(Y)}\otimes \IndCoh(Y)\simeq \QCoh(W_\dr)\otimes \IndCoh(Y).$$

\medskip

Similarly, 
$$\IndCoh(Z'_\dr\underset{Y_\dr}\times Y)\simeq \IndCoh(W_\dr)\otimes \IndCoh(Y).$$

Now, the functor
$$\QCoh(Z'_\dr\underset{Y_\dr}\times Y) 
\underset{\QCoh(Y)}\otimes \IndCoh(Y)\to \IndCoh(Z'_\dr\underset{Y_\dr}\times Y)$$
identifies with
$$\Upsilon_{W_\dr}\otimes \on{Id}:\QCoh(W_\dr)\otimes \IndCoh(Y)\to
 \IndCoh(W_\dr)\otimes \IndCoh(Y),$$
which is an equivalence by \cite[Proposition 2.4.4]{GR1}.

\sssec{Proof of \propref{p:rel cryst as tensored up}, Step 2}

Note that the map $Z_\dr\underset{Y_\dr}\times Y\to Z'_\dr\underset{Y_\dr}\times Y$ is an isomorphism
from $Z_\dr\underset{Y_\dr}\times Y$ to its own formal completion inside $Z'_\dr\underset{Y_\dr}\times Y$.

\medskip

Hence, we have a localization sequence
$$\QCoh(Z_\dr\underset{Y_\dr}\times Y) \rightleftarrows \QCoh(Z'_\dr\underset{Y_\dr}\times Y) \rightleftarrows
\QCoh(\oZ_\dr\underset{Y_\dr}\times Y),$$
which gives rise to the localization sequence
\begin{multline*}
\QCoh(Z_\dr\underset{Y_\dr}\times Y) \underset{\QCoh(Y)}\otimes \IndCoh(Y)
\rightleftarrows \QCoh(Z'_\dr\underset{Y_\dr}\times Y) \underset{\QCoh(Y)}\otimes \IndCoh(Y)\rightleftarrows \\
\rightleftarrows \QCoh(\oZ_\dr\underset{Y_\dr}\times Y) \underset{\QCoh(Y)}\otimes \IndCoh(Y).
\end{multline*}

Similarly, we have a localization sequence
$$\IndCoh(Z_\dr\underset{Y_\dr}\times Y) \rightleftarrows \IndCoh(Z'_\dr\underset{Y_\dr}\times Y) \rightleftarrows
\IndCoh(\oZ_\dr\underset{Y_\dr}\times Y).$$

\medskip

Combined with the fact that
$$ \QCoh(Z'_\dr\underset{Y_\dr}\times Y) \underset{\QCoh(Y)}\otimes \IndCoh(Y)\to
\IndCoh(Z'_\dr\underset{Y_\dr}\times Y)$$
is an equivalence, this implies that \eqref{e:tensor up} is fully faithful. 

\medskip

Thus, we have proved that the functor \eqref{e:tensor up} is fully faithful for \emph{any} $Z$; in particular,
it is fully faithful, and in particular, conservative, for $\oZ$.
Comparing the localization sequences, this implies that the functor \eqref{e:tensor up} is essentially surjective 
for the initial $Z$, as required.

\qed(\propref{p:rel cryst as tensored up})

\sssec{}

Assume now that $Y$ is quasi-smooth. Consider the category $$\QCoh(Z_\dr\underset{Y_\dr}\times Y) \underset{\QCoh(Y)}\otimes \IndCoh(Y),$$
appearing on the left-hand side of the equivalence in \propref{p:rel cryst as tensored up}. It contains as a full subcategory 
\begin{multline*} 
\QCoh(Z_\dr\underset{Y_\dr}\times Y) \simeq \\
\simeq \QCoh(Z_\dr\underset{Y_\dr}\times Y) \underset{\QCoh(Y)}\otimes \QCoh(Y)\overset{\on{Id}\otimes \Xi_Y}\hookrightarrow 
\QCoh(Z_\dr\underset{Y_\dr}\times Y) \underset{\QCoh(Y)}\otimes \IndCoh(Y).
\end{multline*} 

The resulting embedding 
$$\QCoh(Z_\dr\underset{Y_\dr}\times Y)\to \QCoh(Z_\dr\underset{Y_\dr}\times Y) \underset{\QCoh(Y)}\otimes \IndCoh(Y)\simeq
\IndCoh(Z_\dr\underset{Y_\dr}\times Y)$$
differs from the canonical embedding $\Upsilon_{Z_\dr\underset{Y_\dr}\times Y}$ 
(given by the action of the left-hand side on $\omega_{Z_\dr\underset{Y_\dr}\times Y}$) 
by tensoring by the pullback of $\omega_Y$. In particular, the two
embeddings have the same essential image.

\medskip

Set
$$\oInd(Z_\dr\underset{Y_\dr}\times Y):= \IndCoh(Z_\dr\underset{Y_\dr}\times Y)/\QCoh(Z_\dr\underset{Y_\dr}\times Y).$$
We  view it as a full subcategory of $\IndCoh(Z_\dr\underset{Y_\dr}\times Y)$ by identifying it with 
$$\QCoh(Z_\dr\underset{Y_\dr}\times Y)^\perp\subset \IndCoh(Z_\dr\underset{Y_\dr}\times Y).$$

\medskip

In terms of the equivalence of \propref{p:rel cryst as tensored up}, we have
$$\oInd(Z_\dr\underset{Y_\dr}\times Y)=\QCoh(Z_\dr\underset{Y_\dr}\times Y) 
\underset{\QCoh(Y)}\otimes \oInd(Y),$$
as full subcategories of
$$\IndCoh(Z_\dr\underset{Y_\dr}\times Y)\simeq \QCoh(Z_\dr\underset{Y_\dr}\times Y) \underset{\QCoh(Y)}\otimes \IndCoh(Y).$$

\sssec{}

We now claim:

\begin{prop}  \label{p:tensor up via Sing}
There exist canonical equivalences
\begin{multline} \label{e:tensor up via Sing}
\oInd(Z_\dr\underset{Y_\dr}\times Y) \simeq 
\QCoh((Z\underset{Y}\times \BP\Sing(Y))_\dr)\underset{\QCoh((\BP\Sing(Y))_\dr)}\otimes \oInd(Y)\simeq \\
\simeq \bGamma\left((Z\underset{Y}\times \BP\Sing(Y))_\dr,\oInd(Y)^\sim\right)
\end{multline}
\end{prop}

\sssec{Proof of \propref{p:tensor up via Sing}}

Let us show that we have a canonical isomorphism
$$\QCoh(Z_\dr\underset{Y_\dr}\times Y)\underset{\QCoh(Y)}\otimes \bC_Y 
\simeq \QCoh((Z\underset{Y}\times \BP\Sing(Y))_\dr)\underset{\QCoh((\BP\Sing(Y))_\dr)}\otimes \bC_Y$$
for any $\bC_Y\in \QCoh((\BP\Sing(Y))_\dr\underset{Y_\dr}\times Y)\mmod$. 

\medskip

First, the fact $Y_\dr$ is 1-affine implies that
$$\QCoh(Z_\dr)\underset{\QCoh(Y_\dr)}\otimes \QCoh(Y)\to \QCoh(Z_\dr\underset{Y_\dr}\times Y)$$
is an isomorphism, see \lemref{l:QCoh on rel}.

\medskip

Hence, 
$$\QCoh(Z_\dr\underset{Y_\dr}\times Y)\underset{\QCoh(Y)}\otimes \bC_Y\simeq
\QCoh(Z_\dr)\underset{\QCoh(Y_\dr)}\otimes  \bC_Y.$$

\medskip

Next, we rewrite
\begin{multline*}
\QCoh(Z_\dr)\underset{\QCoh(Y_\dr)}\otimes  \bC_Y\simeq \\
\simeq 
\left(\QCoh(Z_\dr)\underset{\QCoh(Y_\dr)}\otimes \QCoh((\BP\Sing(Y))_\dr)\right) \underset{ \QCoh((\BP\Sing(Y))_\dr)}\otimes \bC_Y.
\end{multline*}

Now, the fact that both $Z_\dr$ and $Y_\dr$ are 1-affine implies that the functor
\begin{multline*}
\QCoh(Z_\dr)\underset{\QCoh(Y_\dr)}\otimes \QCoh((\BP\Sing(Y))_\dr)\to\\
\to \QCoh(Z_\dr\underset{Y_\dr}\times (\BP\Sing(Y))_\dr)=\QCoh((Z\underset{Y}\times \BP\Sing(Y))_\dr)
\end{multline*}
is an equivalence.

\medskip

Hence,
$$\QCoh(Z_\dr)\underset{\QCoh(Y_\dr)}\otimes  \bC_Y\simeq 
\QCoh((Z\underset{Y}\times \BP\Sing(Y))_\dr)\underset{\QCoh((\BP\Sing(Y))_\dr)}\otimes \bC_Y,$$
as desired.

\medskip

Finally, the fact that
$$\QCoh((Z\underset{Y}\times \BP\Sing(Y))_\dr)\underset{\QCoh((\BP\Sing(Y))_\dr)}\otimes \oInd(Y)\to
\bGamma\left((Z\underset{Y}\times \BP\Sing(Y))_\dr,\oInd(Y)^\sim\right)$$
is an equivalence follows from the fact that both $\BP\Sing(Y))_\dr$ and 
$(Z\underset{Y}\times \BP\Sing(Y))_\dr$ are 1-affine.

\qed

\ssec{Relative crystals with prescribed singular support}  

Let $f:Z\to Y$ be as before. We now assume 
that $Z$ is quasi-smooth and that $f$ has a perfect relative cotangent 
complex (this is automatic if $Y$ is also quasi-smooth).

\medskip

In this subsection we show how conical subvarieties on 
$\Sing(Z)$ give rise to subcategories of  $\IndCoh(Z_\dr\underset{Y_\dr}\times Y)$.

\sssec{}

The tautological map $p_{\dr/Y,Z}:Z\to Z_\dr\underset{Y_\dr}\times Y$ gives rise to
the forgetful functor
$$(p_{\dr/Y,Z})^!: \IndCoh(Z_\dr\underset{Y_\dr}\times Y)\to \IndCoh(Z).$$

\medskip

According to \cite[Chapter III.3, Proposition 3.1.2]{GR2}, the functor
$(p_{\dr/Y,Z})^!$ is conservative and admits a left adjoint, denoted $(p_{\dr/Y,Z})^\IndCoh_*$. 
Informally, if one views ind-coherent sheaves on $Z_\dr\underset{Y_\dr}\times Y$ as (relative) D-modules for the morphism
$Z\to Y$, then $(p_{\dr/Y,Z})^\IndCoh_*$ is the induction functor from ind-coherent sheaves on $Z$ to 
relative D-modules.

\medskip

The composition $((p_{\dr/Y,Z})^!\circ (p_{\dr/Y,Z})^\IndCoh_*)$ acquires a natural structure of a monad
acting on $\IndCoh(Z)$. Denote by 
$$((p_{\dr/Y,Z})^!\circ (p_{\dr/Y,Z})^\IndCoh_*)\mod(\IndCoh(Z))$$
the category of modules over this monad. The Barr-Beck-Lurie theorem provides an equivalence
\[\IndCoh(Z_\dr\underset{Y_\dr}\times Y)\simeq ((p_{\dr/Y,Z})^!\circ (p_{\dr/Y,Z})^\IndCoh_*)\mod(\IndCoh(Z)).\] 
(The assumption that $Z$ is quasi-smooth is not required for this equivalence.)

\sssec{}  \label{sss:prescribe}

Now fix a conical Zariski-closed subset $\CN\subset \Sing(Z)$. Let
$$\IndCoh_\CN(Z_\dr\underset{Y_\dr}\times Y)\subset \IndCoh(Z_\dr\underset{Y_\dr}\times Y)$$
denote the preimage of
$$\IndCoh_\CN(Z)\subset \IndCoh(Z)$$
under the functor $(p_{\dr/Y,Z})^!$.

\medskip

We claim:

\begin{prop} \label{p:preservation of support by monad}
The functor $(p_{\dr/Y,Z})^\IndCoh_*$ sends $\IndCoh_\CN(Z)$ to $\IndCoh_\CN(Z_\dr\underset{Y_\dr}\times Y)$.
\end{prop} 

\begin{proof}

The assertion of the proposition is equivalent to the fact that $$((p_{\dr/Y,Z})^!\circ (p_{\dr/Y,Z})^\IndCoh_*),$$
viewed as a plain endofunctor of $\IndCoh(Z)$, preserves the full subcategory $\IndCoh_\CN(Z)$.

\medskip

Recall that according to \cite[Chapter IV.5, Theorem 6.1.2]{GR2}, 
$((p_{\dr/Y,Z})^!\circ (p_{\dr/Y,Z})^\IndCoh_*)$ admits a filtration whose $n$-th associated graded
is isomorphic to the functor 
\begin{equation} \label{e:tensor product by Sym}
\Sym^n(T(Z/Y))\sotimes -,
\end{equation} 
where $\Sym^n$ is taken in the symmetric monoidal category $(\IndCoh(Z),\sotimes)$, and $T(Z/Y)\in \IndCoh(Z)$
is as in \cite[Chapter III.1, Sect. 4.3.8]{GR2}.

\medskip

Thus, it suffices to show that the functor \eqref{e:tensor product by Sym} preserves the subcategory
$\IndCoh_\CN(Z)$.

\medskip

Let $T^*(Z/Y)\in \QCoh(Z)$ be the cotangent complex of $Z$. The assumption that $T^*(Z/Y)$ be perfect implies that
$T(Z/Y)\in \IndCoh(Z)$ is canonically isomorphic to
$$\Upsilon_Z((T^*(Z/Y))^\vee),$$
where $(T^*(Z/Y))^\vee\in \QCoh(Z)$ is the monoidal dual of $T^*(Z/Y)$, and where $\Upsilon_Z$ is as in
\cite[Chapter II.3, Sect. 3.2.5]{GR2}.

\medskip

Therefore,
$$\Sym^n(T(Z/Y))\simeq \Upsilon_Z \left(\Sym^n((T^*(Z/Y))^\vee)\right),$$
where $\Sym^n$ is now taken in the symmetric monoidal category $(\QCoh(Z),\otimes)$.
Hence, the functor \eqref{e:tensor product by Sym} is given by 
$$\Sym^n((T^*(Z/Y))^\vee)\otimes-,$$
where $\otimes$ denotes the action of $\QCoh(Z)$ on $\IndCoh(Z)$, and therefore preserves
$\IndCoh_\CN(Z)$, see \cite[Lemma 4.2.2]{AG}. 
\end{proof} 

\sssec{}

As a corollary of \propref{p:preservation of support by monad}, we obtain:

\begin{cor} \label{c:Barr-Beck supp}
There exists a canonical equivalence
$$\IndCoh_\CN(Z_\dr\underset{Y_\dr}\times Y)\simeq ((p_{\dr/Y,Z})^!\circ (p_{\dr/Y,Z})^\IndCoh_*)\mod(\IndCoh_\CN(Z)),$$
commuting with the forgetful functors to $\IndCoh_\CN(Z)$.
\end{cor}

\sssec{}

Let us now assume that $Y$ is quasi-smooth as well. Let $\CN\subset \Sing(Z)$ be a Zariski-closed conical subset
that contains the zero-section. We have
$$\QCoh(Z_\dr\underset{Y_\dr}\times Y)\subset \IndCoh_\CN(Z_\dr\underset{Y_\dr}\times Y),$$
as follows from the commutative diagram
$$
\CD
\QCoh(Z_\dr\underset{Y_\dr}\times Y) @>{\Upsilon_{Z_\dr\underset{Y_\dr}\times Y}}>>   \IndCoh(Z_\dr\underset{Y_\dr}\times Y) \\
@VVV    @VVV   \\
\QCoh(Z)  @>{\Upsilon_Z}>>  \IndCoh(Z).
\endCD
$$

Let us denote by
\begin{equation}  \label{e:full N 1}
\oInd_\CN(Z_\dr\underset{Y_\dr}\times Y)
\end{equation} 
the quotient
$$\IndCoh_\CN(Z_\dr\underset{Y_\dr}\times Y)/\QCoh(Z_\dr\underset{Y_\dr}\times Y),$$
considered as a full subcategory of
$$\oInd(Z_\dr\underset{Y_\dr}\times Y)\subset \IndCoh(Z_\dr\underset{Y_\dr}\times Y).$$

\sssec{} 

Recall now the map
$$\Sing(f):Z\underset{Y}\times \Sing(Y)\to \Sing(Z),$$
see \cite[Sect. 2.4.1]{AG}. For $\{0\}\subset \CN\subset \Sing(Z)$ as above, consider the closed subset
$$\Sing(f)^{-1}(\CN)\subset Z\underset{Y}\times \Sing(Y).$$

Consider the corresponding closed subset 
$$\BP(\Sing(f)^{-1}(\CN))\subset Z\underset{Y}\times \BP\Sing(Y).$$

\medskip

Consider the corresponding full subcategory 
\begin{equation} \label{e:full N 2}
\bGamma\left((\BP(\Sing(f)^{-1}(\CN)))_\dr,\oInd(Y)^\sim\right)\subset
\bGamma\left((Z\underset{Y}\times \BP\Sing(Y))_\dr,\oInd(Y)^\sim\right),
\end{equation}
or, which is the same,
\begin{multline} \label{e:full N 2 bis}
\QCoh\left((\BP(\Sing(f)^{-1}(\CN)))_\dr\right)\underset{\QCoh((\BP\Sing(Y))_\dr)}\otimes \oInd(Y)
\subset \\
\subset \QCoh((Z\underset{Y}\times \BP\Sing(Y))_\dr)\underset{\QCoh((\BP\Sing(Y))_\dr)}\otimes \oInd(Y).
\end{multline}

\sssec{}

Using \propref{p:tensor up via Sing}, we identify 
\begin{equation}   \label{e:identify again}
\oInd(Z_\dr\underset{Y_\dr}\times Y) \simeq 
\bGamma\left((Z\underset{Y}\times \BP\Sing(Y))_\dr,\oInd(Y)^\sim\right)
\end{equation}
or, equivalently,
\begin{equation}   \label{e:identify again bis}
\oInd(Z_\dr\underset{Y_\dr}\times Y) 
\simeq \QCoh((Z\underset{Y}\times \BP\Sing(Y))_\dr)\underset{\QCoh((\BP\Sing(Y))_\dr)}\otimes \oInd(Y).
\end{equation}

\medskip

We claim:

\begin{thm} \label{t:express category}
The full subcategory 
$$\oInd_\CN(Z_\dr\underset{Y_\dr}\times Y)\subset \IndCoh_\CN(Z_\dr\underset{Y_\dr}\times Y)$$
of  \eqref{e:full N 1} 
corresponds under the identifications \eqref{e:identify again} and \eqref{e:identify again bis} to the full subcategory
$$\bGamma\left((\BP(\Sing(f)^{-1}(\CN)))_\dr,\oInd(Y)^\sim\right)\subset
\bGamma\left((Z\underset{Y}\times \BP\Sing(Y))_\dr,\oInd(Y)^\sim\right)$$
from \eqref{e:full N 2}, or, equivalently, to the full subcategory 
\begin{multline*} 
\QCoh\left((\BP(\Sing(f)^{-1}(\CN)))_\dr\right)\underset{\QCoh((\BP\Sing(Y))_\dr)}\otimes \oInd(Y)
\subset \\
\subset \QCoh((Z\underset{Y}\times \BP\Sing(Y))_\dr)\underset{\QCoh((\BP\Sing(Y))_\dr)}\otimes \oInd(Y).
\end{multline*}
from \eqref{e:full N 2 bis}.
\end{thm} 

\sssec{An example}

Let us take $\CN=\{0\}$. In this case we have the following three full subcategories of
$\IndCoh(Z_\dr\underset{Y_\dr}\times Y)$. The largest is $\IndCoh(Z_\dr\underset{Y_\dr}\times Y)$ 
itself. 

\medskip

The smallest is 
$$\QCoh(Z_\dr\underset{Y_\dr}\times Y)\subset \IndCoh(Z_\dr\underset{Y_\dr}\times Y).$$

\medskip

The middle category is $\IndCoh_{\{0\}}(Z_\dr\underset{Y_\dr}\times Y)$, i.e., the preimage of
$\QCoh(Z)\subset \IndCoh(Z)$ under the forgetful functor
$$\IndCoh(Z_\dr\underset{Y_\dr}\times Y)\to \IndCoh(Z).$$ 

\medskip

In terms of the identification
\begin{multline*} 
\oInd(Z_\dr\underset{Y_\dr}\times Y)=\IndCoh(Z_\dr\underset{Y_\dr}\times Y)/\QCoh(Z_\dr\underset{Y_\dr}\times Y)\simeq \\
\simeq \QCoh((Z\underset{Y}\times \BP\Sing(Y))_\dr)\underset{\QCoh((\BP\Sing(Y))_\dr)}\otimes \oInd(Y).
\end{multline*}
of \propref{p:tensor up via Sing}, the subcategory 
$$\oInd_{\{0\}}(Z_\dr\underset{Y_\dr}\times Y)\subset \oInd(Z_\dr\underset{Y_\dr}\times Y)$$
corresponds to subscheme
$$\BP(\Sing(f)^{-1}(\{0\}))\subset Z\underset{Y}\times \BP\Sing(Y).$$

\ssec{Proof of \thmref{t:express category}, Step 1}

We  first show that the assertion of the theorem holds when $f:Z\to Y$ is a closed embedding. 

\sssec{}

Note that in this case $Z_\dr\underset{Y_\dr}\times Y$ is the formal completion $Y^\wedge_Z$ of $Y$ along $Z$.
In particular, $\IndCoh(Z_\dr\underset{Y_\dr}\times Y)$ identifies with the full subcategory
$$\IndCoh(Y)_Z\subset \IndCoh(Y),$$
consisting of objects that are set-theoretically supported on $Z\subset Y$. 

\medskip

Recall that the categories $\IndCoh(Z)$ and $\IndCoh(Z_\dr\underset{Y_\dr}\times Y)$ are related by a pair
of adjoint functors
\[(p_{\dr/Y,Z})^\IndCoh_*:\IndCoh(Z)\rightleftarrows \IndCoh(Z_\dr\underset{Y_\dr}\times Y):(p_{\dr/Y,Z})^!\]
(the induction functor and the forgetful functor). Under the equivalence 
\[\IndCoh(Z_\dr\underset{Y_\dr}\times Y)\simeq \IndCoh(Y)_Z,\]
they are identified with the pair of adjoint functors
\[f^\IndCoh_*:\IndCoh(Z)\rightleftarrows \IndCoh(Y)_Z:f^!|_{\IndCoh(Y)_Z}.\]

\medskip

Under the identification of \eqref{e:tensor up}, the full subcategory
$$\QCoh(Z_\dr\underset{Y_\dr}\times Y) \underset{\QCoh(Y)}\otimes \oInd(Y) 
\subset \QCoh(Z_\dr\underset{Y_\dr}\times Y) \underset{\QCoh(Y)}\otimes \IndCoh(Y)$$
corresponds to
$$\IndCoh(Y)_Z\cap \oInd(Y)\subset \IndCoh(Y)_Z\simeq \IndCoh(Z_\dr\underset{Y_\dr}\times Y).$$

Furthermore, the diagram
$$
\CD
\QCoh(Z_\dr\underset{Y_\dr}\times Y) \underset{\QCoh(Y)}\otimes \oInd(Y)   @>{\text{\propref{p:tensor up via Sing}}}>{\sim}>  
\bGamma\left((Z\underset{Y}\times \BP\Sing(Y))_\dr,\oInd(Y)^\sim\right)  \\
@VV{\sim}V    @VVV   \\
\IndCoh(Y)_Z\cap \oInd(Y) & &  \bGamma\left((\BP\Sing(Y))_\dr,\oInd(Y)^\sim\right)   \\
@VVV   @VV{\sim}V   \\
\oInd(Y) @>{\on{Id}}>> \oInd(Y)
\endCD
$$
commutes. 

\sssec{}

Set
$$\CM:=\Sing(f)^{-1}(\CN)\subset Z\underset{Y}\times \Sing(Y)\subset \Sing(Y).$$
Let $\BP\CM$ denote the corresponding Zariski-closed subset of $\BP\Sing(Y)$. 

\medskip

Then, by \thmref{t:crystal structure}(a), 
$$\bGamma\left(\BP\CM_\dr,\oInd(Y)^\sim\right)\subset 
\bGamma\left((Z\underset{Y}\times \BP\Sing(Y))_\dr,\oInd(Y)^\sim\right)\subset \oInd(Y)$$
identifies with the full subcategory of $\oInd(Y)$ equal to
$$\oInd_{\BP\CM}(Y)=\oInd(Y)\cap \IndCoh_\CM(Y).$$



\medskip

Therefore, in order to establish the assertion of the theorem, it is sufficient to show that
$$\IndCoh_\CM(Y)=\IndCoh_\CN(Z_\dr\underset{Y_\dr}\times Y),$$
as subcategories of $\IndCoh(Z_\dr\underset{Y_\dr}\times Y)\simeq \IndCoh(Y)_Z$. 

\sssec{}

Thus, we need to show that $\IndCoh_\CM(Y)\subset \IndCoh(Y)_Z$ equals the preimage
of $\IndCoh_\CN(Z)$ under the functor $f^!:\IndCoh(Y)_Z\to \IndCoh(Z)$. 

\medskip

We note that the inclusion
$$\IndCoh_\CM(Y)\subset (f^!)^{-1}(\IndCoh_\CN(Z))$$
follows from \cite[Proposition 7.1.3(a)]{AG}. 

\medskip

For the opposite inclusion, by \corref{c:Barr-Beck supp}, it suffices to show that the essential image of 
$\IndCoh_\CN(Z)$ under the functor
$$f^\IndCoh_*:\IndCoh(Z)\to \IndCoh(Y)_Z$$ 
is contained in $\IndCoh_\CM(Y)$. However, this follows from \cite[Proposition 7.1.3(b)]{AG}. 

\ssec{Proof of \thmref{t:express category}, Step 2}

We  now consider the case of a general morphism $f:Z\to Y$. 

\sssec{}

It is easy to see that the assertion of the theorem is Zariski-local on $Z$. Hence, we can assume that the morphism $f$ factors
as $$Z\overset{f'}\to Y' \overset{g}\to Y,$$
where $Z\to Y'$ is a closed embedding, and $g$ is smooth. Furthermore, we can assume that $Y'$ is isomorphic to
$Y\times W$ with $W$ smooth. 

\medskip

By Step 1, we know that the statement of the theorem holds for the morphism $Z\to Y'$. 

\sssec{}

Consider the (forgetful) functor 
\begin{equation} \label{e:Y to Y'}
(\on{id}\times g)^!:\IndCoh(Z_\dr\underset{Y_\dr}\times Y)\to \IndCoh(Z_\dr\underset{Y'_\dr}\times Y').
\end{equation}

By definition,
$$\IndCoh_\CN(Z_\dr\underset{Y_\dr}\times Y)\subset \IndCoh(Z_\dr\underset{Y_\dr}\times Y)$$
is the preimage under \eqref{e:Y to Y'} of 
$$\IndCoh_\CN(Z_\dr\underset{Y'_\dr}\times Y')\subset \IndCoh(Z_\dr\underset{Y'_\dr}\times Y').$$

\sssec{}

The fact that $g$ is smooth implies that
$$\Sing(g):Y'\underset{Y}\times \Sing(Y)\to \Sing(Y')$$
is an isomorphism. In particular,
$$Z\underset{Y}\times \Sing(Y)\simeq Z\underset{Y'}\times \Sing(Y').$$

\medskip

Under this identification, the loci
$$\Sing(f)^{-1}(\CN) \subset Z\underset{Y}\times \Sing(Y) \text{ and }
\Sing(f')^{-1}(\CN)\subset Z\underset{Y'}\times \Sing(Y')$$
correspond to one another. 

\medskip

Under the identifications of \propref{p:tensor up via Sing} for $Y$ and $Y'$, respectively, the pullback functor
$$\QCoh(Z_\dr\underset{Y_\dr}\times Y)\underset{\QCoh(Y)}\otimes \oInd(Y) \to
\QCoh(Z_\dr\underset{Y'_\dr}\times Y')\underset{\QCoh(Y')}\otimes \oInd(Y')$$
corresponds to the functor
\begin{multline}  \label{e:another Y to Y'}
\QCoh((Z\underset{Y}\times \BP\Sing(Y))_\dr)\underset{\QCoh((\BP\Sing(Y))_\dr)}\otimes \oInd(Y)  \to \\
\to \QCoh((Z\underset{Y'}\times \BP\Sing(Y'))_\dr)\underset{\QCoh((\BP\Sing(Y'))_\dr)}\otimes \oInd(Y'),  
\end{multline} 

Hence, we obtain that in order to prove the theorem, it suffices to show that the preimage of 
\begin{multline} \label{e:on Y'}
\QCoh((\BP(\Sing(f')^{-1}(\CN)))_\dr)\underset{\QCoh((\BP\Sing(Y'))_\dr)}\otimes \oInd(Y')
\subset \\
\subset \QCoh((Z\underset{Y'}\times \BP\Sing(Y'))_\dr)\underset{\QCoh((\BP\Sing(Y'))_\dr)}\otimes \oInd(Y').
\end{multline}
under the functor \eqref{e:another Y to Y'} equals 
\begin{multline} \label{e:on Y}
\QCoh((\BP(\Sing(f)^{-1}(\CN)))_\dr)\underset{\QCoh((\BP\Sing(Y))_\dr)}\otimes \oInd(Y)
\subset \\
\subset \QCoh((Z\underset{Y}\times \BP\Sing(Y))_\dr)\underset{\QCoh((\BP\Sing(Y))_\dr)}\otimes \oInd(Y).
\end{multline}

\medskip

I.e., it suffices to show that the functor
\begin{multline*}  
\QCoh((\BP(\Sing(f)^{-1}(\CN)))_\dr)^\perp\underset{\QCoh((\BP\Sing(Y))_\dr)}\otimes \oInd(Y)\to \\ 
\to 
\QCoh((\BP(\Sing(f')^{-1}(\CN)))_\dr)^\perp\underset{\QCoh((\BP\Sing(Y'))_\dr)}\otimes \oInd(Y')
\end{multline*}
is conservative. 

\sssec{}

Since $g$ is smooth, the functor $g^!$ induces an equivalence
$$\QCoh(Y')\underset{\QCoh(Y)}\otimes \IndCoh(Y)\to \IndCoh(Y').$$

Hence, 
$$\QCoh((\BP(\Sing(f')^{-1}(\CN)))_\dr)^\perp\underset{\QCoh((\BP\Sing(Y'))_\dr)}\otimes \oInd(Y')$$
is obtained from 
$$\QCoh((\BP(\Sing(f)^{-1}(\CN)))_\dr)^\perp\underset{\QCoh((\BP\Sing(Y))_\dr)}\otimes \oInd(Y)$$ by
the procedure
$$-\underset{\QCoh(Y'_\dr)\underset{\QCoh(Y_\dr)}\otimes \QCoh(Y)}\otimes \QCoh(Y').$$

\medskip

Now, we claim that for any $\bC\in \QCoh(Y'_\dr\underset{Y_\dr}\times Y)\mmod$, the resulting functor
\begin{equation} \label{e:left adjoint}
\bC\to \bC \underset{\QCoh(Y'_\dr)\underset{\QCoh(Y_\dr)}\otimes \QCoh(Y)}\otimes \QCoh(Y')
\end{equation} 
is conservative. 

\medskip

To show this, it is enough to prove that the pullback functor
\begin{equation} \label{e:smooth dr}
\QCoh(Y'_\dr)\underset{\QCoh(Y_\dr)}\otimes \QCoh(Y)\to \QCoh(Y')
\end{equation} 
admits a left adjoint, which is compatible with the action of $\QCoh(Y'_\dr)\underset{\QCoh(Y_\dr)}\otimes \QCoh(Y)$,
and whose essential image generates $\QCoh(Y'_\dr)\underset{\QCoh(Y_\dr)}\otimes \QCoh(Y)$ as a DG category. 

\medskip 

Indeed, such a left adjoint implies the existence of a left adjoint to \eqref{e:left adjoint}, 
whose essential image generates $\bC$. 

\sssec{}

To establish the required property of \eqref{e:smooth dr}, we use the assumption that $Y'=Y\times W$
with $W$ smooth. 

\medskip

We write
$$\QCoh(Y'_\dr)\underset{\QCoh(Y_\dr)}\otimes \QCoh(Y)\simeq \QCoh(Y)\otimes \QCoh(W_\dr)$$
and
$$\QCoh(Y')\simeq \QCoh(Y)\otimes \QCoh(W).$$

\medskip

Thus, our assertion follows from the fact that the forgetful functor
$$\QCoh(W_\dr)\to \QCoh(W)$$
does admit a left adjoint with the required properties. 

\newpage 

\centerline{\bf Part II: Gluing.}

\bigskip

\section{A paradigm for gluing}  \label{s:gluing}

In this section we formulate the main result of this paper, \thmref{t:main}. 

\ssec{Gluing and lax limits: a reminder}     \label{ss:gluing}

\sssec{}  \label{sss:setting for gluing}

Let $I$ be an index $\infty$-category, and let 
$$i\mapsto \bC_i, \quad (\alpha:i\to j)\mapsto (\Phi_\alpha:\bC_i\to \bC_j).$$
be a functor $I\to \StinftyCat_{\on{cont}}$. 

\medskip

Let $\bC_I$ be the corresponding co-Cartesian fibration over $I$.  The lax limit 
$$\underset{i\in I}{\on{lax-lim}}\, \bC_i$$
is the object of $\StinftyCat_{\on{cont}}$ equal to the category of \emph{all} (i.e., not necessarily co-Cartesian)
sections $I\to \bC_I$ of the projection $\bC_I\to I$. 

\medskip

We have a fully faithful embedding
$$\underset{i\in I}{\on{lim}}\, \bC_i\hookrightarrow \underset{i\in I}{\on{lax-lim}}\, \bC_i$$
that corresponds to taking \emph{co-Cartesian} sections. 

\sssec{}

Objects of $\underset{i\in I}{\on{lax-lim}}\, \bC_i$ can be concretely described as follows: 
An object of $\underset{i\in I}{\on{lax-lim}}\, \bC_i$ is a collection
$$\bc_i\in \bC_i\quad\text{for all }i\in I,$$
equipped with a family of morphisms (but not necessarily isomorphisms)
$$\Phi_\alpha(\bc_i)\to \bc_j\quad\text{for all }\alpha:i\to j,$$
compatible with compositions of $\alpha$'s, and endowed with a homotopy-coherent system
of compatibilities for multi-fold compositions. 

\medskip

An object as above belongs to $\underset{i\in I}{\on{lim}}\, \bC_i$ if and only if the above maps $\Phi_\alpha(\bc_i)\to \bc_j$
are all isomorphisms. 

\sssec{}  \label{sss:functor to glue}

Unwinding the definitions, for a given $\bD\in \StinftyCat_{\on{cont}}$, the datum of a functor
$$\sF:\bD\to \underset{i\in I}{\on{lax-lim}}\, \bC_i$$
consists of a collection of functors
$$\sF_i:\bD\to \bC_i\quad\text{for all }i\in I$$
equipped with a compatible family of natural transformations (but not necessarily isomorphisms)
$$\Phi_\alpha\circ \sF_i\to \sF_j\quad\text{for all }\alpha:i\to j.$$

\medskip

In particular, by taking $\bD=\Vect$, we obtain the description of objects of $\underset{i\in I}{\on{lax-lim}}\, \bC_i$, given above.

\sssec{}

We think of $\underset{i\in I}{\on{lax-lim}}\, \bC_i$ as glued from the categories $\bC_i$ using
the functors $\Phi_\alpha$. 

For this reason, we  also denote
$$\underset{i\in I}{\on{lax-lim}}\, \bC_i=:\on{Glue}(\bC_i,\,i\in I).$$

\begin{rem}\label{rem:limit over lax functor}
The category $\on{Glue}(\bC_i,\,i\in I)$ can be defined in a more general situation. Namely, we do not
need 
$$i\mapsto \bC_i,\quad I\to \StinftyCat_{\on{cont}}$$
to be a functor, but only (either left or right) \emph{lax} functor. I.e., we do not need to have 
an isomorphism between $\Phi_\alpha \circ \Phi_\beta$ and $\Phi_{\alpha\circ \beta}$, but only a morphism
in one direction. 

\medskip

We do not need this more general set-up in the present paper. 
\end{rem} 

\sssec{Example}\label{ex:gluing sheaves}
Let $Y$ be a topological space, and let $Y_0\overset{j}\hookrightarrow Y$ be an open subset and
$Y_1\overset{i}\hookleftarrow Y$ be the complementary closed. Let $I$ be the category $0\to 1$, and 
set 
$$\bC_0=\on{Shv}(Y_0),\,\, \bC_1=\on{Shv}(Y_1),\,\, \Phi_{0\to 1}=i^!\circ j_!.$$

Then the functor
$$\on{Shv}(Y)\to \on{Glue}(\bC_i,\,i\in I),\quad \CF\mapsto (j^!(\CF),i^!(\CF),i^!\circ j_!\circ j^!(\CF)\to i^!(\CF))$$
is an equivalence. The inverse functor sends 
$$(\CF_0,\CF_1,i^!\circ j_!(\CF_0)\to \CF_1)\mapsto \on{Cone}\left(i_!(\on{ker}(i^!\circ j_!(\CF_0)\to \CF_1))\to 
j_!(\CF_0)\right).$$

\sssec{Example}\label{ex:gluing sheaves on strat}
Example~\ref{ex:gluing sheaves} can be generalized to arbitrary stratified topological spaces, but this
requires taking lax limits over lax functors, as in Remark~\ref{rem:limit over lax functor}. Namely, let
$Y=\bigcup_{a\in A}Y_a$ be a stratification of a topological space $Y$ indexed by a finite poset $A$.
Thus, the subspaces $Y_a\subset Y$ are disjoint and locally closed, and
\[\overline{Y_a}\subset\bigcup_{a'\ge a}Y_{a'}\subset Y\]
for all $a\in A$. Denote the embedding $Y_a\hookrightarrow Y$ by $\iota_a$.

\medskip

For every pair $a_1,a_2\in A$ with $a_1\le a_2$, consider the functor 
\[\Phi_{a_1\to a_2}:=\iota_{a_2}^!\circ\iota_{a_1,!}:\on{Shv}(Y_{a_1})\to \on{Shv}(Y_{a_2}).\]
For a triple $a_1,a_2,a_3\in A$ with $a_1\le a_2\le a_3$, the adjunction between $\iota_{a_2}^!$ and $\iota_{a_2,!}$
yields a natural transformation \[\left(\Phi_{a_2\to a_3}\circ\Phi_{a_1\to a_2}\right)\to \Phi_{a_1\to a_3}.\]
In this way, we obtain a lax functor $I\to \StinftyCat_{\on{cont}}$ (here $I$ is the category 
corresponding to the poset $A$) sending $a\in A$ to the category $\on{Shv}(Y_a)$. Similarly to Example~\ref{ex:gluing sheaves}, there is a 
natural equivalence between the resulting glued category and $\on{Shv}(Y)$.

\sssec{}  \label{sss:insert}

For every $i_0\in I$ we let $\on{ev}_{i_0}$ denote the natural evaluation functor
$$\underset{i\in I}{\on{lax-lim}}\, \bC_i\to \bC_{i_0}.$$

The functor $\on{ev}_{i_0}$ admits a left adjoint, denoted
$\on{ins}_{i_0}$. Explicitly, for $\bc_{i_0}\in \bC_{i_0}$ and $i\in I$, we have
$$\on{ev}_i\circ \on{ins}_{i_0}(\bc_{i_0})\simeq \underset{\alpha\in \Maps_I(i_0,i)}{\on{colim}}\, \Phi_\alpha(\bc_{i_0}).$$

\begin{rem}
The latter expression for $\on{ev}_i\circ \on{ins}_{i_0}$ is a feature of \emph{lax limits} of DG categories; it is false for 
usual limits. 
\end{rem} 

\sssec{Subcategories}  \label{sss:glue subcat}

Let $i\mapsto \bC_i$ be as before. Suppose now that for every $i\in I$ we chose a full subcategory
$$\bC'_i\subset \bC_i.$$
These subcategories define a full subcategory $\bC'_I\subset\bC_I$. 

\medskip

Assume that the following condition
holds: for every $(\alpha:i\to j)\in I$, the functor $\Phi_\alpha$ sends $\bC'_i$ to $\bC'_j$. 
In this case, the composition 
$$\bC'_I\hookrightarrow \bC_I\to I$$
is a co-Cartesian fibration, and hence gives rise to a functor
$$i\mapsto \bC'_i, \quad I\to \StinftyCat_{\on{cont}}.$$

\medskip

Consider the corresponding category
$$\underset{i\in I}{\on{lax-lim}}\, \bC'_i=:\on{Glue}(\bC'_i,\,i\in I).$$

By construction, we have a canonical fully faithful functor
\begin{equation} \label{e:embed glue}
\on{Glue}(\bC'_i,\,i\in I)\to \on{Glue}(\bC_i,\,i\in I),
\end{equation} 
that commutes with the evaluation functors $\on{ev}_{i_0}$. 

\medskip

Finally, assume that in the above setting, each of the embeddings 
$\bC'_i\hookrightarrow \bC_i$ admits a continuous right adjoint. In this case, it is
easy to show that the functor \eqref{e:embed glue} also admits a continuous right
adjoint. 

\medskip

The resulting right adjoint $\on{Glue}(\bC_i,\,i\in I)\to \on{Glue}(\bC'_i,\,i\in I)$ 
also commutes with the evaluation functors $\on{ev}_{i_0}$.

\ssec{Gluing of $\IndCoh$}  \label{ss:gluing IndCoh}

\sssec{} 

Consider the following set-up. Let $\CY$ be an algebraic stack. Let $I$ be an index category, and let
$$i\mapsto \CZ_i, \quad (\alpha:i\to j)\mapsto (f_\alpha:\CZ_j\to \CZ_i).$$
be an $I^{\on{op}}$-diagram of algebraic stacks over $\CY$. We denote by $f_i$ the corresponding morphisms
$\CZ_i\to \CY$. 

\medskip  

We assume that $\CY$ and all $\CZ_i$ are quasi-smooth.

\sssec{}

We consider
$$i\mapsto \IndCoh((\CZ_i)_\dr\underset{\CY_\dr}\times \CY), \quad (\alpha:i\to j)\mapsto 
((f_\alpha)_\dr\times \on{id}_\CY)^!$$
as a functor $I\to \StinftyCat_{\on{cont}}$.

\medskip

Let now
$$\CN_i\subset \Sing(\CZ_i)$$
be conical Zariski-closed subsets. We  assume that for every $\alpha:i\to j$ the map
$$\Sing(f_\alpha):\CZ_j\underset{\CZ_i}\times \Sing(\CZ_i)\to \Sing(\CZ_j)$$
sends $\CZ_j\underset{\CZ_i}\times \CN_i$ to $\CN_j$.

\medskip

Consider the corresponding full subcategories
$$\IndCoh_{\CN_i}((\CZ_i)_\dr\underset{\CY_\dr}\times \CY)\subset \IndCoh((\CZ_j)_\dr\underset{\CY_\dr}\times \CY).$$

According to \cite[Lemma 8.4.2]{AG}, the above condition on $f_\alpha$ implies that the functor
$$((f_\alpha)_\dr\times \on{id}_\CY)^!$$
sends $\IndCoh_{\CN_i}((\CZ_i)_\dr\underset{\CY_\dr}\times \CY)$ to 
$\IndCoh_{\CN_j}((\CZ_j)_\dr\underset{\CY_\dr}\times \CY)$. 

\sssec{}

We consider the corresponding pair of adjoint functors 
\begin{equation} \label{e:glue adj}
\on{Glue}(\IndCoh_{\CN_i}((\CZ_i)_\dr\underset{\CY_\dr}\times \CY),\,i\in I) 
\rightleftarrows \on{Glue}(\IndCoh((\CZ_i)_\dr\underset{\CY_\dr}\times \CY),\,i\in I). 
\end{equation}

\medskip

The functors $((f_i)_\dr\times \on{id}_Y)^!$ define a functor
$$\IndCoh(\CY) \to \lim_{i\in I}\IndCoh((\CZ_i)_\dr\underset{\CY_\dr}\times \CY).$$

Thus, for a given conical Zariski-closed subset $\CN\subset \Sing(\CY)$ we obtain the functor
\begin{multline} \label{e:functor in question}
\IndCoh_\CN(\CY) \hookrightarrow \IndCoh(\CY) \to \lim\limits_{i\in I} 
\IndCoh((\CZ_i)_\dr\underset{\CY_\dr}\times \CY) \hookrightarrow \\
\hookrightarrow \on{Glue}(\IndCoh((\CZ_i)_\dr\underset{\CY_\dr}\times \CY),\,i\in I)\to 
\on{Glue}(\IndCoh_{\CN_i}((\CZ_i)_\dr\underset{\CY_\dr}\times \CY),\,i\in I),
\end{multline}  
where the last arrow is the right adjoint from \eqref{e:glue adj}. This functor is our main object
of interest.

\begin{rem}
Note that the image of \eqref{e:functor in question} is usually \emph{not} contained in the full subcategory
\[\lim\limits_{i\in I}\IndCoh_{\CN_i}((\CZ_i)_\dr\underset{\CY_\dr}\times \CY)\subset 
\on{Glue}(\IndCoh_{\CN_i}((\CZ_i)_\dr\underset{\CY_\dr}\times \CY),\,i\in I).\]
\end{rem} 

\ssec{The setting for the main theorem}   \label{ss:setting for main}

\sssec{}  \label{sss:setting for main}

We  now consider a particular case of the above situation. Let $G$ be a reductive group. 

\medskip

We 
let $I^{\on{op}}$ be the category corresponding to the poset $\on{Par}(G)$ of standard parabolics
in $G$ (i.e., the set of subsets of vertices of the Dynkin diagram of $G$). 

\medskip

Given a curve $X$, we let $\CY:=\LocSys_G$ be the algebraic stack of $G$-local systems on $X$.
We consider the functor
$$P\in \on{Par}(G)\,\,\mapsto \,\,\CZ_P:=\LocSys_P.$$

\medskip

We take
$$\CN:=\on{Nilp}_{\on{glob}}\subset \Sing(\LocSys_G)$$
to be the global nilpotent cone, see \cite[Sect. 11.1.1]{AG}. See also \secref{sss:glob cone}
for an explicit description of $\on{Nilp}_{\on{glob}}$. 

\medskip

For every $P\in \on{Par}(G)$, we take $\CN_P\subset \Sing(\LocSys_P)$ to be the zero-section $\{0\}$. 

\sssec{}  

The following conjecture was made by us (it was recorded as \cite[Conjecture 9.3.7]{Ga3}):

\begin{conj} \label{c:main}
The functor 
$$\IndCoh_{\on{Nilp}_{\on{glob}}}(\LocSys_G)\to
\on{Glue}(\IndCoh_{\{0\}}((\LocSys_P)_\dr\underset{(\LocSys_G)_\dr}\times \LocSys_G),\,P\in \on{Par}(G)^{\on{op}})$$
of \eqref{e:functor in question} 
is fully faithful.
\end{conj}

The main result of this paper is:

\begin{thm}  \label{t:main}
\conjref{c:main} holds.
\end{thm} 

The rest of this paper is devoted to the proof of this theorem. 

\ssec{Gluing for D-modules}  \label{ss:setting Dmod}

In this subsection we  formulate another gluing situation, in the context of D-modules. 
We then state a result that says that (under certain circumstances) the full faithfulness 
of the functor \eqref{e:functor in question} is equivalent to the full faithfulness of a
certain functor in the context of D-modules. 

\sssec{}

In what follows, for a prestack $\CY$ locally almost of finite type we  consider the category $\Dmod(\CY)$.
By definition, 
$$\Dmod(\CY):=\QCoh(\CY_\dr),$$
and thus can be viewed as a symmetric monoidal category.

\medskip

Recall also that according to \cite[Proposition 2.4.4]{GR1}, the functor $\Upsilon_{\CY_\dr}$ defines an equivalence
$$\QCoh(\CY_\dr)\to \IndCoh(\CY_\dr).$$

\medskip

For a morphism $g:\CY_1\to \CY_2$ we  denote by $g^{\dr,!}$ the corresponding pullback functor
$$\Dmod(\CY_2)\to \Dmod(\CY_1).$$

\medskip

By definition, $g^{\dr,!}$ identifies with either of the vertical arrows in the following diagram:
$$
\CD
\QCoh((\CY_1)_\dr)   @>{\Upsilon_{(\CY_1)_\dr}}>> \IndCoh((\CY_1)_\dr)  \\
@A{(g_\dr)^*}AA    @AA{(g_\dr)^!}A     \\
\QCoh((\CY_2)_\dr)   @>{\Upsilon_{(\CY_2)_\dr}}>> \IndCoh((\CY_2)_\dr).  
\endCD
$$

\medskip

If $g$ is schematic and quasi-compact, we  denote by $g_{\dr,*}$ the corresponding
direct image functor
$$\Dmod(\CY_1)\to \Dmod(\CY_2).$$

\sssec{}

Let $\CY'$ be a prestack locally almost of finite type. Let $I$ be again an index category, and let
$$i\mapsto \CZ'_i, \quad (\alpha:i\to j)\mapsto (f'_\alpha:\CZ'_j\to \CZ'_i).$$
be an $I^{\on{op}}$-diagram of algebraic stacks over $\CY$. We denote by $f'_i$ the corresponding morphisms
$\CZ'_i\to \CY'$. 

\medskip

We consider
$$i\mapsto \Dmod(\CZ'_i), \quad (\alpha:i\to j)\mapsto (f'_\alpha)^{\dr,!}$$
as a functor $I\to \StinftyCat_{\on{cont}}$.

\medskip

Let now
$$\CM_i\subset \CZ'_i$$ be Zariski-closed subsets. We  assume that for every $\alpha:i\to j$ we have
$$(f'_\alpha)^{-1}(\CM_i)\subset \CM_j.$$

Let $\CM$ be a Zariski-closed subset of $\CY'$. 

\sssec{}

We consider the corresponding pair of adjoint functors 
\begin{equation} \label{e:adj glue Dmod}
\on{Glue}(\Dmod(\CM_i),\,i\in I) \rightleftarrows \on{Glue}(\Dmod(\CZ'_i),\,i\in I).
\end{equation}

\medskip

The functors $(f'_i)^{\dr,!}$ define a functor
$$\Dmod(\CY') \to \underset{i\in I}{\on{lim}}\, \Dmod(\CZ'_i).$$

Consider the composition
\begin{multline} \label{e:functor for Dmod}
\Dmod(\CM)\hookrightarrow  \Dmod(\CY') \to \underset{i\in I}{\on{lim}}\, \Dmod(\CZ'_i) \hookrightarrow \\
\hookrightarrow \on{Glue}(\Dmod(\CZ'_i),\,i\in I)\to \on{Glue}(\Dmod(\CM_i),\,i\in I),
\end{multline}  
where the last arrow is the right adjoint from \eqref{e:adj glue Dmod}.

\sssec{}

Consider again the setting of \secref{ss:gluing IndCoh}. Put 
\begin{align*}
\CY'&=\BP\Sing(\CY)&\CZ'_i&=\CZ_i\underset{\CY}\times \BP\Sing(\CY)\\
\CM&=\BP(\CN)& \CM_i&=\BP\left(\Sing(f_i)^{-1}(\CN_i)\right)\subset \CZ'_i.
\end{align*}

In \secref{s:proof top to IndCoh} we prove:

\begin{thm}  \label{t:top to IndCoh} Assume that the maps $f_i:\CZ_i\to \CY$ are schematic and proper. 
Assume also that the following conditions hold:

\begin{enumerate}

\item For every index $i$, we have $\{0\}\subset \CN_i$
and $$\Sing(f_i)^{-1}(\CN_i)\subset \CZ_i\underset{\CY}\times \CN.$$

\smallskip

\item The functor
$$\QCoh(\CY)\to \underset{i\in I}{\on{lim}}\, \QCoh((\CZ_i)_\dr\underset{\CY_\dr}\times \CY)$$
is fully faithful;

\smallskip

\item The functor
$$\Dmod(\CM)\to \on{Glue}\left(\Dmod\left(\CM_i\right),\,i\in I\right)$$
of \eqref{e:functor for Dmod} is fully faithful. 

\end{enumerate}

Then the functor 
$$\IndCoh_\CN(\CY) \to \on{Glue}(\IndCoh_{\CN_i}((\CZ_i)_\dr\underset{\CY_\dr}\times \CY),\,i\in I)$$
of \eqref{e:functor in question} is fully faithful.

\end{thm} 

\begin{rem}
In \secref{ss:ff Dmod via contr} we  express condition (3) in \thmref{t:top to IndCoh} in more concrete
terms: it amounts to acyclicity of certain explicit objects of $\Vect$, or, equivalently, to 
\emph{homological contractibility} of certain homotopy types. 

\medskip

Thus, \thmref{t:top to IndCoh} claims that a certain full faithfulness assertion for $\IndCoh$ 
is essentially of topological nature. The proof of \thmref{t:top to IndCoh} is based on \thmref{t:express category} from
Part I of the paper. 
\end{rem}

\begin{rem}
With a little extra work, one can show that \thmref{t:top to IndCoh} holds without the condition that
$$\Sing(f_i)^{-1}(\CN_i)\subset \CZ_i\underset{\CY}\times \CN.$$
\end{rem}

\sssec{}

We will apply \thmref{t:top to IndCoh} to deduce \thmref{t:main}. We take $I=\on{Par}(G)^{\on{op}}$
and $\CY,\CZ_i,\CN,\CN_i$ as in \secref{sss:setting for main}.

\medskip

Note that condition (1) of \thmref{t:top to IndCoh} is trivially satisfied. Condition (2) is satisfied because
the category $\on{Par}(G)^{\on{op}}$ has an initial object (the improper parabolic $P=G$), so
$$\underset{P\in \on{Par}(G)^{\on{op}}}{\on{lim}}\, \QCoh((\LocSys_P)_\dr
\underset{(\LocSys_G)_\dr}\times \LocSys_G)\simeq \QCoh(\LocSys_G).$$

\medskip

Thus, \thmref{t:main} follows from \thmref{t:top to IndCoh}, combined with the following result:

\begin{thm} \label{t:Dmod ff LocSys}
The functor
$$\Dmod\left(\BP(\on{Nilp}_{\on{glob}})\right)\to 
\on{Glue}\left(\Dmod\left(\BP(\CM_P)\right),\,P\in \on{Par}(G)^{\on{op}}\right)$$
is fully faithful, where
$$\CM_P\subset \LocSys_P\underset{\LocSys_G}\times \Sing(\LocSys_G)$$
is the preimage of $\{0\}\subset \Sing(\LocSys_P)$ under the map
$$\LocSys_P\underset{\LocSys_G}\times \Sing(\LocSys_G)\to \Sing(\LocSys_P).$$
\end{thm}

We prove \thmref{t:Dmod ff LocSys} in Part III of the paper.  

\section{Proof of \thmref{t:top to IndCoh}} \label{s:proof top to IndCoh}

\ssec{A criterion for fully faithfulness} 

\sssec{}

Let $(\bC_i,\Phi_\alpha)$ be as in \secref{sss:setting for gluing}. Let $\bC_i'\subset \bC_i$
be full subcategories such that
$$\Phi_\alpha(\bC'_i)\subset \bC'_j,\quad (\alpha:i\to j)\in I.$$ 

Set $\obC_i:=(\bC_i')^\perp\subset \bC_i$. Assume that 
$$\Phi_\alpha(\obC_i)\subset \obC_j,\quad (\alpha:i\to j)\in I.$$

\medskip

Denote
$$\bC:=\on{Glue}(\bC_i,\,i\in I), \quad \bC':=\on{Glue}(\bC'_i,\,i\in I),\quad \obC:=\on{Glue}(\obC_i,\,i\in I).$$

Thus, we have a pair of full subcategories
$$\bC'\hookrightarrow \bC \hookleftarrow \obC.$$

We have an inclusion
$$\obC\subset (\bC')^\perp,$$
which, in general, is not an equality. 

\sssec{}

Let now
$$\sF_i:\bD\to \bC_i$$ be a family of functors as in \secref{sss:functor to glue}. 

\medskip

Let $\bD'\subset \bD$
be a full subcategory, and set 
$$\obD:=(\bD')^\perp \subset \bD.$$

We  assume that for every $i\in I$, the functor $\sF_i$ satisfies:
$$\sF_i(\bD')\subset \bC'_i, \quad \sF_i(\obD)\subset \obC_i.$$

These conditions imply that $\sF$ restricts to well-defined functors
$$\sF':\bD'\to \bC' \text{ and } \osF:\obD\to \obC.$$

\medskip

We claim:

\begin{prop} \label{p:ff ses}
Assume that:

\smallskip

\noindent{\em(a)} Each of the functors $\sF_i$ admits a left adjoint, denoted $\sF_i^L$, and 
$$\sF_i^L(\obC_i)\subset \obD\quad\text{for all }i\in I.$$ 

\smallskip

\noindent{\em(b)} The functors $\sF'$ and $\osF$ are both fully faithful. 

\medskip

Then $\sF$ is also fully faithful. 
\end{prop}

\sssec{}

The rest of this subsection is devoted to the proof of \propref{p:ff ses}, which is rather formal. 

\medskip

It is easy to see that the assumption that the functors $\sF_i$ each admits a left adjoint implies that 
the functor $\sF:\bD\to \bC$ admits a left adjoint\footnote{In \secref{ss:expl left adjoint} we give a more explicit description of the functor $\sF^L$.} 
(denoted $\sF^L$), which satisfies
$$\sF^L\circ \on{ins}_i\simeq \sF^L_i\quad\text{for all }i\in I,$$
where $\on{ins}_i$ is as in \secref{sss:insert}. 

\medskip

We will need the following:

\begin{lem}  \label{l:left adjoints}
If
$$\sF_i^L(\obC_i)\subset \obD\quad\text{for all }i\in I,$$
then the diagram
$$
\CD
\bC   @<<<  \obC  \\
@V{\sF^L}VV    @VV{\osF{}^L}V    \\
\bD   @<<<   \obD
\endCD
$$
commutes.
\end{lem}

\begin{proof}

It is enough to show that for every $i\in I$ the diagram
$$
\CD
& & \obC_i \\
& & @VV{\on{ins}_i}V  \\
\bC   @<<<  \obC  \\
@V{\sF^L}VV    @VV{\osF{}^L}V    \\
\bD   @<<<   \obD
\endCD
$$
commutes.  Note, however, that the diagram 
$$
\CD
\bC_i   @<<< \obC_i \\
@V{\on{ins}_i}VV   @VV{\on{ins}_i}V    \\
\bC   @<<<  \obC 
\endCD
$$
commutes. Hence, it is enough to establish the commutativity of
$$
\CD
\bC_i   @<<< \obC_i \\ 
@V{\sF^L\circ \on{ins}_i}VV     @VV{\osF{}^L\circ \on{ins}_i}V    \\
\bD   @<<<   \obD.
\endCD
$$

However, the latter diagram identifies with
$$
\CD
\bC_i   @<<< \obC_i \\ 
@V{\sF^L_i}VV     @VV{\osF{}^L_i}V    \\
\bD   @<<<   \obD,
\endCD
$$
and the commutativity follows from the assumption.

\end{proof}

\sssec{Proof of \propref{p:ff ses}}

It suffices to check that for $\bd'\in \bD'$ and $\obd\in \obD$, the map
\begin{equation} \label{e:compare Homs}
\Hom_{\bD}(\obd,\bd')\to \Hom_{\bC}(\osF(\obd),\sF'(\bd'))
\end{equation} 
is an isomorphism. 

\medskip

Using \lemref{l:left adjoints}, we can identify \eqref{e:compare Homs} with the map
$$\Hom_{\bD}(\obd,\bd')\to  \Hom_{\bD}(\osF{}^L\circ \osF(\obd),\bd'),$$
which comes from the co-unit of the adjunction
$$\osF{}^L\circ \osF(\obd)\to \obd.$$

Since, the latter is an an isomorphism ($\osF$ was assumed fully faithful), the assertion
of the proposition follows.

\qed

\ssec{Proof of \thmref{t:top to IndCoh}, Step 0}  

\sssec{}

It is easy to see by descent that the property of the functor \eqref{e:functor in question} to be fully faithful
is local in the smooth topology on $\CY$.  The same is true for conditions (2) and (3) in \thmref{t:top to IndCoh}.

\medskip

Hence, we can assume that $\CY=:Y$ and $\CZ_i=:Z_i$ are DG schemes. 

\sssec{}

We  prove \thmref{t:top to IndCoh} by applying \propref{p:ff ses}. We take
$$\bD=\IndCoh_\CN(Y),\quad \bD'=\QCoh(Y),\quad \obD=\oInd(Y)\cap \IndCoh_\CN(Y).$$

We take
$$\bC_i:=\IndCoh_{\CN_i}((Z_i)_\dr\underset{Y_\dr}\times Y).$$

Recall the identification 
$$\IndCoh((Z_i)_\dr\underset{Y_\dr}\times Y)\simeq \QCoh((Z_i)_\dr\underset{Y_\dr}\times Y)\underset{\QCoh(Y)}\otimes \IndCoh(Y)$$
of \propref{p:rel cryst as tensored up}.

\medskip

We take 
\begin{multline*}
\bC'_i=\QCoh((Z_i)_\dr\underset{Y_\dr}\times Y))=
\QCoh((Z_i)_\dr\underset{Y_\dr}\times Y))\underset{\QCoh(Y)}\otimes \QCoh(Y) \subset \\
\subset \QCoh((Z_i)_\dr\underset{Y_\dr}\times Y)\underset{\QCoh(Y)}\otimes \IndCoh(Y)\simeq
\IndCoh((Z_i)_\dr\underset{Y_\dr}\times Y).
\end{multline*}

We have:
$$\bC'_i\subset \IndCoh_{\{0\}}((Z_i)_\dr\underset{Y_\dr}\times Y)\subset 
\IndCoh_{\CN_i}((Z_i)_\dr\underset{Y_\dr}\times Y)=\bC_i.$$

\medskip

Thus,
$$\obC_i=
\left(\QCoh((Z_i)_\dr\underset{Y_\dr}\times Y)\underset{\QCoh(Y)}\otimes \oInd(Y)\right)\cap
\left(\IndCoh_{\CN_i}((Z_i)_\dr\underset{Y_\dr}\times Y)\right).$$

\medskip

The functors $\sF_i$ are the compositions
\begin{multline}  \label{e:functor i}
\IndCoh_\CN(Y)\hookrightarrow \IndCoh(Y) \overset{((f_i)_\dr\times \on{id}_Y)^!}\longrightarrow \\
\to \IndCoh((Z_i)_\dr\underset{Y_\dr}\times Y)\to \IndCoh_{\CN_i}((Z_i)_\dr\underset{Y_\dr}\times Y),
\end{multline}
where the last arrow is the right adjoint to the embedding
$$\IndCoh_{\CN_i}((Z_i)_\dr\underset{Y_\dr}\times Y)\hookrightarrow
\IndCoh((Z_i)_\dr\underset{Y_\dr}\times Y).$$

\medskip

It is clear that the above functor sends $\bD'=\QCoh(Y)$ to
$\bC'_i=\QCoh((Z_i)_\dr\underset{Y_\dr}\times Y))$.

\medskip

In Steps 1 and 2 below we will verify that the above data satisfies the conditions of \propref{p:ff ses}.

\ssec{Proof of \thmref{t:top to IndCoh}, Step 1} 

In this subsection we will show the following:

\medskip

\noindent(i) The above functor $\sF_i:\bD\to \bC_i$ 
$$\IndCoh_\CN(Y)\to \IndCoh_{\CN_i}((Z_i)_\dr\underset{Y_\dr}\times Y)$$
admits a left adjoint. 

\medskip

\noindent(ii) The left adjoint in (i) sends $\bC'_i$ to $\bD'$, i.e.,  
$$\QCoh((Z_i)_\dr\underset{Y_\dr}\times Y)\subset \IndCoh_{\CN_i}((Z_i)_\dr\underset{Y_\dr}\times Y)$$
to 
$$\QCoh(Y)\subset \IndCoh_\CN(Y).$$

\medskip

Note that (ii) is equivalent to the fact that $\sF_i$ sends $\obD$ to $\obC_i$. 

\medskip

\noindent(iii) The left adjoint in (i) sends $\obC_i$ to $\obD$. 

\sssec{}  \label{sss:left adjoint}

First, we claim that the functor
$$\IndCoh(Y) \overset{((f_i)_\dr\times \on{id}_Y)^!}\longrightarrow \IndCoh((Z_i)_\dr\underset{Y_\dr}\times Y)$$
admits a left adjoint\footnote{More conceptually, the left adjoint in question exists because the map
$(f_i)_\dr\times \on{id}_Y)$ is \emph{inf-schematic and nil-proper}, see \cite[Chapter III.3, Proposition 3.2.4]{GR2}.}.
Indeed, we rewrite
$$\IndCoh((Z_i)_\dr\underset{Y_\dr}\times Y)\simeq
\QCoh((Z_i)_\dr\underset{Y_\dr}\times Y)\underset{\QCoh(Y)}\otimes \IndCoh(Y).$$

So, it is enough to show that the functor 
$$((f_i)_\dr\times \on{id}_Y)^*:\QCoh(Y)\to \QCoh((Z_i)_\dr\underset{Y_\dr}\times Y)$$
admits a left adjoint (it automatically commutes with the action of $\QCoh(Y)$, because $\QCoh(Y)$
is rigid as a monoidal category). 

\medskip

We write
$$\QCoh((Z_i)_\dr\underset{Y_\dr}\times Y) \simeq \QCoh((Z_i)_\dr)\underset{\QCoh(Y_\dr)}\otimes \QCoh(Y),$$
see \lemref{l:QCoh on rel}.

\medskip

So, it is enough to show that the functor
$$(f_i)_\dr^*:\QCoh(Y_\dr)\to \QCoh((Z_i)_\dr)$$
admits a left adjoint, which commutes with the action of $\QCoh(Y_\dr)$.  

\medskip

We interpret the latter functor as
$$f_i^{\dr,!}:\Dmod(Y)\to \Dmod(Z_i).$$

\medskip

Since $f_i$ is proper, the left adjoint in question
is the functor
$$(f_i)_{\dr,*}:\Dmod(Z_i)\to \Dmod(Y).$$
The commutativity with the action of $\QCoh(Y_\dr)=\Dmod(Y)$ is given by the projection
formula for $(f_i)_{\dr,*}$. 

\sssec{}

Now, the left adjoint to the functor \eqref{e:functor i} is given by the composition
$$
\IndCoh_{\CN_i}((Z_i)_\dr\underset{Y_\dr}\times Y) \hookrightarrow  
\IndCoh((Z_i)_\dr\underset{Y_\dr}\times Y) \overset{(((f_i)_\dr\times \on{id}_Y)^!)^L}\longrightarrow \IndCoh(Y).
$$

We claim that the essential image of the above functor belongs to $\IndCoh_\CN(Y)$. Indeed, by 
\corref{c:Barr-Beck supp}, it suffices to show that the composition
\begin{multline*}
\IndCoh_{\CN_i}(Z_i)\to 
\IndCoh_{\CN_i}((Z_i)_\dr\underset{Y_\dr}\times Y) \hookrightarrow \\
\to \IndCoh((Z_i)_\dr\underset{Y_\dr}\times Y) \overset{(((f_i)_\dr\times \on{id}_Y)^!)^L}\longrightarrow \IndCoh(Y)
\end{multline*}
maps to $\IndCoh_\CN(Y)$. However, the latter functor identifies with
$$\IndCoh_{\CN_i}(Z_i) \hookrightarrow \IndCoh(Z_i) \overset{(f_i)^\IndCoh_*}\longrightarrow
\IndCoh(Y),$$
and the desired containment follows from condition (1) in \thmref{t:top to IndCoh} and \cite[Proposition 7.1.3(b)]{AG}.

\sssec{}

The fact that the left adjoint to \eqref{e:functor i} sends 
$$\QCoh((Z_i)_\dr\underset{Y_\dr}\times Y)\subset \IndCoh_{\CN_i}((Z_i)_\dr\underset{Y_\dr}\times Y)$$
to 
$$\QCoh(Y)\subset \IndCoh_\CN(Y)$$
follows from the construction. 

\sssec{}

The fact that the left adjoint to \eqref{e:functor i} sends $\obC_i$ to $\oInd(Y)$ follows 
from the fact that the functor left adjoint to $((f_i)_\dr\times \on{id}_Y)^!$ sends
$$\QCoh((Z_i)_\dr\underset{Y_\dr}\times Y)\underset{\QCoh(Y)}\otimes \oInd(Y)$$ to $\oInd(Y)$,
which follows from the description of this left adjoint in \secref{sss:left adjoint}. 

\ssec{Proof of \thmref{t:top to IndCoh}, Step 2} 

In order to apply \propref{p:ff ses}, we need to show that the functors
$\QCoh(Y)\to \bC'$
and 
$\oInd(Y)\cap \IndCoh_\CN(Y)\to \obC$
are both fully faithful.

\sssec{}

The fact that $\QCoh(Y)\to \bC'$ is fully faithful is given by condition (2) in \thmref{t:top to IndCoh}. 

\sssec{}

It remains to show that the functor 
\begin{equation} \label{e:glue open}
\oInd(Y)\cap \IndCoh_\CN(Y)\to \obC
\end{equation}
is fully faithful. 

\medskip

We are now going to use the results from Part I of the paper. Namely, according to \thmref{t:express category}, 
the functor $$I\to \StinftyCat_{\on{cont}}, \quad i\mapsto \obC_i$$
identifies with the functor
$$i\mapsto \Dmod\left(\BP\left(\Sing(f_i)^{-1}(\CN_i)\right)\right)\underset{\Dmod(\BP\Sing(Y))}\otimes \oInd(Y).$$

Similarly, by \thmref{t:crystal structure}, 
$$\oInd(Y)\cap \IndCoh_\CN(Y)\simeq
\Dmod(\BP\CN)\underset{\Dmod(\BP\Sing(Y))}\otimes \oInd(Y).$$

\sssec{}



We have the following general assertion: 

\begin{lem}  \label{l:glue and tensor}
Suppose that in the setting of \secref{sss:setting for gluing}, the functor $I\to \StinftyCat_{\on{cont}}$
$$i\mapsto \bC_i,\quad (\alpha:i\to j)\mapsto \Phi_\alpha$$
upgrades to a functor $I\to \bO\mod$, where $\bO$ is a monoidal DG category. Then
for a right $\bO$-module category $\wt\bC$, the functor
$$\wt\bC \underset{\bO}\otimes \on{Glue}(\bC_i,\,i\in I)\to \on{Glue}(\wt\bC\underset{\bO}\otimes \bC_i,\,i\in I)$$
is an equivalence.
\end{lem} 

\begin{proof}

Follows from \secref{sss:insert}.

\end{proof} 

\medskip

Applying \lemref{l:glue and tensor}, we obtain that the functor \eqref{e:glue open} identifies with the functor obtained from
\begin{equation}
\sF_{\Dmod}:\Dmod(\BP(\CN))\to \on{Glue}\left(\Dmod\left(\BP\left(\Sing(f_i)^{-1}(\CN_i)\right)\right),\,i\in I\right)
\end{equation}
by tensoring over $\Dmod(\BP\Sing(Y))$ with $\oInd(Y)$. 

\sssec{}

The functor $\sF_{\Dmod}$ admits a left adjoint that commutes with the monoidal
action of $\Dmod(\BP\Sing(Y))$ (by the same argument as in \lemref{l:left adjoints});
denote it by $\sF_{\Dmod}^L$. Hence, the functor \eqref{e:glue open} also admits a left
adjoint that can be identified with
$$\sF_{\Dmod}^L\otimes \on{Id}_{\oInd(Y)}.$$

We need to show that the co-unit of the adjunction
\begin{multline*}
(\sF_{\Dmod}\otimes \on{Id}_{\IndCoh(Y)})^L\circ (\sF_{\Dmod}\otimes \on{Id}_{\IndCoh(Y)})\simeq \\
\simeq (\sF^L_{\Dmod}\otimes \on{Id}_{\IndCoh(Y)})\circ (\sF_{\Dmod}\otimes \on{Id}_{\IndCoh(Y)})\simeq 
(\sF^L_{\Dmod}\circ \sF_{\Dmod}) \otimes \on{Id}_{\IndCoh(Y)}\to \on{Id}
\end{multline*}
is an isomorphism.

\medskip

For that, it is enough to know that $\sF^L_{\Dmod}\circ \sF_{\Dmod} \to \on{Id}$ is an isomorphism, i.e.,
that $\sF_{\Dmod}$ is fully faithful.

\medskip

However, the latter is given by condition (3) in \thmref{t:top to IndCoh}.

\section{Gluing for D-modules and homological contractibility}  \label{s:fibers}

For the rest of the paper we work within the usual (as opposed to derived) algebraic geometry. The reason for this
is that for a derived scheme $Y$, the map $^{\on{cl}}Y\to Y$ gives rise to an isomorphism
$({}^{\on{cl}}Y)_\dr\to Y_\dr$ (here $^{\on{cl}}Y$ denotes the classical scheme underlying $Y$),
so the pullback functor $\Dmod(Y)\to \Dmod({}^{\on{cl}}Y)$ is an equivalence. 

\medskip

From now on, our goal is to prove \thmref{t:Dmod ff LocSys}; that is, we need to verify condition (3) in  
\thmref{t:top to IndCoh} in a particular situation. Condition (3) may appear somewhat obscure. 
In this section, we restate it in more concrete terms as 
\emph{homological contractibility} of certain homotopy types. 

\ssec{D-modules on prestacks}

In this subsection we consider a simplified version of the set-up of \secref{ss:setting Dmod}, namely, one where $\CM_i=\CZ'_i$,
and where instead of the glued category we consider the actual (strict) limit. 

\medskip

Strictly speaking, some of this material is not necessary
for the sequel; it is included for completeness and in order to familiarize the reader with the objects involved.  
For a more comprehensive review of the theory, the reader is referred to \cite[Sects. 1 and 7.4]{Ga4}.

\medskip

For the duration of the paper, we let $\affSch$ denote the category of (classical) affine schemes of finite type. 

\sssec{}

Recall that for a prestack $\CZ$, the DG category $\Dmod(\CZ)$ is defined to be
$$\underset{S\in (\affSch_{/\CZ})^{\on{op}}}{\on{lim}}\, \Dmod(S),$$
where the limit is formed using !-pullbacks as transition functors.

\medskip

If $\CZ$ is written as a colimit over $I^{\on{op}}$ (where $I$ is an index $\infty$-category) as
\begin{equation} \label{e:prestack colimit}
\CZ\simeq\underset{i\in I^{\on{op}}}{\on{colim}}\, Z_i,
\end{equation} 
where $Z_i\in \Sch$, then the functor 
$$I^{\on{op}}\to  \Sch_{/\CZ},\quad i\mapsto Z_i$$
is cofinal, and so the restriction map
$$\Dmod(\CZ)\to \underset{i\in I}{\on{lim}}\, \Dmod(Z_i).$$
is an equivalence. 

\sssec{}  \label{sss:pseudo-proper}

A prestack $\CZ$ is said to be a \emph{pseudo-scheme} if it admits a presentation \eqref{e:prestack colimit}
$$\CZ\simeq\underset{i\in I^{\on{op}}}{\on{colim}}\, Z_i,$$
where $Z_i\in \Sch$, and the transition maps $Z_i\to Z_j$ are proper. 

\medskip

In this case, by \cite[Proposition 1.7.5]{DrGa}, the evaluation functors
$$\on{ev}_i:\Dmod(\CZ)\to \Dmod(Z_i)$$
admit left adjoints (to be denoted $\on{ins}_i$), and the resulting functor
\begin{equation} \label{e:Dmod as colim}
\underset{i\in I^{\on{op}}}{\on{colim}}\, \Dmod(Z_i)\to \Dmod(\CZ),
\end{equation} 
is an equivalence. In the formation of the above colimit, for an arrow $i_1\overset{\alpha}\to i_2$ in $I$
and the corresponding map $Z_{i_2}\overset{f_\alpha}\longrightarrow Z_{i_1}$, the functor
$$\Dmod(Z_{i_2})\to  \Dmod(Z_{i_2})$$
is $(f_\alpha)_{\dr,!}=(f_\alpha)_{\dr,*}$.  The functors $\Dmod(Z_i)\to \Dmod(\CZ)$ in \eqref{e:Dmod as colim}
are $\on{ins}_i$.

\sssec{} \label{sss:pseudo-proper dir image}

Suppose now that $\CZ$ is a prestack over a scheme $Y$ that 
admits a presentation \eqref{e:prestack colimit} 
where all $Z_i$ are \emph{proper} over $Y$. We then say that 
$\CZ$ is \emph{pseudo-proper} over $Y$.

\medskip

Let $f$ (reps. $f_i$) denote the map $\CZ\to Y$ (reps. $Z_i\to Y$). Consider the pullback functor
$$f^{\dr,!}:\Dmod(Y)\to \Dmod(\CZ).$$

By \secref{sss:pseudo-proper}, the functor $f^{\dr,!}$ admits a left adjoint, to be denoted $f_{\dr,!}$, which is given in terms
of the equivalence \eqref{e:Dmod as colim} by the compatible family of functors
$$(f_i)_{\dr,!}=(f_i)_{\dr,*}:\Dmod(Z_i)\to \Dmod(Y).$$
That is, $f_{\dr,!}\circ \on{ins}_i\simeq (f_i)_{\dr,!}$. 

\medskip

The properness assumption on the $f_i$'s implies the following base-change
property: for a morphism of schemes $Y'\overset{g}\longrightarrow Y$ and the corresponding Cartesian square
\begin{equation} \label{e:base change ps proper}
\CD
\CZ'  @>{g_Z}>>  \CZ  \\
@V{f'}VV    @VV{f}V   \\
Y'  @>{g}>>  Y,
\endCD
\end{equation} 
the canonical map
\begin{equation} \label{e:base change isom}
(f')_{\dr,!}\circ (g_Z)^{\dr,!}\to g^{\dr,!}\circ f_{\dr,!},
\end{equation} 
arising by adjunction from the isomorphism $(g_Z)^{\dr,!}\circ f^{\dr,!}\simeq (f')^{\dr,!}\circ g^{\dr,!}$, is an isomorphism.

\sssec{}
 
We say that a prestack $\CZ$ over a scheme $Y$ is  \emph{homologically contractible} over $Y$ if the pullback functor 
$$(f)^{\dr,!}:\Dmod(Y)\to \Dmod(\CZ)$$
is fully faithful.  

\medskip

Since $\CZ$ is pseudo-proper over $Y$, the functor $f_{\dr,!}$ admits a right adjoint $f^{\dr,!}$. Hence, 
$\CZ$ is homologically contractible over $Y$ if and only if the co-unit of the adjunction
$$f_{\dr,!}\circ f^{\dr,!}\to \on{Id}_{\Dmod(Y)}$$
is an isomorphism. 

\medskip

The endofunctor $f_{\dr,!}\circ f^{\dr,!}$ of $\Dmod(Y)$ can be described explicitly as
\begin{equation} \label{e:monad as colim}
f_{\dr,!}\circ f^{\dr,!}(\CF)=\colim_{i\in I^{\on{op}}}(f_i)_{\dr,!}\circ (f_i)^{\dr,!}(\CF).
\end{equation} 
Therefore, $\CZ$ is homologically contractible over $\CZ$ if and only if the natural map
\[\colim_{i\in I^{\on{op}}}(f_i)_{\dr,*}\circ (f_i)^{\dr,!}(\CF)\to \CF\]
is an isomorphism for every $\CF\in \Dmod(Y)$.

\sssec{} \label{sss:contractibleoverpt}
Assume for a moment that $Y=\on{pt}$, so that $\Dmod(Y)=\Vect$. 
Then the endofunctor $f_{\dr,!}\circ f^{\dr,!}$ of $\Vect$ is given by
tensor product with the object 
$$f_{\dr,!}(\omega_\CZ)=f_{\dr,!}\circ f^{\dr,!}(k).$$
(As a side remark, $f_{\dr,!}(\omega_\CZ)$ is defined even if $\CZ$ is not pseudo-proper; 
this is due to the fact that the value of $\omega_\CZ$ on any $S\in \Sch_{/\CZ}$ is $\omega_S$, which is holonomic.)

\medskip

Put
$$f_{\dr,!}(\omega_\CZ)=:\on{C}_*(\CZ)\in\Vect.$$
We call $\on{C}_*(\CZ)$ the \emph{homology} of $\CZ$. 

\begin{rem} 
When $k=\BC$, we can attach to $\CZ$ a homotopy type $\CZ^{\on{top}}$ given by
$$\CZ^{\on{top}}:=\colim_{S\in \Sch_{/\CZ}}S^{\on{top}}.$$
Here $S\mapsto S^{\on{top}}$ is the functor sending a scheme to
the underlying analytic space, and the colimit is taken in the $\infty$-category of spaces.
The Riemann-Hilbert correspondence yields a canonical isomorphism
$$\on{C}_*(\CZ)\simeq \on{C}_*(\CZ^{\on{top}},k),$$
where the right-hand side is the homology of the homotopy type $\CZ^{\on{top}}$.
\end{rem}

\medskip

We now claim:

\begin{lem}  \label{l:pointwise}
Let $\CZ$ be pseudo-proper over $Y$. Then the following conditions are equivalent:

\smallskip

\noindent{\em(i)}
The prestack $\CZ$ is homologically contractible over $Y$;

\smallskip

\noindent{\em(ii)}
The prestack $\CZ$ is \emph{universally} homologically contractible over $Y$:
for any morphism of schemes $Y'\to Y$, the fiber product $\CZ':=\CZ\underset{Y}\times Y$ is homologically contractible over $Y'$;

\smallskip

\noindent{\em(iii)}
The map $$f_{\dr,!}(\omega_\CZ)\simeq f_{\dr,!}\circ f^{\dr,!}(\omega_Y)\to \omega_Y$$ is an isomorphism;

\smallskip

\noindent{\em(iv)}
For every field extension $k'\supset k$ and every $k'$-point $y$ of $Y$, the prestack 
$\CZ_y=\Spec(k')\underset{Y}\times \CZ$  is homologically trivial over $\Spec(k')$;

\smallskip

\noindent{\em(v)}
For every field extension $k'\supset k$ and every $k'$-point $y$ of $Y$, the $k'$-prestack 
$\CZ_y$ has trivial homology:
the natural map
$$\on{C}_*(\CZ_y)\to k'$$
is an isomorphism. 

\end{lem}

\begin{proof}

We have (ii) $\Rightarrow$ (i) $\Rightarrow$ (iii) for tautological reasons. The implication
(iii) $\Rightarrow$ (v) follows from base change (\secref{sss:pseudo-proper dir image}). The equivalence
(iv) $\Leftrightarrow$ (v) follows from \secref{sss:contractibleoverpt}.
Let us prove that (iv) $\Rightarrow$ (ii).

\medskip

Note first that if for a scheme $Y$ and $\CF\in \Dmod(Y)$, we have $\CF=0$ 
if and only if for every field extension $k'\supset k$
and every $k'$-point $y$ of $Y$, the !-fiber of 
$$k'\underset{k}\otimes \CF\in \Dmod(k'\underset{k}\otimes Y)$$
at $y$ is zero. 
Using the fact that the formation of $\CF\mapsto f_{\dr,!}\circ f^{\dr,!}(\CF)$ commutes with field extensions 
(which follows, for instance, from the description of $f_{\dr,!}\circ f^{\dr,!}$ as \eqref{e:monad as colim}), and the base-change
isomorphism \eqref{e:base change isom}, we obtain that (ii) is equivalent to each $\CZ_y$ being
homologically contractible, as claimed.

\end{proof} 

\ssec{Explicit description of the left adjoint: a digression}    \label{ss:expl left adjoint} 

Consider the general set-up of \secref{sss:functor to glue}. Thus, we have an index category $I$ and an $I$-diagram
of categories
$$i\mapsto \bC_i,\quad (i\overset{\alpha}\longrightarrow j)\mapsto (\bC_i\overset{\Phi_\alpha}\longrightarrow \bC_j).$$

Let $\sF$ be the functor 
$$\bD\to \on{Glue}(\bC_i,\,i\in I)$$
given by a lax-compatible family of functors $\sF_i:\bD\to \bC_i$. Let us assume that
each of the functors $\sF_i$ admits a left adjoint, which we denote $\sF_i^L$. 

\medskip

Let us give an explicit formula for the left adjoint of the functor
$$\sF:\bD\to \on{Glue}(\bC_i,\,i\in I).$$

\sssec{}

Consider the category $\on{String}(I)$, whose objects are
strings of objects of $I$:
\begin{equation} \label{e:string}
(i_0\to i_1\to\dots\to i_n),
\end{equation} 
and whose morphisms are induced by order preserving maps $[m]\to [n]$. In other words, $\on{String}(I)$ is
the co-Cartesian fibration in groupoids over $\bDelta^{\on{op}}$ corresponding to the functor
$$\bDelta^{\on{op}}\to \infty\on{-Grpd}$$
given by the nerve of $I$. 

\sssec{}

There exists a canonically defined functor
$$\sF^L_{\on{String}}:\on{Glue}(\bC_i,\,i\in I)\to \on{Funct}(\on{String}(I),\bD).$$

Namely, given an object
$$\{i\mapsto \bc_i,\,\, (i\overset\alpha\longrightarrow j) \mapsto 
(\Phi_\alpha(\bc_i)\to  \bc_j)\}\in \on{Glue}(\bC_i,\,i\in I),$$
the functor $\sF^L_{\on{String}}$ sends it to the functor $\on{String}(I)\to \bD$ that sends \eqref{e:string} to 
$$\sF^L_{i_n}(\Phi_{i_0\to i_n}(\bc_{i_0})).$$

\sssec{}

Consider the composition functor
\begin{equation} \label{e:candidate for left adj}
\on{Glue}(\bC_i,\,i\in I) \overset{\sF^L_{\on{String}}}\longrightarrow \on{Funct}(\on{String}(I),\bD) \overset{\colim}\longrightarrow \bD,
\end{equation}
where the right arrow is the functor of colimit along $\on{String}(I)$.
We claim:

\begin{prop} \label{p:left adj as geom real}
The functor \eqref{e:candidate for left adj}
is the left adjoint of the functor $\sF$.
\end{prop}

\begin{proof}

We can factor $\sF$ as a composition
$$\bD\to \on{Glue}(\bD,\,i\in I) \to \on{Glue}(\bC_i,\,i\in I),$$
where $\on{Glue}(\bD,\,i\in I)$ is formed using the constant functor 
\begin{equation} \label{e:constant family}
I\to \StinftyCat_{\on{cont}},\quad i\mapsto \bD.
\end{equation} 

\medskip

This reduces the statement of the proposition to the case when $\bC_i=\bD$, as in
 \eqref{e:constant family}. We then identify
$$\on{Glue}(\bD,\,i\in I)\simeq \on{Funct}(I,\bD),$$
and the assertion becomes equivalent to the usual expression of colimits
along $I$ via its nerve. 

\end{proof} 

\ssec{Full faithfulness as homological contractibility}      \label{ss:ff Dmod via contr}

We return to the situation of \secref{ss:setting Dmod} (with a slightly simplified notation). Let
$\CY=Y\in \Sch$ be a base scheme and
$$i\mapsto (Z_i\overset{f_i}\longrightarrow Y), \quad (i\overset{\alpha}\longrightarrow j)\mapsto
(Z_j\overset{f_\alpha}\longrightarrow Z_i)$$
an $I^{\on{op}}$-diagram of schemes over it. Let 
$\CM_i\subset Z_i$ be closed subschemes such that for every $i\overset{\alpha}\longrightarrow j$
we have
$$(f_\alpha)^{-1}(\CM_i)\subset \CM_j,$$
i.e., we have the diagrams
$$
\CD
(f_\alpha)^{-1}(\CM_i)   @>>> \CM_j \\
@V{f_\alpha}VV  \\
\CM_i.
\endCD
$$

\sssec{}

We consider the category 
$$\on{Glue}(\Dmod(\CM_i),\,i\in I).$$

For every $i$, we let $\sF_i:\Dmod(Y)\to \Dmod(\CM_i)$ be the functor
$$\Dmod(Y) \overset{f_i^{\dR,!}}\longrightarrow \Dmod(Z_i) \to \Dmod(\CM_i),$$
where the second arrow is the !-pullback along the embedding $\CM_i\hookrightarrow Z_i$.

\medskip

The functors $\sF_i$ give rise to a functor
$$\sF:\Dmod(Y) \to \on{Glue}(\Dmod(\CM_i),\,i\in I),$$
and we want to give an explicit criterion for full faithfulness of $\sF$. 

\sssec{}

Let us assume that all $Z_i$ are proper over $Y$. 

\medskip

In this situation, each of the functors $\sF_i$ admits a left adjoint, and we find ourselves in the setting of \secref{ss:expl left adjoint}. Hence the functor $\sF$
admits a left adjoint given by \eqref{e:candidate for left adj}. 

\medskip

Denote the left adjoint of $\sF$ by 
\[\sF^L:\on{Glue}(\Dmod(\CM_i),\,i\in I)\to\Dmod(Y).\]
The functor $\sF$ is fully faithful if and only if the co-unit of the adjunction
$$\sF^L\circ F\to \on{Id}_{\Dmod(Y)}$$
is an isomorphism. 

\medskip

Let us describe the functor $\sF^L\circ F$ explicitly. 

\sssec{}  \label{sss:glued prestack}

Consider the following prestack over $Y$, denoted $\CM_{\on{Glued}}$:

\medskip

The prestack is the colimit over the category $\on{String}(I^{\on{op}})$ of the functor
$$\on{String}(I^{\on{op}})\to \on{PreStk}, \quad (i_0\to i_1\to\dots\to i_n) \mapsto \CZ_{i_0}\underset{\CZ_{i_n}}\times \CM_{i_n}.$$
(Note that the categories $\on{String}(I^{\on{op}})$ and $\on{String}(I)$ are naturally equivalent.)
Denote by $f_{\on{Glued}}$ the natural map
$$\CM_{\on{Glued}}\to Y.$$

\medskip

Note that $\CM_{\on{Glued}}$ is by definition pseudo-proper over $Y$. By the results of 
\secref{sss:pseudo-proper dir image}, the functor $(f_{\on{Glued}})^{\dr,!}$ admits a left
adjoint 
$$(f_{\on{Glued}})_{\dr,!}:\Dmod(\CM_{\on{Glued}})\to \Dmod(Y).$$

\sssec{}

From \propref{p:left adj as geom real} we obtain:

\begin{cor}
There is a canonical isomorphism of endofunctors of $\Dmod(Y)$ \emph{over} $\on{Id}_{\Dmod(Y)}$:
$$\sF^L\circ \sF \simeq (f_{\on{Glued}})_{\dr,!}\circ (f_{\on{Glued}})^{\dr,!}.$$
\end{cor}

Hence, we obtain:

\begin{cor} \label{c:map contr}
The functor $\sF$ is fully faithful if and only if the map $f_{\on{Glued}}$ is \emph{homologically contractible}, that is, if the 
functor $(f_{\on{Glued}})^{\dr,!}$ is fully faithful.
\end{cor}

\sssec{}

Let $k'\supset k$ be a field extension and let $y$ be a $k'$-point of $Y$.  Let $\CM_{\on{Glued},y}$ be the fiber of
$\CM_{\on{Glued}}$ over $y$, that is,
$$\CM_{\on{Glued},y}=\Spec(k')\underset{y,\CY}\times \CM_{\on{Glued}}.$$

Explicitly,
$$\CM_{\on{Glued},y}\simeq \underset{(i_0\to i_1\to\dots\to i_n)\in \on{String}(I^{\on{op}})}{\on{colim}}\,\,
\Spec(k')\underset{y,\CY}\times \left(\CZ_{i_0}\underset{\CZ_{i_n}}\times \CM_{i_n}\right).$$

\medskip

Combining \corref{c:map contr} and \lemref{l:pointwise}, we obtain:

\begin{cor}  \label{c:pointwise bis}
The functor 
$$\sF: \Dmod(\CY) \to \on{Glue}(\Dmod(\CM_i),\,i\in I)$$ is fully faithful if and only if for every field extension $k'\supset k$ and every $k'$-point $y$ of $Y$, 
the prestack $\CM_{\on{Glued},y}$ is 
homologically contractible over $k'$; that is, the map
$$\on{C}_*(\CM_{\on{Glued},y})\to k'$$
is an isomorphism. 
\end{cor}

\newpage 

\centerline{\bf Part III: Springer fibers.}

\bigskip

\section{Reduction to a homological contractibility statement}  \label{s:red to contr}

The goal of the remainder of the paper is to prove \thmref{t:Dmod ff LocSys} and thereby
finish the proof of \thmref{t:main}. Recall that we work in the framework of usual (non-derived) algebraic geometry,
which suffices for the study of D-modules. In other words, all (DG) schemes/stacks are replaced by the corresponding
classical subschemes/substacks.

\ssec{What do we need to show?}

\sssec{}
Recall the statement of \thmref{t:Dmod ff LocSys}. For any $P\in\on{Par}(G)$, we consider the stack of $P$-local 
systems $\LocSys_P$. When $P=G$, we take the global nilpotent cone
\[\on{Nilp}_{\on{glob}}\subset \Sing(\LocSys_G).\]
For every $P\in\on{Par}(G)$, we put
\[\CZ_P:=\LocSys_P\underset{\LocSys_G}\times \Sing(\LocSys_G),\]
and let 
$$\CM_P\subset \LocSys_P\underset{\LocSys_G}\times \Sing(\LocSys_G)$$
be the preimage of $\{0\}\subset \Sing(\LocSys_P)$ under the singular codifferential map
$$\CZ_P=\LocSys_P\underset{\LocSys_G}\times \Sing(\LocSys_G)\to \Sing(\LocSys_P).$$

\medskip

\thmref{t:Dmod ff LocSys} is the statement that the natural functor
$$\Dmod\left(\BP(\on{Nilp}_{\on{glob}})\right)\to 
\on{Glue}\left(\Dmod\left(\BP(\CM_P)\right),\,P\in \on{Par}(G)^{\on{op}}\right)$$
is fully faithful.

\sssec{} According to \corref{c:pointwise bis}, \thmref{t:Dmod ff LocSys} is equivalent to homological 
contractibility of the following prestacks. Let $k'\supset k$ be a field extension, and let $y$ be a $k'$-point
of $\BP(\on{Nilp}_{\on{glob}})$. Construct the prestack $\CM_{\on{Glued},y}$ as follows.

\medskip

Consider the category $\on{String}(\on{Par}(G))$. 
By definition, its objects are chains of standard 
parabolic subgroups
\[(P_0\subset P_1\subset\dots\subset P_n)\quad (n\ge 0, P_i\in\on{Par}(G)),\]
and morphisms are induced by order-preserving maps $[m]\to [n]$. Now consider the functor
$\on{String}(\on{Par}(G))\to\Sch$ given by
\[(P_0\subset P_1\subset\dots\subset P_n)\mapsto\Spec(k')
\underset{y,\BP(\on{Nilp}_{\on{glob}})}
\times(\CZ_{P_0}\underset{\CZ_{P_n}}\times\BP(\CM_{P_n})),\]
and put
\[\CM_{\on{Glued},y}=\colim_{(P_0\subset P_1\subset\dots\subset P_n)\in\on{Strings}(\on{Par}(G))}\Spec(k')
\underset{y,\BP(\on{Nilp}_{\on{glob}})}
\times(\CZ_{P_0}\underset{\CZ_{P_n}}\times\BP(\CM_{P_n})).\]

\medskip 

\thmref{t:Dmod ff LocSys} is equivalent to homological contractibility of prestacks $\CM_{\on{Glued},y}$
for every $k'$ and $y$. 
Without loss of generality, we can replace $k$ with its extension $k'$. Therefore, we need to verify
that $\CM_{\on{Glued},y}$ is homologically contractible for every $k$-point $y$ of $\BP(\on{Nilp}_{\on{glob}})$.

\sssec{}  \label{sss:glob cone}
Let us now restate the above condition in explicit terms. 
First, recall the description of $k$-points of the algebraic stack $\Sing(\LocSys_G)$
and of the substack
$$\on{Nilp}_{\on{glob}}\subset \Sing(\LocSys_G),$$
see \cite[Sect. 11.1]{AG}.

\medskip

Namely, this groupoid of $k$-points $\Sing(\LocSys_G)(k)$ 
consists of pairs $(\sigma,A)$, where $\sigma$ is a $G$-local system on $X$, and $A$ is
a horizontal section of the vector bundle $\fg^*_\sigma$ associated with the co-adjoint representation. 
We  identify $\fg^*$ with $\fg$ by means of a $G$-invariant bilinear form, and thus think of $A$ as a horizontal
section of $\fg_\sigma$.  

\medskip

The sub-groupoid of $k$-points $\on{Nilp}_{\on{glob}}(k)$ corresponds to pairs $(\sigma,A)$ with \emph{nilpotent} $A$.

\sssec{}

Given a $k$-point $(\sigma,A)$ of $\on{Nilp}_{\on{glob}}$ and a standard parabolic $P\in \Par(G)$, 
we define schemes
$$\Spr^{\sigma,A}_{P,\on{unip}} \subset \Spr^{\sigma,A}_P \subset \Spr^\sigma_P$$
as follows.

\medskip

$\Spr^\sigma_P$ is the scheme of reductions of $\sigma$ (as a local system)
to the parabolic $P$, and $\Spr^{\sigma,A}_{P,\on{unip}}$ and $\Spr^{\sigma,A}_P$ are its subschemes 
corresponding to the condition that $A$ be a section of 
$$\fu(P)_\sigma \subset \fg_\sigma \text{ or } \fp_\sigma \subset \fg_\sigma,$$
respectively, where $\fu(P)$ denotes the Lie algebra of the unipotent radical $U(P)$ of $P$. 

\sssec{}  \label{sss:identify with Springer}

For fixed $\sigma\in\LocSys_G(k)$ as above, the diagram
$$P\rightsquigarrow \Spr^\sigma_P$$
identifies with the diagram of schemes 
$$P\rightsquigarrow \LocSys_P\underset{\LocSys_G}\times \{\sigma\}.$$

\medskip

For fixed $(\sigma,A)\in \on{Nilp}_{\on{glob}}(k)$, the diagram
$$P \rightsquigarrow \Spr^{\sigma,A}_{P,\on{unip}}$$
identifies with the diagram of schemes 
$$P\rightsquigarrow \CM_P \underset{\on{Nilp}_{\on{glob}}}\times  \{(\sigma,A)\},$$
where 
$$\CM_P\subset \LocSys_P\underset{\LocSys_G}\times \Sing(\LocSys_G)$$
is as in \thmref{t:Dmod ff LocSys}. 

\sssec{}

Note that $$P \rightsquigarrow \Spr^{\sigma,A}_{P}$$
is a diagram in the usual sense: for any pair of standard parabolics $P_1\subset P_2$,
there is a morphism between the corresponding schemes
\[\Spr^{\sigma,A}_{P_1}\to\Spr^{\sigma,A}_{P_2}.\]
On the other hand, in the diagram 
\[P \rightsquigarrow \Spr^{\sigma,A}_{P,\on{unip}},\]
the schemes
\[\Spr^{\sigma,A}_{P_1,\on{unip}}\quad\text{and}\quad\Spr^{\sigma,A}_{P_2,\on{unip}}\qquad(P_1\subset P_2)\]
are connected by a correspondence:
$$
\CD
\Spr^{\sigma,A}_{P_1}\underset{\Spr^{\sigma,A}_{P_2}}\times \Spr^{\sigma,A}_{P_2,\on{unip}}@>>>  \Spr^{\sigma,A}_{P_1,\on{unip}}  \\
@VVV  \\
\Spr^{\sigma,A}_{P_2,\on{unip}}.
\endCD
$$

\sssec{}

Explicitly, in the inclusion 
$$\Spr^\sigma_{P_1}\underset{\Spr^\sigma_{P_2}}\times \Spr^{\sigma,A}_{P_2,\on{unip}}\subset
\Spr^{\sigma,A}_{P_1,\on{unip}},$$
the left-hand side (resp. the right-hand side) parametrizes reductions of the local system $\sigma$ 
to the parabolic $P_1$ such that $A$ is a section of 
$$\fu(P_2)_\sigma \subset \fg_\sigma \qquad\text{(resp. of } \fu(P_1)_\sigma \subset \fg_\sigma \text{)}.$$

\medskip

Let us now form the prestack
$$\Spr_{\on{Glued,unip}}^{\sigma,A}:=
\colim_{(P_0\subset P_1\subset\dots\subset P_n)\in
\on{Strings}(\on{Par}(G))}\Spr^\sigma_{P_0}\underset{\Spr^\sigma_{P_n}}\times \Spr^{\sigma,A}_{P_n,\on{unip}}.$$

\medskip

Provided $A\ne 0$, a pair $(\sigma,A)\in (\sigma,A)\in \on{Nilp}_{\on{glob}}(k)$ projects to a $k$-point
$y$ of $\BP(\on{Nilp}_{\on{glob}})$, and $\Spr_{\on{Glued,unip}}^{\sigma,A}$ identifies with the prestack
$\CM_{\on{Glued},y}$. We therefore see that \thmref{t:Dmod ff LocSys}  
is implied by the following assertion:

\begin{thm} \label{t:contr1}
Let $(\sigma,A)$ be a $k$-point of $\on{Nilp}_{\on{glob}}$. Then the prestack 
$\Spr_{\on{Glued,unip}}^{\sigma,A}$ is homologically contractible, that is, the trace map
$$\on{C}_*(\Spr_{\on{Glued,unip}}^{\sigma,A})\to k$$
is an isomorphism. 
\end{thm} 

The rest of the paper is devoted to the proof of \thmref{t:contr1}. 

\begin{rem} Note that \thmref{t:contr1} claims, in particular, 
that for any such $(\sigma, A)$, $\Spr_{\on{Glued,unip}}^{\sigma,A}$ is non-empty; this amounts to 
checking that $\Spr^{\sigma,A}_{P,\on{unip}}\ne\emptyset$ for some $P\in\on{Par}(G)$. 
This easily follows from the Jacobson-Morozov Theorem, see \secref{sss:JM}.
\end{rem}

\begin{rem}
Note that in \thmref{t:contr1} we allow $A=0$. The case $A=0$ is not needed to deduce
\thmref{t:Dmod ff LocSys}, but it is used in the inductive step in the proof of \thmref{t:contr1}. 
Note, however, that the case $A=0$ in \thmref{t:contr1} is reasonably easy: 

\medskip

It is not hard to check (see Remark \ref{r:more complicated} below) that for $A=0$, the prestack $\Spr_{\on{Glued,unip}}^{\sigma,A}$ 
identifies with $$\Spr_{\on{Glued}}^\sigma:=\underset{P\in \Par(G)}{\on{colim}}\, \Spr^\sigma_P.$$

\medskip

Now, the category $\Par(G)$ has a final object (with $P=G$), and 
$\Spr^\sigma_G=\on{pt}$. From here, $\Spr_{\on{Glued}}^\sigma\simeq \on{pt}$.
\end{rem}

\ssec{Reduction to another contractibility statement}

One difficulty with \thmref{t:contr1} is due to a rather complicated colimit used to define
the prestack $\Spr_{\on{Glued,unip}}^{\sigma,A}$.
We shall now replace  \thmref{t:contr1} by an
equivalent statement, namely \thmref{t:contr2}, which is simpler from the combinatorial point of view. 

\sssec{}

Denote by $\Par'(G)\subset\Par(G)$ the subset of \emph{proper} parabolics; thus
\[\Par(G)=\Par'(G)\sqcup\{G\}.\]

Consider the assignment
$$P\rightsquigarrow \Spr^{\sigma,A}_P$$
as a functor 
$$\Par'(G)\to \{\text{Schemes}\}.$$

Set
$$\Spr_{\on{Glued}}^{\sigma,A}:=\underset{P\in \Par'(G)}{\on{colim}}\, \Spr^{\sigma,A}_P.$$

\begin{rem} \label{r:more complicated}
The stack $\Spr_{\on{Glued}}^{\sigma,A}$ is also equal to the (more complicated) 
colimit  over $\on{String}(\Par'(G))$ of the functor 
$$(P_0\subset P_1\subset\dots\subset P_n)\mapsto  \Spr^{\sigma,A}_{P_0}.$$
\end{rem} 

\sssec{}

In Sects. \ref{ss:proof of another version 1} and \ref{ss:proof of another version 2}, we  prove:

\begin{prop}  \label{p:another version}
Assume the validity of \thmref{t:contr1} for all \emph{proper} Levi subgroups of $G$. Then for $A\neq 0$ there
exists a naturally defined isomorphism
$$\on{C}_*(\Spr_{\on{Glued,unip}}^{\sigma,A}) \simeq \on{C}_*(\Spr_{\on{Glued}}^{\sigma,A}).$$
\end{prop}

Assuming \propref{p:another version}, we obtain that \thmref{t:contr1} is equivalent to the 
following:

\begin{thm} \label{t:contr2}
Let $(\sigma,A)$ be a $k$-point of $\on{Nilp}_{\on{glob}}$ with $A\neq 0$. Then the prestack 
$\Spr_{\on{Glued}}^{\sigma,A}$ is homologically contractible.
\end{thm} 

\medskip

We  prove \thmref{t:contr2} in \secref{s:Schubert}. In \secref{s:Springer} we  give an alternative 
proof of \thmref{t:contr2} in the special case when $\sigma$ is the trivial local system. 

\sssec{}  \label{sss:top}

Both Theorems \ref{t:contr1} and \ref{t:contr2} have topological counterparts. Let us sketch these counterparts 
in case the reader finds their statement more transparent; they are \emph{not} logically necessary for the proof. 

\medskip

Let $A$ be a 
nilpotent element of $\fg$, but instead of a local system $\sigma$ fix a family $\{g_\alpha\}$ of 
elements in $G$ that centralize $A$. 

\medskip

For every $P\in \Par(G)$ consider the corresponding partial flag variety $\Fl_P$; we  think of it as
the scheme classifying parabolics $P'$ in the conjugacy class of $P$. Let
$$\Spr_{P,\on{unip}}^{\{g_\alpha\},A}\subset \Spr_P^{\{g_\alpha\},A}\subset \Spr_P^{\{g_\alpha\}}$$
be the closed subschemes of $\Fl_P$ that correspond to $P'\in \Fl_P$ that satisfy the conditions
$$(g_\alpha\in P',A\in \fu(P')),\,\, (g_\alpha\in P',A\in \fp'),\,\, (g_\alpha\in P'),$$
respectively.

\medskip

We can form the prestacks $\Spr_{\on{Glued,unip}}^{\sigma,A}$ and $\Spr_{\on{Glued}}^{\sigma,A}$,
and the assertions parallel to Theorems~\ref{t:contr1} and \ref{t:contr2}  hold in this context as well. We
leave it to the reader to verify that the argument of this paper can be used to prove these
topological counterparts of the theorems.

\medskip

Note that when $k=\BC$, Theorems \ref{t:contr1} and \ref{t:contr2}, as stated above, follow from their
topological counterparts via the Riemann-Hilbert correspondence. 

\medskip

Namely, fix a base point $x\in X$, and trivialize the 
fiber of $\sigma$ at $x$. Then the monodromy of $\sigma$ gives a homomorphism
$\pi_1(X,x)\to G$, and we take $\{g_\alpha\}$ to be the images in $G$ of some set of generators 
of $\pi_1(X,x)$. Then the analytic spaces corresponding to the schemes $\Spr_{P,\on{unip}}^{\sigma,A}$ and
$\Spr_{P,\on{unip}}^{\{g_\alpha\},A}$  (resp., $\Spr_P^{\sigma,A}$ and
$\Spr_P^{\{g_\alpha\},A}$) are canonically identified.

\sssec{}

Let us consider some examples of \thmref{t:contr2}.

\medskip

First, we consider the case of $G=SL_2$, in which case \thmref{t:main} is already non-obvious. 
But all of its complexity is contained in the reduction of \thmref{t:main} to  \thmref{t:contr2}, as the latter is quite
easy: 

\medskip 

For $G$ of rank $1$, the poset $\Par'(G)$ consists of one element, namely, $P=B$. Since $A\neq 0$, the scheme 
$\Spr^{\sigma,A}_B$ is a `fat point': it is a nilpotent thickening of a point. 
Hence, it is homologically contractible.

\sssec{}

Consider now the case of $G=SL_3$. We  distinguish two cases: (a) when $A$ is a regular nilpotent; (b) when
$A$ is a sub-regular nilpotent. 

\medskip

In case (a), for all three parabolics, the corresponding schemes $\Spr^{\sigma,A}_P$
are again fat points. So, the contractibility follows from the fact that the poset $\Par'(G)$ 
$$P_1\supset B\subset P_2$$
is contractible as a category (it has an initial object, namely $B$). 

\medskip

Case (b) is more interesting. The scheme $\Spr^{\sigma,A}_B$ has the shape
$$Z_1\underset{\on{pt}}\sqcup\, Z_2,$$
i.e., its obtained by joining certain subschemes $Z_1$ and $Z_2$ along a common point. 
(To see this, use the topological version described in \secref{sss:top}, first with $\{g_\alpha\}$
being trivial, and then deduce the general case.)

\medskip

The projection
$$\Spr^{\sigma,A}_B\to \Spr^{\sigma,A}_{P_1}$$
maps $Z_1$ isomorphically onto its image, and it collapses $Z_2$ onto the image of $\on{pt}=Z_1\cap Z_2$.

\medskip

Similarly, the projection
$$\Spr^{\sigma,A}_B\to \Spr^{\sigma,A}_{P_1}$$
maps $Z_2$ isomorphically onto its image, and it collapses $Z_1$ onto the image of $\on{pt}=Z_1\cap Z_2$.

\medskip

This description makes the statement of \thmref{t:contr2} manifest. 

\ssec{Proof of \propref{p:another version}, Step 1}   \label{ss:proof of another version 1}

\sssec{}   \label{sss:Spr glue}

Recall that the prestack $\Spr_{\on{Glued,unip}}^{\sigma,A}$ is the the following colimit
over the index category $\on{Strings}(\on{Par}(G))$
of chains of standard parabolic subgroups
\[(P_0\subset P_1\subset\dots\subset P_n)\] 
and morphisms are given by order-preserving maps $[m]\to [n]$. 

\medskip

To each $(P_0\subset P_1\subset\dots\subset P_n)\in\on{Strings}(\on{Par}(G))$ we attach the scheme 
$$\Spr_{{P_0}}^\sigma\underset{\Spr_{{P_n}}^\sigma}\times\Spr_{P_n,\on{unip}}^{\sigma,A} .$$

\medskip

In the above diagram, the $P_i$'s are all standard parabolics. It is possible that $P_n=G$, but in this case
$\Spr_{P_n,\on{unip}}^{\sigma,A}$ is empty, because $A\neq 0$. Thus 
we can work with chains of proper standard parabolic subgroups \[(P_0,\dots,P_n)\in\on{String}(\Par'(G)).\] 

\medskip

Let $I_1:=\on{String}(\Par'(G))$ be the index category of chains of proper standard parabolic subgroups.
Denote by $\sF_1:I_1\to\Sch$ the functor
\[(P_0,\dots,P_n)\mapsto\Spr_{{P_0}}^\sigma\underset{\Spr_{{P_n}}^\sigma}\times\Spr_{P_n,\on{unip}}^{\sigma,A}.\]
Thus
\[\Spr_{\on{Glued,unip}}^{\sigma,A}\simeq \colim_{i\in I_1}\sF_1(i).\]

\sssec{}

We recall that by definition,
\[\Spr_{\on{Glued}}^{\sigma,A}=\colim_{i\in I_2}\sF_2(i),\]
where we put
$I_2:=\Par'(G)$ and
$$\sF_2:I_2\to \Sch:P\mapsto \Spr_P^{\sigma,A}.$$

\sssec{}

Consider now the category $I$ whose objects are collections
\begin{equation} \label{e:index comb}
(P_0\subset\dots\subset P_n\subset P)\quad(n\ge 0;P_0,\dots,P_n,P\in\Par'(G)).
\end{equation} 
A morphism 
$$(P^1_0\subset\dots\subset P^1_{n^1}\subset P^1)\to (P^2_0\subset\dots\subset P^2_{n^2}\subset P^2)$$
is specified by an order-preserving map $[n^2]\to [n^1]$ and an inclusion $P^1\subset P^2$. 

\medskip

Define a functor $F:I\to\Sch$ by
\[F:(P_0\subset\dots\subset P_n\subset P)\mapsto
\Spr_{P_n,\on{unip}}^{\sigma,A}\underset{\Spr_{P_n}^\sigma}\times \Spr_{P_0}^\sigma,\]
and put
$$\Spr_{\on{Glued,mixed}}^{\sigma,A}:=\colim_{i\in I}\sF(i).$$

\sssec{}

We have canonical forgetful functors 
$$I_1\overset{\phi_1}\longleftarrow I \overset{\phi_2}\longrightarrow I_2.$$

By construction $\sF\simeq \sF_1\circ \phi_1$. This isomorphism gives rise to a map
\begin{equation} \label{e:mixed unip}
\Spr_{\on{Glued,mixed}}^{\sigma,A}\to \Spr_{\on{Glued,unip}}^{\sigma,A}.
\end{equation}

We claim that the map \eqref{e:mixed unip} is an isomorphism of prestacks. Indeed, this
follows from the fact that the functor $\phi_1$ is a co-Cartesian fibration with contractible
fibers (each fiber has an initial object, namely $P=P_n$). 

\medskip

Thus, to prove \propref{p:another version}, we need to construct a homological equivalence between
prestacks $\Spr_{\on{Glued,mixed}}^{\sigma,A}$ and $\Spr_{\on{Glued}}^{\sigma,A}$. 

\sssec{}

Note now that we have a canonically defined natural transformation
\begin{equation} \label{e:F to F2}
\sF\to \sF_2\circ \phi_2.
\end{equation}
Indeed, for any $(P_0\subset\dots\subset P_n\subset P)\in I$, we have a natural map
\begin{multline*}
\sF(P_0\subset\dots\subset P_n\subset P)= 
\Spr_{P_n,\on{unip}}^{\sigma,A}\underset{\Spr_{P_n}^\sigma}\times \Spr_{P_0}^\sigma\hookrightarrow
\Spr_{P_0}^\sigma\to\\
\Spr_P^\sigma=\sF_2(P)=\sF_2\circ\phi_2(P_0\subset\dots\subset P_n\subset P).
\end{multline*}

Hence, we obtain a map of prestacks 
\begin{equation} \label{e:mixed}
\Spr_{\on{Glued,mixed}}^{\sigma,A}\to \Spr_{\on{Glued}}^{\sigma,A}.
\end{equation}

Let us prove that the map
\begin{equation} \label{e:mixed homology}
\on{C}_*(\Spr_{\on{Glued,mixed}}^{\sigma,A})\to \on{C}_*(\Spr_{\on{Glued}}^{\sigma,A}),
\end{equation}
induced by \eqref{e:mixed} is an isomorphism. 

\sssec{}

Let 
$$\sF'_2:I_2\to \Sch$$ 
denote the left Kan extension of the functor $\sF$ along $\phi_2$. By adjunction, the natural transformation
\eqref{e:F to F2} gives rise to a natural transformation 
$$\sF'_2\to \sF_2.$$

\medskip

Composing with the functor
$$\on{C}_*:\Sch\to \Vect,$$
we obtain a natural transformation
\begin{equation} \label{e:F to F2 again}
\on{C}_*\circ \sF'_2\to \on{C}_*\circ \sF_2
\end{equation} 
of functors $I_2\to \Vect$.

\medskip

The map \eqref{e:mixed homology} is obtained from \eqref{e:F to F2 again} by taking colimits over $I_2$.
Thus, in order to prove that  \eqref{e:mixed homology} is an isomorphism, it suffices to show that the map
\eqref{e:F to F2 again} is an isomorphism of functors $I_2\to \Vect$. 

\medskip

The latter will be done in Step 2, using \thmref{t:contr1} for
proper Levi subgroups of $G$ (including the case $A=0$). 

\ssec{Proof of \propref{p:another version}, Step 2}   \label{ss:proof of another version 2}

\sssec{}

Note that the functor $\phi_2$ is also a co-Cartesian fibration. Hence, the value of $\on{C}_*\circ \sF'_2$ on an object
$P\in \Par'(G)=I_2$
is computed as the colimit of the functor $\on{C}_*\circ \sF_2$ over the fiber of $\phi_2$ over $P$. I.e., it is
the homology of the prestack equal to the colimit of the restriction of $\sF$ to the above fiber. Denote this
prestack by $\Spr_{\on{Glued,mixed},P}^{\sigma,A}$. 

\medskip

Note that we have a tautologically defined map
$$f:\Spr_{\on{Glued,mixed},P}^{\sigma,A}\to \Spr_{P}^{\sigma,A}.$$

We need to show that the above map $f$ induces an isomorphism on homology. It suffices to check that
the trace map
\begin{equation} \label{e:rel trace}
f_{\dr,!}(\omega_{\Spr_{\on{Glued,mixed},P}^{\sigma,A}})\to \omega_{\Spr_P^{\sigma,A}}
\end{equation}
is an isomorphism in $\Dmod(\Spr_P^{\sigma,A})$. 

\sssec{}

The fact that \eqref{e:rel trace} is an isomorphism can be checked at the level of !-fibers at $k$-points
of $\Spr_{P}^{\sigma,A}$.  

\medskip

Fix a point $\sigma_P\in\Spr_{P}^{\sigma,A}(k)$. Thus, $\sigma_P$ is a reduction of $\sigma$ to $P$ that is
compatible with $A$. Let $M$ be the
Levi quotient of $P$, and let $(\sigma_M,A_M)$ be the resulting $k$-point of $\on{Nilp}_{\on{glob}}$ for
the group $M$. (Note that $A_M$ may be zero.)

\medskip

Note that $\Spr_{\on{Glued,mixed},P}^{\sigma,A}$ is a colimit of schemes each of which is proper,
and in particular, maps properly to $\Spr_P^{\sigma,A}$. Hence, by proper base change, the !-fiber of 
$f_{\dr,!}(\omega_{\Spr_{\on{Glued,mixed},P}^{\sigma,A}})$ at $\sigma_P$ 
is isomorphic to the homology of the fiber of $\Spr_{\on{Glued,mixed},P}$ over $\sigma_P$; denote
this fiber by  $\Spr^{\sigma,A}_{\on{Glued,mixed},P,\sigma_P}$. 

\sssec{}

Thus, we have to show that the trace map
$$\on{C}_*(\Spr^{\sigma,A}_{\on{Glued,mixed},P,\sigma_P})\to k$$
is an isomorphism. 

\medskip

However, we notice that there is a canonical isomorphism
$$\Spr^{\sigma,A}_{\on{Glued,mixed},P,\sigma_P}\simeq \Spr_{\on{Glued,unip}}^{\sigma_M,A_M},$$
(the latter prestack taken for the reductive group $M$). 

\medskip 

Hence, the required assertion follows from \thmref{t:contr1}, applied to $M$. 

\section{Schubert stratification} \label{s:Schubert}

The goal of this section is to prove \thmref{t:contr2}.

\ssec{Conventions regarding roots}

\sssec{}
Recall that we fixed a Borel subgroup $B\subset G$ and a maximal torus $T\subset B$. 
Let $\Lambda:=\Hom(T,\BG_m)$ be the character lattice of $T$; it is a free abelian group, which we write additively. 
The standard parabolics $P\subset G$ are the parabolic subgroups containing $B$.

\medskip

Let $\ft\subset\fb\subset\fg$ be the Lie algebras of $T$, $B$, and $G$ respectively. For every $\alpha\in\Lambda$, we 
denote by $\fg_\alpha\subset\fg$ the 
corresponding root subspace; in particular, $\fg_0=\ft$. Let
\[\sR=\{\alpha\in\Lambda-\{0\}:\fg_\alpha\ne 0\}\]
be the set of roots.
Denote by $\sS\subset\sR^+\subset\sR$ the subsets of simple and positive roots with respect to $B$. Thus, 
\[\fb=\ft\oplus\bigoplus_{\alpha\in\sR^+}\fg_\alpha.\]
We identify $\sS$ and the set of the vertices of the Dynkin diagram of $G$.

\sssec{}
We  think of $\Par(G)$ as the poset of subsets $\sJ\subset\sS$ (ordered by inclusion) via
$\sJ\mapsto P_\sJ$. Explicitly,
given $\sJ\in\Par(G)$, the Lie subalgebras
\begin{align*}
\fp_\sJ&:=\bigoplus\{\fg_\alpha:\alpha\in\sR^+\cup\Span(\sJ)\}\\
\fm_\sJ&:=\bigoplus\{\fg_\alpha:\alpha\in\Span(\sJ)\}\\
\fu(P_\sJ)&:=\bigoplus\{\fg_\alpha:\alpha\in\sR^+-\Span(\sJ)\}
\end{align*}
correspond to $P_\sJ$, the standard Levi subgroup $M_\sJ\subset P_\sJ$, and the unipotent radical $U(P_\sJ)$ of $P_\sJ$, 
respectively. We denote by
\[\sR_\sJ:=\sR\cap\Span(\sJ)\] the set of roots of $M_\sJ$, so that $\sJ\subset\sR_\sJ$ is the set of simple roots. 



\sssec{} 
Let $N(T)\subset G$ be the normalizer of $T$. 
The Weyl group $\sW=N(T)/T$ acts on $\Lambda$ preserving $\sR$.  For any $\sJ\in\Par(G)$, denote by $\sW_\sJ\subset\sW$ the 
subgroup generated by the reflections around the roots in $\sJ$. Thus, $\sW_\sJ$ is the Weyl group of $M_\sJ$.

\sssec{} Given $\sJ\in\Par(G)$, we denote by 
\[\Fl_\sJ=\{P'\subset G:P'\text{ is conjugate to }P_\sJ\}\] 
the flag variety of parabolic subgroups of type $\sJ$. We have a natural isomorphism $\Fl_\sJ=G/P_{\sJ}$.
If $\sJ=\emptyset$, then $P_\sJ=B$, and we write simply 
\[\Fl=\Fl_\sJ=G/B\qquad(\sJ=\emptyset)\]
for the complete flag variety.

\medskip

Whenever $\tilde\sJ\subset\sJ$ in $\Par(G)$, we have a natural morphism
\[f=f_{\tilde\sJ,\sJ}:\Fl_{\tilde\sJ}\to\Fl_\sJ.\]

\sssec{}
Given two Borel subgroups $B',B''\subset G$, we denote their relative position by $w(B',B'')\in\sW$. 
Explicitly, for $B''=B$ being the fixed Borel, the equality $w=w(B',B)$ means that 
\begin{align*}
B'&=\Ad_g(B), \quad g\in BwB.
\end{align*}
We then expand $w$ to arbitrary pairs $(B',B'')\in\Fl\times \Fl$ by $G$-invariance. 

\medskip

More generally, suppose $\sJ_0,\sJ\in\Par(G)$. The relative position of two parabolic subgroups $P'\in\Fl_{\sJ_0}$, $P''\in\Fl_\sJ$
is given by the double coset
\[\{w(B',B'')\in\sW:B'\subset P'\text{ and }B''\subset P''\text{ are Borel subgroups}\}\in \sW_{\sJ_0}  \backslash \sW/\sW_\sJ.\]
This double coset contains a unique minimal element with respect to the Bruhat order on $\sW$; we denote it by
$w(P',P'')\in\sW$.
The condition that $w\in\sW$ is minimal in its double coset $\sW_{\sJ_0}w\sW_\sJ$ is equivalent to the condition
\[w(\sJ)\subset\sR^+\qquad\text{and}\qquad w^{-1}(\sJ_0)\subset\sR^+.\] 

\ssec{Some Weyl group combinatorics}\label{ss:Weylcombo}

In this subsection we fix $\sJ_0\in\Par(G)$ and the corresponding standard
parabolic subgroup $P_0:=P_{\sJ_0}$. 

\sssec{}

Put 
\begin{equation}\label{eq:W'}
\sW':=\{w\in\sW:w^{-1}(\sJ_0)\subset\sR^+\}=\{w\in\sW:w\text{ is minimal in }\sW_{\sJ_0}w\}.
\end{equation}
There is a unique maximal element $w'_0\in\sW'$; it is characterized by the property that 
\[w'_0(\sR^+)\cap\sR^+=\sR_{\sJ_0}\cap\sR^+.\]
Explicitly, $w'_0$ is the minimal element of the coset $\sW_{\sJ_0}w_0$, where $w_0\in\sW$ is the longest element; also, 
$w_0'w_0\in \sW_{\sJ_0}$ 
is the longest element of the Coxeter group $\sW_{\sJ_0}$.

\sssec{}  \label{sss:pos neg}
Fix $w\in\sW$, and consider the partition $\sS=\sS_w^0\cup\sS_w^+\cup\sS_w^-$ given by
\begin{align*}
\sS_w^0&:=\sS\cap w^{-1}(\sR_{\sJ_0})\\
\sS_w^+&:=\sS\cap w^{-1}(\sR^+\setminus\sR_{\sJ_0})\\
\sS_w^-&:=\sS\cap w^{-1}(-\sR^+\setminus\sR_{\sJ_0}).
\end{align*}
(For simplicity, the dependence of this partition on $\sJ_0$ is suppressed in the notation.)
The following properties of this partition are clear.

\begin{lem} Suppose $w\in\sW$. Then
\begin{enumerate} 
\item $\sS_w^-=\emptyset$ if and only if $w\in\sW_{\sJ_0}$, and
\item $\sS_w^+=\emptyset$ if and only if $w\in\sW_{\sJ_0}w_0$.\qed
\end{enumerate}
\end{lem}
\begin{cor} \label{co:partition of S}
Suppose $w\in\sW'$. Then
\begin{enumerate} 
\item \label{it:co:partition of S1} $\sS_w^-=\emptyset$ if and only if $w=e$, and
\item \label{it:co:partition of S2} $\sS_w^+=\emptyset$ if and only if $w=w_0'$. \qed
\end{enumerate}
\end{cor}

\sssec{} Let now $P'$ be another parabolic subgroup (not necessarily a standard one). 
Consider $w(P_0,P')\in\sW$. Clearly, $w(P_0,P')\in\sW'$. We need the following 
easy observation.

\begin{lem} \label{l:unip radical}
Let $U(P_0)\subset P_0$ be the unipotent radical, and let $\fp'$ and $\fu(P_0)$ be the Lie algebras of
$P'$ and $U(P_0)$, respectively. Then $w(P_0,P')=w_0'$ if and only if $\fp'\cap \fu(P_0)=\{e\}$.
\end{lem}
\begin{proof} Follows from Corollary~\ref{co:partition of S}\eqref{it:co:partition of S2}.
\end{proof}

\sssec{} Let us now fix $\sJ\in\Par(G)$, and consider the flag variety $\Fl_\sJ$. 
For $w\in\sW$, define the \emph{Schubert stratum} in $\Fl_\sJ$ as
follows:
\begin{align*}
\Fl_\sJ^w&:=\{P'\in\Fl_\sJ:w(P_0,P')=w\}\subset\Fl_\sJ.\\
\intertext{Also, put}
\Fl_\sJ^{\le w}&:=\{P'\in\Fl_\sJ:w(P_0,P')\le w\}\subset\Fl_\sJ\\
\intertext{and}
\Fl_\sJ^{<w}&:=\{P'\in\Fl_\sJ:w(P_0,P')<w\}\subset\Fl_\sJ.
\end{align*}
(One again, we omit the parabolic subgroup $P_0$ from the notation.) 

\begin{rem} 
We emphasize that in the definition of $\Fl_\sJ^w$, the equality $w(P_0,P')=w$ takes place in $W$ and \emph{not} in
$W_{\sJ_0}\backslash W/W_\sJ$ (and similarly for $\Fl_\sJ^{\le w}$ and $\Fl_\sJ^{< w}$). 

Hence, if $w\not\in\sW'$, then $\Fl_\sJ^w=\emptyset$ and 
$\Fl_\sJ^{\le w}=\Fl_\sJ^{<w}$. Also, $\Fl_\sJ^{\le w'_0}=\Fl_\sJ$. 

\end{rem}

\sssec{}
Suppose $\tilde\sJ\subset \sJ$ in $\Par(G)$. Consider the natural map $f:\Fl_{\tilde\sJ}\to\Fl_\sJ$. 
Clearly, 
\[
f(\Fl_{\tilde\sJ}^{\le w})\subset\Fl_\sJ^{\le w}\quad\text{and}\quad
f(\Fl_{\tilde\sJ}^{<w})\subset\Fl_\sJ^{<w};
\]
however, it is not true in general that $f(\Fl_{\tilde\sJ}^w)\subset\Fl_\sJ^w$.

\begin{lem} \label{lm:sch} Fix $w\in\sW$ (and recall that $\sJ\in\Par(G)$ is also fixed).
\begin{enumerate}
\item\label{it:lm:sch1} If $\sJ\cap\sS^-_w\ne\emptyset$, then $\Fl_\sJ^w=\emptyset$.
\item\label{it:lm:sch2} Put $\tilde\sJ=\sJ\setminus \sS^+_w$. Then the map $f:\Fl_{\tilde\sJ}\to\Fl_\sJ$ induces an isomorphism 
$\Fl^w_{\tilde\sJ}\simeq\Fl_\sJ^w$.
\end{enumerate}
\end{lem}     
\begin{proof}
\eqref{it:lm:sch1} Indeed, if $\sJ\cap\sS^-_w\ne\emptyset$, then $w(\sJ)\not\subset\sR^+$ and $w$ is not the minimal 
element of $\sW_{\sJ_0}w\sW_\sJ$.

\eqref{it:lm:sch2} The inverse map sends $P'\in\Fl_\sJ$ to the parabolic subgroup $(P'\cap P_0)U(P')\subset P'$, 
where $U(P')\subset P'$ is the unipotent radical. 
\end{proof}      

\ssec{Proof of \thmref{t:contr2}: setting up the induction}

\sssec{}\label{sss:JM}

Recall that in \thmref{t:contr2} we fix a $G$-local system $\sigma$ and a \emph{non-zero} horizontal section $A$
of $\fg_\sigma$. 

\medskip

By the Jacobson-Morozov Theorem, $A$ determines a canonical reduction of $\sigma$ to a standard 
parabolic subgroup, which we denote $P_0$. Moreover, $A$ belongs to the nilradical
of this reduction, in the sense that $A$ lies in 
$\fu(P_0)_\sigma\subset\fg_\sigma$. (Here we abuse the notation slightly by writing $\sigma$ for the reduction to $P_0$.)
Equivalently, the reduction corresponds to a point of $\Spr^{\sigma, A}_{P_0,\on{unip}}$.
In particular, since $A\neq 0$, we have $\fu(P_0)\ne 0$ and hence $P_0\neq G$.

\begin{rem}
For most of the argument, we only need to know that $\sigma$ is reduced to a proper parabolic. The fact that
$A$ belongs to the nilradical of the reduction is used only in \secref{sss:largestisempty}.
\end{rem}

\sssec{}
Set $P_0=P_{\sJ_0}$; that is, $\sJ_0$ is the type of the standard parabolic $P_0$. Let us use the formalism of
\secref{ss:Weylcombo} for this choice of $\sJ_0$.

\medskip

Each of the schemes $\Spr^{\sigma,A}_P$ comprising $\Spr^{\sigma,A}_{\on{Glued}}$ acquires a 
stratification by the set $\sW'$, where $\sW'$ is given by \eqref{eq:W'}; denote the corresponding subschemes by 
$$\Spr^{\sigma,A,<w}_P \subset \Spr^{\sigma,A,\leq w}_P \supset \Spr^{\sigma,A,w}_P.$$
Explicitly, the stratification is determined by the relative position the reduction of $\sigma$ to $P$ 
(corresponding to a point of $\Spr^{\sigma,A}_P$) and the fixed reduction of $\sigma$ to $P_0$.

\medskip

Consider the corresponding prestacks
\begin{align*}
\Spr^{\sigma,A,<w}_{\on{Glued}}&=\colim_{P\in\Par'(G)}\Spr^{\sigma,A,<w}_P\\
\Spr^{\sigma,A,\le w}_{\on{Glued}}&=\colim_{P\in\Par'(G)}\Spr^{\sigma,A,\le w}_P.
\end{align*} 
(Note that the schemes $\Spr^{\sigma,A,w}_P$ do not form a diagram indexed by $P\in\Par'(G)$.)

\medskip

Consider also the quotients 
$$\Spr^{\sigma,A,\leq w}_P/\Spr^{\sigma,A,< w}_P:=
\Spr^{\sigma,A,\leq w}_P \underset{\Spr^{\sigma,A,< w}_P}\sqcup \on{pt}$$
and 
$$\Spr^{\sigma,A,\leq w}_{\on{Glued}}/\Spr^{\sigma,A,< w}_{\on{Glued}}:=
\Spr^{\sigma,A,\leq w}_{\on{Glued}} \underset{\Spr^{\sigma,A,< w}_{\on{Glued}}}\sqcup \on{pt},$$
the latter being the same as
$$\underset{P\in \Par'(G)}{\on{colim}}\, \Spr^{\sigma,A,\leq w}_P/\Spr^{\sigma,A,< w}_P,$$
since the category $\Par'(G)$ is contractible (having an initial object). 

\medskip

In what follows we  also use the notation
$$\Spr^{\sigma,A,<w}_\sJ:=\Spr^{\sigma,A,<w}_P \text{ for } P=P_\sJ,$$
etc. 

\sssec{}

We need to show that the trace map
$$\on{C}_*(\Spr^{\sigma,A}_{\on{Glued}})\to k$$
is an isomorphism.

\medskip

We will prove that for every $w\in \sW'$, the trace map
\begin{equation} \label{e:contr w}
\on{C}_*(\Spr^{\sigma,A,\leq w}_{\on{Glued}})\to k
\end{equation} 
is an isomorphism. (That is, $\Spr^{\sigma,A,\leq w}_{\on{Glued}}$ is a homologically contractible $k$-prestack.) Applying this to $w=w'_0$, 
we obtain the desired result.

\sssec{}  

The proof that \eqref{e:contr w} is an isomorphism uses the following two statements, proved
in Sections \ref{ss:base} and \ref{ss:step}, respectively:

\smallskip

\noindent {\bf Case $w=1$:} the trace map 
$$\on{C}_*(\Spr^{\sigma,A,1}_{\on{Glued}})\to k$$ is an isomorphism;

\smallskip

\noindent {\bf Case $w\neq 1$:} For any $1\neq w\in \sW'$, the trace map 
$$\on{C}_*(\Spr^{\sigma,A,\leq w}_{\on{Glued}}/\Spr^{\sigma,A,< w}_{\on{Glued}})\to k$$ is an isomorphism.

\medskip

Let us show how the combination of these two statements implies that \eqref{e:contr w} is an isomorphism. 
This will be completely formal. 

\medskip

We argue by induction on the poset $\sW'$. The base of the induction is the statement in
Case $w=1$. Let us now perform the induction step, so take $w\neq 1$. 

\medskip

We have a push-out square of prestacks
$$
\CD
\Spr^{\sigma,A,\leq w}_{\on{Glued}} @>>>  \Spr^{\sigma,A,\leq w}_{\on{Glued}}/\Spr^{\sigma,A,< w}_{\on{Glued}} \\
@AAA    @AAA   \\
\Spr^{\sigma,A,< w}_{\on{Glued}} @>>>  \on{pt}.
\endCD
$$
and hence a cofiber square in $\Vect$:
$$
\CD
\on{C}_*(\Spr^{\sigma,A,\leq w}_{\on{Glued}}) @>>>  \on{C}_*(\Spr^{\sigma,A,\leq w}_{\on{Glued}}/\Spr^{\sigma,A,< w}_{\on{Glued}}) \\
@AAA    @AAA   \\
\on{C}_*(\Spr^{\sigma,A,< w}_{\on{Glued}}) @>>>  \on{C}_*(\on{pt})\simeq k.
\endCD
$$
Taking into account the statement in Case $w\neq 1$,  it suffices to show that the trace map
$$\on{C}_*(\Spr^{\sigma,A,< w}_{\on{Glued}})\to k$$
is an isomorphism. This is done below.

\sssec{}

Consider the prestack
$$\underset{w_1<w}{\on{colim}}\, \Spr^{\sigma,A,\leq w_1}_{\on{Glued}}.$$
We have an isomorphism
$$\underset{w_1<w}{\on{colim}}\, \Spr^{\sigma,A,\leq w_1}_{\on{Glued}}\to \Spr^{\sigma,A,< w}_{\on{Glued}},$$
and hence an isomorphism
$$\on{C}_*\left(\underset{w_1<w}{\on{colim}}\, \Spr^{\sigma,A,\leq w_1}_{\on{Glued}}\right)\to 
\on{C}_*(\Spr^{\sigma,A,< w}_{\on{Glued}}).$$

Hence, it remains to show that the trace map
$$\on{C}_*\left(\underset{w_1<w}{\on{colim}}\, \Spr^{\sigma,A,\leq w_1}_{\on{Glued}}\right)\to k$$
is an isomorphism. We have
$$\on{C}_*\left(\underset{w_1<w}{\on{colim}}\, \Spr^{\sigma,A,\leq w_1}_{\on{Glued}}\right)\simeq 
\underset{w_1<w}{\on{colim}}\, \on{C}_*(\Spr^{\sigma,A,\leq w_1}_{\on{Glued}}).$$

Now, by the induction hypothesis, for every $w_1<w$, the trace map 
$$\on{C}_*(\Spr^{\sigma,A,\leq w_1}_{\on{Glued}})\to k$$
is an isomorphism. Hence, the assertion follows from the fact that the index category, i.e.,
$w_1$ with $w_1<w$, is contractible (it contains an initial element $w_1=1$). 

\ssec{Verifying Case $w=1$}  \label{ss:base}

\sssec{}

Let us show that the prestack $\Spr^{\sigma,A,1}_{\on{Glued}}$ itself is isomorphic to $\on{pt}$. By definition,
\[\Spr^{\sigma,A,1}_{\on{Glued}}=\colim_{\sJ\in\Par'(G)}\Spr^{\sigma,A,1}_\sJ.\]

\medskip

Let $M_0$ denote the Levi quotient of $P_0$. Let $\Par(M_0)$ be the poset of all standard parabolics of $M_0$.
We identify $\Par(M_0)$ with the poset of all subsets of $\sJ_0$ (including $\sJ_0$ itself). 

\medskip

The inclusion
\begin{equation} \label{e:par in par}
\Par(M_0)\hookrightarrow \Par'(G)
\end{equation}
admits a right adjoint, given by 
$$\sJ\mapsto \sJ\cap \sJ_0.$$

\medskip

Note now that for any $\sJ\subset \sS$, the map
$$\Fl_{\sJ\cap \sJ_0}^{1}\to \Fl_\sJ^{1}$$
is an isomorphism. Indeed, this is a special case of \lemref{lm:sch}(2).
Therefore, the map
$$\Spr^{\sigma,A,1}_{\sJ\cap \sJ_0}\to \Spr^{\sigma,A,1}_\sJ$$
is an isomorphism as well.

\sssec{}  \label{sss:adj lemma}

We have the following general assertion:

\medskip

Let $I$ be an index category, and $I'\overset{\phi}\hookrightarrow I$ a full subcategory such that the 
inclusion $\phi$ admits a right adjoint, which we denote by $\psi$. 

\medskip

Let $\sF:I\to \bD$ be a functor with values in some $\infty$-category $\bD$. Assume that for every $i\in I$,
the co-unit of the adjunction
$$\phi\circ \psi(i)\to i$$
induces an isomorphism 
$$\sF\circ \phi\circ \psi(i)\to \sF(i).$$

\medskip

\begin{lem}  \label{l:adj cofin}
Under the above circumstances, the canonical map
$$\underset{i'\in I'}{\on{colim}}\, \sF\circ \phi \to \underset{i\in I}{\on{colim}}\, \sF$$
is an isomorphism. \qed
\end{lem}

\sssec{}

Applying \lemref{l:adj cofin} to \eqref{e:par in par} and the functor 
$$\sJ\rightsquigarrow \Spr^{\sigma,A,1}_\sJ,$$
we see that $\Spr^{\sigma,A,1}_{\on{Glued}}$ is isomorphic to the prestack 
\begin{equation} \label{e:smaller colimit}
\underset{\sJ\subset \sJ_0}{\on{colim}}\, \Spr^{\sigma,A,1}_\sJ.
\end{equation}

Now, the index category of subsets of $\sJ_0$ has a final object (namely, $\sJ=\sJ_0$),
and $\Spr^{\sigma,A,1}_{\sJ_0}=\on{pt}$. Hence, the colimit in \eqref{e:smaller colimit}
is isomorphic to $\on{pt}$. 

\ssec{Verifying Case $w\neq 1$}  \label{ss:step}

\sssec{}

We need to show that for $w\neq 1$, the trace map 
\begin{equation} \label{e:colim w}
\underset{P\in \Par'(G)}{\on{colim}}\, \on{C}_*(\Spr^{\sigma,A,\leq w}_\sJ/\Spr^{\sigma,A,< w}_\sJ)\to k
\end{equation}
is an isomorphism. Consider the case of $w\neq w'_0$ first. 

\medskip 

Put 
\[\Par'_w(G):=\{\sJ\in\Par'(G)\,|\,\sJ\subset \sS^0_w\cup \sS^-_w\}\subset \Par'(G).\]
Recall (see \secref{sss:pos neg}) that $\sJ\subset\sS^0_w\cup \sS^-_w$ means that for every simple root $\alpha\in \sJ$, $w(\alpha)$ is either negative, or a root of $\sR_0$. 

\sssec{}

We claim that the inclusion $\Par'_w(G)\hookrightarrow\Par'(G)$
satisfies the conditions of \lemref{l:adj cofin} for the
functor 
$$\sJ\mapsto \on{C}_*(\Spr^{\sigma,A,\leq w}_\sJ/\Spr^{\sigma,A,< w}_\sJ).$$

\medskip

Indeed, note that the inclusion $\Par'_w(G) \hookrightarrow \Par'(G)$ admits a right adjoint given by 
$$\sJ\mapsto \tilde\sJ:=\sJ\setminus \sS^+_w.$$

\medskip

Now, we claim that for $\sJ$ and $\tilde\sJ$ as above, the map
$$\Spr^{\sigma,A,\leq w}_{\tilde\sJ}/\Spr^{\sigma,A,< w}_{\tilde\sJ}\to \Spr^{\sigma,A,\leq w}_\sJ/\Spr^{\sigma,A,< w}_\sJ$$
induces an isomorphism on homology. 

\medskip

This follows by \lemref{lm:sch}(2) from the following general assertion:

\begin{lem}
Let $f:Y_1\to Y_2$ be a proper map between schemes. Let $Y'_i\subset Y_i$ for $i=1,2$ be closed subschemes such that
$f(Y'_1)\subset Y'_2$, and $f$ induces an isomorphism $Y_1\setminus Y'_1\to Y_2\setminus Y'_2$. Then the
induced map
$$\on{C}_*(Y_1\underset{Y'_1}\sqcup \on{pt})\to \on{C}_*(Y_2\underset{Y'_2}\sqcup \on{pt})$$
is an isomorphism.
\end{lem} 

\begin{proof} 

It is enough to show that the map
$$\on{Cone}\left(\on{C}_*(Y'_1)\to \on{C}_*(Y_1)\right) \to \on{Cone}\left(\on{C}_*(Y'_2)\to \on{C}_*(Y_2)\right),$$
defined by $f$, is an isomorphism. 

\medskip

Let $\iota_i$ (resp. $j_i$) denote the closed embedding $Y_i'\hookrightarrow Y_i$ 
(resp. the open embedding $(Y_i\setminus Y'_i)\hookrightarrow Y_i$). From the excision exact triangle
\[(\iota_i)_{\dr,*}(\omega_{Y'_i})\to\omega_{Y_i}\to(j_i)_{\dr,*}(\omega_{Y_i\setminus Y_i'})\]
we obtain an isomorphism
$$\on{Cone}\left(\on{C}_*(Y'_i)\to \on{C}_*(Y_i)\right)\simeq(p_{Y_i})_{\dr,!}((j_i)_{\dr,*}(\omega_{Y_i\setminus Y'_i})),$$
where $p_{Y_i}:Y_i\to \on{pt}$ is the projection to the point.

\medskip

Now, the fact that $f$ is proper and the assumption of the lemma imply that 
$$f_{\dr,!}((j_1)_{\dr,*}(\omega_{Y_1\setminus Y'_1}))\simeq (j_2)_{\dr,*}(\omega_{Y_2\setminus Y'_2}),$$
implying the desired isomorphism. 

\end{proof} 

\begin{rem}
The above argument involves the excision exact triangle. For this reason, it does not imply 
that the prestack $\Spr^{\sigma,A,\leq w}_\sJ/\Spr^{\sigma,A,< w}_\sJ$ itself is isomorphic to $\on{pt}$
(and we do not know whether this is true). 
\end{rem}

\sssec{}

Thus, by \lemref{l:adj cofin}, the colimit in \eqref{e:colim w} is isomorphic to the colimit
$$\underset{P\in \Par'_w(G)}{\on{colim}}\, \on{C}_*(\Spr^{\sigma,A,\leq w}_\sJ/\Spr^{\sigma,A,< w}_\sJ),$$
and it suffices to show that the trace map from the latter to $k$ is an isomorphism.
Let us show that the prestack 
\begin{equation} \label{e:colim w'}
\underset{P\in \Par'_w(G)}{\on{colim}}\, \Spr^{\sigma,A,\leq w}_\sJ/\Spr^{\sigma,A,< w}_\sJ
\end{equation}
itself is isomorphic to $\on{pt}$. 

\medskip

By the assumption that $w\neq w'_0$ and \corref{co:partition of S}, the poset $\Par'_w(G)$
contains a maximal element, namely, $\sJ=\sS^-_w\cup \sS^0_w$. Hence, the colimit 
\eqref{e:colim w'} is isomorphic to 
$$\Spr^{\sigma,A,\leq w}_{\sS^-_w\cup \sS^0_w}/\Spr^{\sigma,A,< w}_{\sS^-_w\cup \sS^0_w}.$$

Now, by the assumption that $w\neq 1$ and \lemref{lm:sch}(1), we have 
$$\Spr^{\sigma,A,\leq w}_{\sS^-_w\cup \sS^0_w}=\Spr^{\sigma,A,< w}_{\sS^-_w\cup \sS^0_w},$$and so
$$\Spr^{\sigma,A,\leq w}_{\sS^-_w\cup \sS^0_w}/\Spr^{\sigma,A,< w}_{\sS^-_w\cup \sS^0_w}\simeq \on{pt}.$$

\sssec{}\label{sss:largestisempty}

Finally, we consider the case of $w=w'_0$. We claim that in this case the prestack 
$$\underset{P\in \Par'(G)}{\on{colim}}\, \Spr^{\sigma,A,\leq w'_0}_\sJ/\Spr^{\sigma,A,< w'_0}_\sJ$$
is isomorphic to $\on{pt}$. In fact, we claim that for every $\sJ$, we have 
$$\Spr^{\sigma,A,w'_0}_\sJ=\emptyset,$$
and so 
$$\Spr^{\sigma,A,\leq w'_0}_\sJ/\Spr^{\sigma,A,< w'_0}_\sJ\simeq \on{pt}.$$

\medskip 

Indeed, the fact that $\Spr^{\sigma,A,w'_0}_\sJ$ is empty follows from \lemref{l:unip radical}
and the fact that $A$ is a horizontal section of $\fu(P_0)_\sigma$, while $A\neq 0$
by assumption. 

\section{A proof via the Grothendieck-Springer correspondence}  \label{s:Springer}

In this section we  give an alternative proof of \thmref{t:contr2} in the special case of the trivial
local system $\sigma$.

\ssec{Making the nilpotent vary}

\sssec{}

As was mentioned above, in this section the local system is trivial. Hence, we can think of $A$ as a nilpotent 
element of the Lie algebra $\fg$, and $\Spr^{\sigma,A}_P$ is thus the usual parabolic Springer fiber
$$\Spr^A_P:=\{P'\in\Fl_P\,|\, A\in \fp'\}.$$

\sssec{}

For an element $P\in \Par(G)$, let 
\[\wt\fg_P:=\{(x,P')\in\fg\times \Fl_P\,|\, x\in \fp'\}\subset\fg\times \Fl_P\]
be the \emph{parabolic Grothendieck-Springer} variety.
Denote by $\pi_P$ the tautological projection $\wt\fg_P\to \fg$, and put
$$\CS_P:=(\pi_P)_{\dr,!}(\omega_{\Fl_P})\in \Dmod(\fg).$$

\medskip

The assignment
$$P\rightsquigarrow \CS_P$$
is a functor $\Par(G)\to \Dmod(\fg)$.
Consider the colimit
$$\CS_{\on{Glued}}:=\underset{P\in \Par'(G)}{\on{colim}}\, \CS_P\in \Dmod(\fg).$$

\sssec{}

Let $\on{Nilp}_\fg\overset{i}\hookrightarrow \fg$ be the subvariety of nilpotent elements.
Consider the object
$$i^!(\CS_{\on{Glued}})\in \Dmod(\on{Nilp}_\fg).$$

\medskip

By construction, the assertion of \thmref{t:contr2} is equivalent to the following:

\begin{prop}  \label{p:springer glued}
The trace map 
$$i^!(\CS_{\on{Glued}})\to \omega_{\on{Nilp}_\fg}$$
is an isomorphism away from $0\in \on{Nilp}_\fg$.
\end{prop}

\ssec{Interpretation via the Springer theory}

In this subsection we  recall some basic facts about the Springer theory.

\sssec{}

Put
$$\CS:=\CS_B.$$

It is well known that $\CS[-\dim(\fg)]$ lies in the heart of the t-structure (note that the usual t-structure for
$D$-modules corresponds to the perverse t-structure under the Riemann-Hilbert correspondence), and that it
carries a canonically defined action of $\sW$. 

\medskip

Here are some well-known facts regarding $\CS$:

\begin{lem}   \label{l:springer} \hfill

\smallskip

\noindent{\em(a)} The trace map $\CS\to \omega_{\fg}$ induces an isomorphism $\on{coinv}(\sW,\CS)\to \omega_{\fg}$.
Here $\on{coinv}(\sW,\CS)$ is the $D$-module of coinvariants of the action of $\sW$ on $\CS$.

\smallskip

\noindent{\em(b)} 
Let $\on{anti-inv}(\sW,\CS)$ be the sign isotopic component in $\CS$. 
Then the !-restriction of $\on{anti-inv}(\sW,\CS)$ to $\on{Nilp}_\fg$ vanishes outside
of $0\in \on{Nilp}_\fg$.

\smallskip

\noindent{\em(c)} For a parabolic $P=P_\sJ$, we have $\CS_P\simeq \on{coinv}(\sW_\sJ,\CS)$, and for $\sJ_1\subset \sJ_2$
the natural map $\CS_{P_1}\to \CS_{P_2}$ is induced by the inclusion $\sW_{\sJ_1}\subset \sW_{\sJ_2}$. 

\end{lem}

\sssec{}

In view of the above lemma, \propref{p:springer glued} follows from the next more precise result:

\begin{prop} \label{p:springer glued bis}
There exists a canonical isomorphism in $\Dmod(\on{Nilp}_\fg)$:
$$\CS_{\on{Glued}}\simeq \on{coinv}(\sW,\CS) \oplus \on{anti-inv}(\sW,\CS)[\rk(\fg)-1].$$
\end{prop}

\ssec{Proof of \propref{p:springer glued bis}}

\sssec{}

In view of \lemref{l:springer}(c), the object $\CS_{\on{Glued}}$ has the form
$$\sM \underset{k[\sW]}\otimes \CS,$$
for $\sM\in \Rep(\sW)$ equal to
$$\underset{\sJ\in \Par'(G)}{\on{colim}}\, k[\sW]\underset{k[\sW_\sJ]}\otimes k.$$

\medskip

Thus, it remains to show that
\begin{equation} \label{e:shape of M}
\sM\simeq k\oplus \sign[\rk(\fg)-1],
\end{equation}
viewed as representations of $\sW$. 

\sssec{}

Instead of proving the isomorphism \eqref{e:shape of M} directly, let us provide a more elegant geometric argument. 

\medskip

Consider the diagram of finite sets equipped with an action of $\sW$:
$$\sJ\mapsto \sW/\sW_\sJ.$$

Consider the homotopy type
$$\sW_{\on{Glued}}:=\underset{\sJ\in \Par'(G)}{\on{colim}}\, \sW/\sW_\sJ.$$

We have:
$$\sM=\on{C}_*(\sW_{\on{Glued}}).$$

\sssec{}

We claim that the geometric realization of $\sW_{\on{Glued}}$ is $\sW$-equivariantly homotopy 
equivalent to a $(\rk(\fg)-1)$-dimensional sphere in the Euclidean 
space $$\ft_\BR:=\Lambda \underset{\BZ}\otimes \BR.$$
Indeed, fix a generic point $\gamma\in\ft_\BR$, and let $\sB_\gamma$ be the convex hull of the orbit $\sW\gamma$.
For each $j=0,\dots,\rk(\fg)$, the $j$-faces of the polytope $\sB_\gamma$ are indexed by the union \[\coprod_{|\sJ|=j}\sW/\sW_\sJ.\] From this, we obtain a $\sW$-equivariant homotopy equivalence between
$\sB_\gamma$ and the geometric realization of
\[\underset{\sJ\in \Par(G)}{\on{colim}}\, \sW/\sW_\sJ,\]
and also between the boundary $\partial(\sB_\gamma)$ (which is homeomorphic to a sphere)
and the geometric realization of $\sW_{\on{Glued}}$.

\medskip 

The sign representation $\sW\to\{\pm 1\}$ identifies with the action of $\sW$
on the torsor of orientations of $\ft_\BR$, and hence also on the torsor of orientations of $\partial(\sB_\gamma)$. 

\medskip

This implies the desired formula for $\on{C}_*(\sW_{\on{Glued}})$.

\end{document}